\documentclass[12pt]{article}
\usepackage[utf8]{inputenc}
\usepackage{amsmath}
\usepackage{mathtools}
\usepackage{amsthm}
\usepackage{bbm}
\usepackage{amsfonts}
\usepackage{amssymb}
\usepackage[left=2.5cm,right=2.5cm,top=2cm,bottom=2cm]{geometry}

\usepackage{tabto}
\TabPositions{2 cm, 4 cm, 6 cm, 8 cm}

\usepackage{tikz-cd}

\usepackage{tikz}
\usetikzlibrary{arrows, automata, positioning}

\usepackage{tocloft}
\usepackage{appendix}

\usepackage{enumerate}

\usepackage{hyperref}
\hypersetup{
  colorlinks   = true, 
  urlcolor     = blue, 
  linkcolor    = blue, 
  citecolor   = red 
}
 
\newtheorem{theorem}{Theorem}[section]
\newtheorem{corollary}[theorem]{Corollary}
\newtheorem{lemma}[theorem]{Lemma}
\newtheorem{claim}[theorem]{Claim}
\newtheorem{proposition}[theorem]{Proposition}

\theoremstyle{definition}
\newtheorem{definition}[theorem]{Definition}
\newtheorem{example}[theorem]{Example}
\newtheorem*{definition*}{Definition}

\newtheorem*{lemma*}{Lemma}
\newtheorem*{proposition*}{Proposition}
\newtheorem*{theorem*}{Theorem}
\newtheorem*{corollary*}{Corollary}

\theoremstyle{definition}

\newtheorem{question}[theorem]{Question}
\newtheorem{problem}[theorem]{Problem}

\newtheorem{remark}[theorem]{Remark}

\newtheorem*{remark*}{Remark}
\newtheorem*{notation}{Notation}
\newtheorem*{acks}{Acknowledgements}
\newtheorem*{outline}{Outline}

\newcommand{\Ab}{\operatorname{Ab}}

\newcommand{\cd}{\operatorname{cd}}

\newcommand{\U}{\operatorname{U}}

\newcommand{\f}{\varphi}
\newcommand{\hf}{\hat{\varphi}}
\newcommand{\of}{\overline{\varphi}}
\newcommand{\ff}{\psi}
\newcommand{\hff}{\hat{\psi}}

\newcommand{\K}{\mathbb{K}}
\newcommand{\oo}{\mathfrak{o}}
\newcommand{\pp}{\mathfrak{p}}
\newcommand{\kk}{\mathfrak{k}}
\newcommand{\ok}{\mathfrak{o}/ \mathfrak{p}^k}

\newcommand{\uni}{\overline{\omega}}

\newcommand{\Qp}{\mathbb{Q}_p}
\newcommand{\Zp}{\mathbb{Z}_p}

\newcommand{\Fp}{\mathbb{F}_p}
\newcommand{\Fq}{\mathbb{F}_q}
\newcommand{\Zpk}{\mathbb{Z}/p^k \mathbb{Z}}

\newcommand{\GGL}{\operatorname{GL}}
\newcommand{\MM}{\operatorname{M}}

\newcommand{\BS}{\operatorname{BS}}

\newcommand{\Sym}{\operatorname{Sym}}
\newcommand{\Alt}{\operatorname{Alt}}

\newcommand{\Gal}{\operatorname{Gal}}

\newcommand{\Aut}{\operatorname{Aut}}

\newcommand{\pr}{\operatorname{pr}}
\newcommand{\id}{\operatorname{id}}
\newcommand{\diag}{\operatorname{diag}}
\newcommand{\res}{\operatorname{res}}

\newcommand{\normal}[1]{\left<\! \left< #1\right> \!\right>}

\newcommand{\HH}{\operatorname{H}}
\newcommand{\EH}{\operatorname{EH}}
\newcommand{\CC}{\operatorname{C}}
\newcommand{\ZZ}{\operatorname{Z}}
\newcommand{\BB}{\operatorname{B}}

\newcommand{\en}{\varepsilon_n}
\newcommand{\ez}{\varepsilon_0}
\newcommand{\ek}{\varepsilon_k}
\newcommand{\ee}{\varepsilon}
\newcommand{\eee}{\overline{\varepsilon}}
\newcommand{\dd}{\delta}

\newcommand{\G}{\mathcal{G}}
\newcommand{\C}{\mathcal{C}}

\newcommand{\defe}{\operatorname{def}}
\newcommand{\dist}{\operatorname{dist}}
\newcommand{\Hdist}{\operatorname{Hdist}}

\newcommand{\MQ}{\operatorname{MQ}}

\begin{document}

\title{Ultrametric analogues of Ulam stability of groups}
\author{Francesco Fournier-Facio}
\date{\today}
\maketitle

\vspace{0.5cm}

\begin{abstract}
We study stability of metric approximations of countable groups with respect to groups endowed with ultrametrics, the main case study being a $p$-adic analogue of Ulam stability, where we take $\GGL_n(\Zp)$ as approximating groups instead of $\U(n)$. For finitely presented groups, the ultrametric nature implies equivalence of the pointwise and uniform stability problems, and the profinite one implies that the corresponding approximation property is equivalent to residual finiteness. Moreover, a group is uniformly stable if and only if its largest residually finite quotient is. We provide several examples of uniformly stable groups, including finite groups, virtually free groups, some groups acting on rooted trees, and certain lamplighter and (Generalised) Baumslag--Solitar groups. We construct a finitely generated group that is not uniformly stable. Finally, we prove and apply a (bounded) cohomological criterion for stability of a finitely presented group.
\end{abstract}

\setlength{\cftbeforetoctitleskip}{1.8cm}

\tableofcontents

\medskip


\medskip


%
%
%
%


\pagebreak

\section{Introduction}

Let $\Gamma$ be a countable discrete group, and let $\G$ be a family of groups $G$ equipped with bi-invariant metrics $d_G$. The question of \emph{stability of $\Gamma$ with respect to $\G$} asks whether a map $\f \colon \Gamma \to G \in \G$ that is a homomorphism up to a small error, is close to an actual homomorphism $\ff \colon \Gamma \to G$. This can be made rigorous in (at least) two different ways, that depend on whether one wants the errors and closeness to be pointwise or uniform. The notion of ``error'' is defined as follows:

\begin{definition}
\label{def:def}

Let $\f \colon \Gamma \to G \in \G$ be a map. We define the \emph{defect of $\f$ at $(g, h) \in \Gamma^2$} to be
\[\defe_{g, h}(\f) \coloneqq d_G(\f(gh), \f(g) \f(h)).\]
The \emph{defect of $\f$} is
\[\defe(\f) \coloneqq \sup\limits_{g, h \in \Gamma} \defe_{g, h}(\f).\]
A sequence of maps $(\f_n \colon \Gamma \to G_n \in \G)_{n \geq 1}$ is called a \emph{pointwise asymptotic homomorphism} if $\defe_{g, h}(\f_n) \xrightarrow{n \to \infty} 0$ for all $(g, h) \in \Gamma^2$; and a \emph{uniform asymptotic homomorphism} if $\defe(\f_n) \xrightarrow{n \to \infty} 0$.
\end{definition}

Other commonly used terms are \emph{almost-representation} \cite{defstab} for the pointwise notion and \emph{quasi-representation} \cite{Shtern} or \emph{$\ee$-representation} \cite{amenst} for the uniform notion. The notion of ``close'' is defined as follows:

\begin{definition}
\label{def:dist}

Given two maps $\f, \ff \colon \Gamma \to G \in \G$, we define their \emph{distance at $g \in \Gamma$} to be
\[\dist_g(\f, \ff) \coloneqq d_G(\f(g), \ff(g));\]
and their \emph{distance} to be
\[\dist(\f, \ff) \coloneqq \sup\limits_{g \in \Gamma} \dist_g(\f, \ff).\]
Two sequences of maps $(\f_n, \ff_n \colon \Gamma \to G_n \in \G)_{n \geq 1}$ are said to be \emph{pointwise asymptotically close} if $\dist_g(\f_n, \ff_n) \xrightarrow{n \to \infty} 0$ for all $g \in \Gamma$; and \emph{uniformly asymptotically close} if $\dist(\f_n, \ff_n) \xrightarrow{n \to \infty} 0$.
\end{definition}

This leads to two notions of stability, that we attribute to Ulam, after \cite{Ulam}:

\begin{definition}[Ulam]
\label{def:stab}

The group $\Gamma$ is \emph{pointwise $\G$-stable} if every pointwise asymptotic homomorphism is pointwise asymptotically close to a sequence of homomorphisms. It is \emph{uniformly $\G$-stable} if every uniform asymptotic homomorphism is uniformly asymptotically close to a sequence of homomorphisms.
\end{definition}

Early mentions of these problems can be found in works of von Neumann \cite{vN} and Turing \cite{Turing}. The problem of pointwise stability of $\mathbb{Z}^2$ with respect to certain families of matrices, for instance self-adjoint and with the operator norm, received a lot of attention during the second half of the twentieth century (see \cite{Lin} and the references therein). In \cite[Chapter 6]{Ulam}, Ulam discusses more generally the question of stability of certain functional equations: because of this, the term \emph{Ulam stability} was introduced in \cite{BOT} to refer to uniform stability with respect to unitary groups equipped with the distance induced by the operator norm (see below). Some of the most common families $\G$ of approximating groups are the unitary groups $\U(n)$ or the symmetric groups $S_n$.

$\U(n)$ is typically considered with a metric induced by a norm defined on $\MM_n(\mathbb{C})$. The first example is that of the \emph{operator norm}, where pointwise stability has striking topological and $K$-theoretic interpretations \cite{bundles1, bundles2}, all amenable groups are known to be uniformly stable \cite{amenst}, and groups with non-trivial quasimorphisms are known not to be uniformly stable \cite{Rolli, BOT}. Recently, the list of so-called Ulam stable groups, as been extended to include most high-rank lattices \cite{Bharat} and Thompson groups \cite{Bharatandi}. Another example is that of the \emph{Frobenius norm} $\| A \|_{Frob} \coloneqq \sqrt{|A_{ij}|^2}$, that is, the norm induced by the embedding of $\U(n)$ into $\mathbb{C}^{n \times n}$: this has the advantage of allowing a cohomological criterion for pointwise stability \cite{GLT}. The third main example is given by the \emph{Hilbert--Schmidt norm} $\| A \|_{HS} \coloneqq \frac{1}{\sqrt{n}} \| A \|_{Frob}$, which is the normalisation of the Frobenius norm: this has the advantage of allowing a $C^\ast$-algebraic approach, which yields a characterisation of pointwise stability of amenable groups in terms of characters \cite{HSstab1, HSstab2}, as well as a simple algebraic characterisation of uniform stability of finitely generated residually finite groups \cite{uHS}.

On the other hand the groups $S_n$ are studied with the normalised Hamming distance $d_H(\sigma, \tau) \coloneqq 1 - \frac{1}{n}|Fix(\sigma^{-1} \tau)|$. Pointwise stability of equations in permutation was initially considered by Glebsky and Rivera \cite{GR}, then by Arzhantseva and P\u{a}unescu \cite{a:comm} who proved that this can be translated to a group property, as in Definition \ref{def:stab}. Since then this pointwise stability problem has been under intense investigation, as well as some variants thereof: flexible \cite{T, surf}, quantitative \cite{quant}, uniform, probabilistic \cite{BChap}, local \cite{henry:local, us:local}, and connections to computer science \cite{test, test2}. A standout result is a characterisation of pointwise stability of amenable groups in terms of invariant random subgroups \cite{IRS}.

\medskip

The pointwise and uniform problems typically exhibit a very different behaviour. For example, consider the family $\mathcal{G} = \{ (\U(n), \| \cdot \|_{op}) : n \geq 1 \}$ of unitary groups equipped with the metric induced by the operator norm, and the two stability problems with respect to $\mathcal{G}$. On the one hand, $\mathbb{Z}^2$ is not pointwise stable \cite{Voie}, but it is uniformly stable, as are all amenable groups \cite{amenst}. On the other hand, a non-abelian free group of finite rank is not uniformly stable \cite{Rolli}, but it is pointwise stable. In fact, free groups are pointwise $\G$-stable for every family $\G$: if $(\f_n)_{n \geq 1}$ is a pointwise asymptotic homomorphism, then letting $\ff_n$ be the unique homomorphism that coincides with $\f_n$ on a given free basis, $(\ff_n)_{n \geq 1}$ is pointwise asymptotically close to $(\f_n)_{n \geq 1}$. 

\medskip

In this paper, we study \emph{ultrametric} versions of these problems, that is, we look at approximating families $\G$ where the metrics $d_G$ are ultrametric. The main example throughout the paper will be a $p$-adic analogue of Ulam stability: we choose $\GGL_n(\Zp)$ -- which is maximal compact in $\GGL_n(\Qp)$ -- as an analogue of $\U(n)$ -- which is maximal compact in $\GGL_n(\mathbb{C})$. The natural norm on $\Qp$-vector spaces, that is, the one that preserves the non-Archimedean nature, is the $\ell^\infty$-norm relative to the $p$-adic norm $|\cdot|_p$ on $\Qp$. Keeping this and the case of $\U(n)$ in mind, there are three norms that one could choose to induce a distance on $\GGL_n(\Zp)$: the operator norm with respect to the $\ell^\infty$-norm on $\mathbb{Q}_p^n$, the norm induced by the embedding of $\GGL_n(\Zp)$ into $\mathbb{Q}_p^{n \times n}$ with the $\ell^\infty$-norm, and a normalised version of the latter. It turns out that all of these coincide (Lemma \ref{lem:GL}), and so
\[\| \cdot \| \colon \MM_n(\Qp) \to \mathbb{R}_{\geq 0} : A = (A_{ij})_{1 \leq i, j \leq n} \mapsto \max\limits_{1 \leq i, j \leq n} |A_{ij}|_p\]
is, in some sense, the canonical norm to consider. It induces a bi-invariant ultrametric $d$, and moreover it reflects the profinite structure of $\GGL_n(\Zp)$: in fact $\|A - I \| \leq p^{-k}$ if and only if $A \equiv I \mod p^k$. We denote by $\GGL(\Zp)$ this family of metric groups, and will focus on this example of approximating family for the statements of the results, mentioning which properties we are using. Each result can be generalised to families of groups satisfying such properties, and the statements will be given in full generality in the paper.

To the author's knowledge, the only previous mention of $p$-adic versions of stability is in Kazhdan's work \cite[Proposition 1]{amenst}, where it is shown that for every $n \geq 1$ the standard representative map $\f \colon \mathbb{Z}/p^n\mathbb{Z} \to \Zp$ satisfies $\defe(\f) = p^{-n}$ and $\dist(\f, \ff) = 1$ for every homomorphism $\ff$. This result does not however show that these groups are unstable with respect to the family $\{ \Zp \}$: this is indeed not the case as we will see in Proposition \ref{prop:unfin}. However, the fact that these bad estimates for stability arise when looking at maps from a finite $p$-group to a pro-$p$ group is not a coincidence, and will be exploited to construct unstable groups in Section \ref{s:unst}. 

\medskip

Using only the ultrametric inequality, we prove a relation between the pointwise and uniform stability problems, that as we have seen above does not hold in the Archimedean setting:

\begin{theorem}[Theorem \ref{thm:pw_un}]
\label{intro:thm:pw_un}

Let $\Gamma$ be finitely generated and pointwise $\GGL(\Zp)$-stable. Then $\Gamma$ is uniformly $\GGL(\Zp)$-stable. If moreover $\Gamma$ is finitely presented, then the converse holds.
\end{theorem}

The techniques developed for the proof of Theorem \ref{intro:thm:pw_un} also apply further: in Proposition \ref{prop:unfin} we show that if $\G$ is a finite family of compact ultrametric groups, then every finitely generated group is uniformly $\G$-stable. This also does not hold in the Archimedean setting: for instance if $\G = \{ ( \U(1), \| \cdot \|_{op} ) \}$, then a non-abelian free group is not uniformly $\G$-stable \cite{Rolli}. 

\medskip

Using the fact that the metric reflects the profinite structure, we are able to reduce the uniform stability problem to residually finite groups:

\begin{theorem}[Theorem \ref{thm:rf}]
\label{intro:thm:rf}

Let $\Gamma$ be a group, and $\Omega$ its largest residually finite quotient. Then $\Gamma$ is uniformly $\GGL(\Zp)$-stable if and only if $\Omega$ is. If $\Gamma$ is pointwise $\GGL(\Zp)$-stable, then so is $\Omega$.
\end{theorem}

In particular, a group without finite quotients is uniformly $\GGL(\Zp)$-stable.
The techniques developed for the proof of Theorem \ref{intro:thm:rf} also apply further: in Proposition \ref{prop:aut_stab} we provide the complete solution to two other stability problems. The first one is with respect to the family $T(R)$ of invertible upper-triangular matrices over a commutative unital ring $R$; the second one is with respect to the family $\Aut(X^*_\bullet)$ of automorphism groups of regular rooted trees of increasing degrees.
Under natural ultrametrics that reflect a projective structure, we prove that all groups are uniformly stable with respect to these two families.

\medskip

Related to the stability problem is the corresponding approximation problem. We attribute the following definition to Gromov, after \cite{Gromov}:

\begin{definition}[Gromov]
\label{def:approx}

A sequence $(\f_n \colon \Gamma \to G_n \in \G)_{n \geq 1}$ is \emph{asymptotically injective} if
\[\liminf_{n \to \infty} d_{G_n}(\f_n(g), 1) > 0\]
for every $1 \neq g \in \Gamma$. A pointwise asymptotic homomorphism that is also asymptotically injective is called a \emph{$\G$-approximation}: if one exists, $\Gamma$ is said to be \emph{$\G$-approximable}.
\end{definition}

This leads to the important notions of \emph{sofic groups}, when $\G = \{ (S_n, d_H) : n \geq 1\}$, introduced by Gromov \cite{Gromov} and named by Weiss \cite{Weiss}; and \emph{hyperlinear groups}, when $\G = \{ (\U(n), \| \cdot \|_{HS}) : n \geq 1\}$, introduced by Radulescu \cite{Radulescu} in the context of the Connes Embedding Conjecture \cite{Connes}. These classes of groups are very large, so large that no non-example is known to date (although there has been some important recent progress \cite{MIP, aldous1, aldous2, hyperlinearIRS}). In contrast, the profinite nature of $\GGL_n(\Zp)$ allows to characterise approximation in terms of other well-studied properties:

\begin{theorem}[Propositions \ref{prop:approx_1} and \ref{prop:approx_2}]
\label{intro:thm:approx}

A countable group is $\GGL(\Zp)$-approximable if and only if it is LEF (locally embeddable in the class of finite groups). In particular, a finitely presented group is $\GGL(\Zp)$-approximable if and only if it is residually finite.
\end{theorem}

The class of LEF groups was formally introduced by Gordon and Vershik in \cite{LEF}, although it is already present in Malcev's work \cite{Malcev}: we refer the reader to Subsection \ref{ss:rf} for the precise definitions. We are also able to characterise \emph{uniform} approximability, where the approximation is required to be a uniform asymptotic homomorphism, as being equivalent to residual finiteness, for arbitrary countable groups (Propositions \ref{prop:approx_1} and \ref{prop:approx_2}). Moreover, the techniques developed for the proof of Theorem \ref{intro:thm:approx} allow to characterise approximability for a few other families (Corollaries \ref{cor:approxTR} and \ref{cor:galfin_approx_1}). 

\medskip

Going back to stability, the strongest results that are proved in this paper concern ultrametric families that are \emph{virtually pro-$p$} for some prime $p$ (see Definition \ref{def:vprop}), which includes the case of $\GGL(\Zp)$. For such families, we can prove stability results for fundamental groups of graphs of groups with some restrictions on the orders of finite quotients. These include the following classes of examples:

\begin{theorem}
\label{intro:thm:pifree}

The following groups are uniformly $\GGL(\Zp)$-stable:
\begin{enumerate}
\item \emph{(Proposition \ref{prop:pifree})} Groups without finite virtual $p$-quotients (i.e., groups all of whose finite quotients have order coprime to $p$).
\item \emph{(Corollary \ref{cor:vfree_p})} Finitely generated virtually free groups without elements of order $p$.
\item \emph{(Corollary \ref{cor:wr})} Wreath products $\Gamma \wr \mathbb{Z}$, when $\Gamma$ does not surject onto $\mathbb{Z}/p\mathbb{Z}$.
\item \emph{(Corollary \ref{cor:BS_p})} Baumslag--Solitar groups $\BS(m, n)$, when $p$ divides exactly one of $m$ and $n$.
\item \emph{(Corollary \ref{cor:rfBS_p})} $\mathbb{Z} \left[ \frac{1}{mn} \right] \rtimes_{\frac{m}{n}} \mathbb{Z}$, when $p$ divides exactly one of $m$ and $n$, and moreover $(m, n) = 1$ and $|m|, |n| \neq 1$.
\end{enumerate}
\end{theorem}

Groups as in $1.$ include all periodic groups without elements of order $p$ (Example \ref{ex:period}), as well as groups of automorphisms of regular rooted trees of degree smaller than $p$ with the congruence subgroup property (Example \ref{ex:CSP}). Item $4.$, with the appropriate $p$, applies to every non-Hopfian Baumslag--Solitar group and every residually finite Baumslag--Solitar group, with the exception of $\mathbb{Z}^2$ and the Klein bottle group (see \cite{BS:Hopf} or \cite{Alex} for a more detailed proof). The group from Item $5.$ is the largest residually finite quotient of $\BS(m, n)$ \cite{moldavanskii}, so it also provides an example of an infinitely presented pointwise stable group, by Item $4.$ and Theorem \ref{intro:thm:rf}.

The proof of Theorem \ref{intro:thm:pifree} relies on the Schur--Zassenhaus Theorem (Theorem \ref{thm:SZ}), which states that an extension of finite groups with coprime orders splits, and that any two splittings are conjugate. The first part is used to prove Item $1.$, the second one is used to treat graphs of groups. All these results are quantitatively precise, in particular, the quantitative estimates involved with stability are optimal. Moreover, the statement about graphs of groups falls both in the framework of \emph{constraint stability} \cite{a:const} and of \emph{stability of epimorphisms} \cite{stepi}, providing new examples of these recently introduced notions (Remark \ref{rem:constraint:epi}). 

\medskip

These results only use that the family $\GGL(\Zp)$ is virtually pro-$p$, so in particular they also apply to the characteristic $p$ setting, where $\Zp$ is replaced by $\Fq[[X]]$, for $q$ a power of $p$. But for the case of $\Zp$ we can make these criteria more flexible: a cohomological argument implies an analogue of the Schur--Zassenhaus Theorem suitable to this setting (Lemma \ref{lem:cohopk}), that yields the following strengthening of Theorem \ref{intro:thm:pifree}:

\begin{theorem}
\label{intro:thm:vpfree}

The following groups are uniformly $\GGL(\Zp)$-stable:
\begin{enumerate}
\item \emph{(Proposition \ref{prop:vpfree})} Groups with a bound on the order of their finite virtual $p$-quotients.
\item \emph{(Corollary \ref{cor:vfree})} Finitely generated virtually free groups.
\item \emph{(Corollary \ref{cor:BS})} Baumslag--Solitar groups $\BS(m, n)$, when $\nu_p(m) \neq \nu_p(n)$.
\end{enumerate}
\end{theorem}

Item $1.$ includes all finite groups, and more generally all groups with finite exponent, by Zelmanov's solution of the restricted Burnside Problem \cite{RBP_odd, RBP_2}. In Item $3.$ above, $\nu_p$ denotes the $p$-adic valuation: with the appropriate $p$ it applies to every non-residually finite Baumslag--Solitar group \cite{BS:Hopf, Alex}, extending the remark following Theorem \ref{intro:thm:pifree}. Also these results are quantitatively precise, and the estimates involved are linear. Both stability results on Baumslag--Solitar groups are part of more general statements on \emph{Generalised Baumslag--Solitar groups} (Corollaries \ref{cor:GBS_p} and \ref{cor:GBS}), which give combinatorial and arithmetic conditions on the underlying weighted graphs that imply stability.

Let us also point out that the more general results on stability of graphs of groups imply that the fundamental group of a graph of groups with uniformly $\GGL(\Zp)$-stable vertex groups and finite edge groups is uniformly $\GGL(\Zp)$-stable (Corollary \ref{cor:vpfree_gog}). Similar results for pointwise permutation and Hilbert--Schmidt stability have recently been under investigation \cite{arzhantseva:amalgamated, maria:amalgamated}.

The methods used for the proof of Theorem \ref{intro:thm:vpfree} rely strongly on the fact that $\mathbb{Q}_p$ has characteristic $0$, and in particular they cannot be used to determine whether finite $p$-groups are stable with respect to $\GGL(\mathbb{F}_q[[X]])$, where $q$ is a power of $p$. Still, we are able to show that stability does hold for $\mathbb{Z}/2\mathbb{Z}$ in characteristic $2$, with a quadratic estimate (Proposition \ref{prop:z2z}). However our proof relies on the solution of the similarity problem for representations of $\mathbb{Z}/2\mathbb{Z}$ over finite commutative local rings \cite{Braw2}. The analogous problem for all other $p$-groups in characteristic a power of $p$ is computationally wild \cite{wild}. 

\medskip

The previous results suggest that uniform $\GGL(\oo)$-stability behaves well when passing to finite-index supergroups. We show that this holds under some additional hypotheses, via a method which is conceptually very different from the rest of the paper, and is reserved to Section \ref{s:BC}. In \cite{GLT}, the authors consider the family $\G \coloneqq \{ (\U(n), \| \cdot \|_{Frob}) : n \geq 1 \}$ and prove a cohomological criterion that ensures pointwise $\G$-stability of a finitely presented group. The key feature of the Frobenius norm that is exploited is its \emph{submultiplicativity}, and indeed the same approach works for other submultiplicative norms on $\U(n)$ \cite{pnorm}. The $\ell^\infty$-norm on $\GGL_n(\Zp)$ also has this property, which allows to carry over the arguments. But the non-Archimedean setting has a peculiarity of its own: the cocycles appearing in the process are moreover bounded, so it is natural to state the results in terms of \emph{bounded cohomology} instead. This is a rich theory over the reals (see e.g. \cite{monod, Frig}), whose study over non-Archimedean fields was recently initiated by the author \cite{mio}. 

This way, we are able to prove a bounded-cohomological criterion for stability, as well as a cohomological one (Corollary \ref{cor:BC:abs}). The reader may suspect that the passage from cohomology in the Archimedean setting to bounded cohomology in the non-Archimedean setting is accounted for by the close relation between pointwise and uniform stability; and that a bounded cohomological criterion for uniform stability should also hold in the Archimedean setting. However, in that case the situation is more delicate and requires the introduction of a different cohomology theory, called \emph{asymptotic cohomology} \cite{Bharat}. The advantage of our setting is that, while cocycles are bounded thanks to the uniform nature of the problem, we can still use the same ultraproduct techniques that apply to the pointwise setting of \cite{GLT}.

Our criterion however does not seem to apply further than virtually free groups, whose stability was already established in Theorems \ref{intro:thm:pifree} and \ref{intro:thm:vpfree}. On the other hand, phrasing the correspondence between asymptotic homomorphisms and bounded cohomology classes more precisely by means of the recently introduced \emph{property of defect diminishing} \cite{defdim} (Proposition \ref{prop:defdim:BC}) we are also able to prove relative versions of these criteria (Corollary \ref{cor:BC:rel}), which imply the following:

\begin{theorem}[Theorem \ref{thm:fi}]
\label{intro:thm:fi}

Let $\Gamma$ be a finitely presented group, which is uniformly $\GGL(\Zp)$-stable with a linear estimate. Then the same is true for finite-index supergroups of $\Gamma$, as well as for quotients of $\Gamma$ by a normally finitely generated normal subgroup $N$ satisfying $\HH^1(N; \mathbb{Q}) = 0$.
\end{theorem}

We refer the reader to Definition \ref{def:quant} for the precise definition of \emph{stability with a linear estimate}; suffice it to say that all groups whose $\GGL(\Zp)$-stability is proved in this paper satisfy this condition (this is also true of our results in positive characteristic, with the exception of $\mathbb{Z}/2\mathbb{Z}$ in characteristic $2$, see Proposition \ref{prop:z2z}). This allows to recover some of the results collected in Theorem \ref{intro:thm:vpfree} from the corresponding ones in Theorem \ref{intro:thm:pifree}, but is more flexible and allows for more examples of stability (Example \ref{ex:fi:application}). Let us point out that, in contrast with Ulam stability \cite[Lemma 2.2]{BOT}, uniform $\GGL(\Zp)$-stability is not preserved under taking general quotients (Remark \ref{rem:quotients}).

\medskip

Although most of the paper is concerned with proving positive results, we also produce non-examples. Groups which are not pointwise stable are easy to construct via an argument due to Arzhantseva and P\u{a}unescu \cite{a:comm} (see Lemma \ref{lem:GR}), which implies in our setting that a group which is LEF but not residually finite is not pointwise $\GGL(\Zp)$-stable (Corollary \ref{cor:cex_stab}). This construction is flexible enough to produce some groups with interesting combinations of properties (Proposition \ref{prop:sharp}).

However this method fails in the uniform setting, since uniform $\GGL(\Zp)$-approximability is equivalent to residual finiteness (by contrast, this method is fruitful in the study of uniform approximability with respect to symmetric and unitary groups \cite[Section 7]{Bharatandi}). We are still able to provide non-examples, by putting together groups with increasingly bad stability estimates; in particular we show that the group $\bigoplus_{i \geq 1} \mathbb{Z}/p^i\mathbb{Z}$ is not uniformly $\GGL(\Zp)$-stable (Proposition \ref{prop:countable:unst}). It turns out however that it is pointwise $\GGL(\Zp)$-stable.

These non-examples prove that some of our previous results were sharp:

\begin{proposition}[Propositions \ref{prop:sharp} and \ref{prop:countable:unst}]
\label{intro:prop:sharp}

Theorems \ref{intro:thm:pw_un} and \ref{intro:thm:rf} are sharp, in the following sense:
\begin{enumerate}
\item There exists a finitely generated group that is uniformly but not pointwise $\GGL(\Zp)$-stable.
\item There exists a countable group that is pointwise but not uniformly $\GGL(\Zp)$-stable.
\item There exists a finitely generated group that is not pointwise $\GGL(\Zp)$-stable, but whose largest residually finite quotient is.
\end{enumerate}
\end{proposition}

In order to turn the non-example of uniform stability into a finitely generated one, we use a modification of the wreath product constructions from \cite{embedding} to embed $\bigoplus_{i \geq 1} \mathbb{Z}/p^i\mathbb{Z}$ into a finitely generated group $\Gamma$, in such a way that the asymptotic homomorphisms proving instability extend. This leads to the following construction:

\begin{theorem}[Theorem \ref{thm:fg:unst}]
\label{intro:thm:unst}

There exists a finitely generated, residually finite group that is not uniformly $\GGL(\Zp)$-stable, and thus not pointwise stable.
\end{theorem}

The existence of a finitely presented unstable group remains open. This is one of a number of open questions that are listed in Section \ref{s:q}.

\begin{outline}
In Section \ref{s:preli} we review a few general facts about stability, approximation, residual properties and local embeddings. Moreover, we recall the interplay between lifting, splitting and cohomology, proving a useful technical lemma (Lemma \ref{lem:cohopk})
. In Section \ref{s:fam} we introduce the general framework of ultrametric families which will be the subject of our stability results, focusing on examples. In Section \ref{s:ultra} we treat stability with respect to general ultrametric families, proving Theorems \ref{intro:thm:pw_un} and \ref{intro:thm:rf}. In Section \ref{s:approx} we treat approximation, prove Theorem \ref{intro:thm:approx} and give our first examples of groups that are not pointwise stable. In Section \ref{s:vpropi} we focus on families that are virtually pro-$\pi$ for some set $\pi$ of primes, and prove generalisations of Theorem \ref{intro:thm:pifree}. In Section \ref{s:char0} we focus on the case of $\GGL(\Zp)$ (or more generally $\GGL(\oo)$ where $\oo$ is the ring of integers of a finite extension of $\Qp$), prove Theorem \ref{intro:thm:vpfree}, and end by discussing the differences in the case of positive characteristic. In Section \ref{s:unst}, we give examples of groups that are not uniformly stable. Finally, in Section \ref{s:BC} we take a bounded-cohomological approach to stability, which leads to the proof of Theorem \ref{intro:thm:fi}. Section \ref{s:q} is dedicated to open questions and suggestions for further research.
\end{outline}

\begin{remark*}
The results of this paper were part of the author's PhD thesis \cite{thesis}. He is currently supported by the Herchel Smith Postdoctoral Fellowship Fund.
\end{remark*}

\begin{acks}
I would like to start by thanking my advisor Alessandra Iozzi, as well as Marc Burger and Konstantin Golubev, for the constant support and guidance. Secondly, thanks to Bharatram Rangarajan for introducing me to the subject of stability, as well as the organisers of the Geometry Graduate Colloquium at ETH Zurich for inviting him; and to Alexander Lubotzky and Konstantin Golubev, who first suggested to me the study of $p$-adic stability and bounded cohomology, respectively.

I had the great chance to talk with many people during the development of this project, and for that I would like to thank Dante Bonolis, Michael Chapman, Yves de Cornulier (via MathOverflow), Arman Darbinyan, Alexandra Edletzberger, Dominik Francoeur, Jakob Glas, Victor Jaeck, Jan Kohlhaase, Maxim Mornev, Amritanshu Prasad, Peter Schneider, Zoran \v{S}uni\'c and Michele Zordan for lending me their time, and for the interesting and useful discussions, suggestions and comments. Most of all, thanks to Goulnara Arzhantseva for the invaluable comments, insight, and help with navigating the literature.

This project was started during the first wave of the COVID-19 pandemic, and the first version was completed during the first round of vaccinations in Switzerland. Therefore I need to extend a special thank you to my flatmates Etienne Batelier, Victor Jaeck and Lauro Silini, for making working from home a pleasant experience.

The current version of the paper is the result of many revisions, thanks to invaluable comments from referees, to whom I am very grateful.
\end{acks}

\pagebreak

\section{Preliminaries}
\label{s:preli}

\begin{notation}
In the sequel, $p$ always denotes a prime. If $\pi$ is a set of primes, we denote by $\pi'$ its complement, in particular $p'$ is the set of primes other than $p$. For an integer $n$, the $p$-adic valuation of $n$ is denoted by $\nu_p(n)$. That is, $p^{\nu_p(n)}$ is the largest power of $p$ that divides $n$. The set of natural numbers $\mathbb{N}$ starts at $1$. For simplicity, we use $x_n \to x$ instead of $x_n \xrightarrow{n \to \infty} x$ to denote convergence of sequences, whenever this does not lead to confusion.

Unless stated otherwise, $\Gamma$ denotes a countable discrete group. Given a set $S$ of letters, $F_S$ denotes the corresponding free group. We will always assume that $S \cap S^{-1} = \emptyset$. The trivial homomorphism to a group will be denoted by $\mathbbm{1}$. If $R \subset F_S$ we denote by $\normal{R}$ its normal closure, and $\langle S \mid R \rangle \coloneqq F_S/\normal{R}$. Once the presentation is fixed, we denote the projection map by $F_S \to \langle S \mid R \rangle : w \mapsto \overline{w}$.
\end{notation}

\subsection{Stability and approximation}

Let $\G$ be a family of groups equipped with arbitrary bi-invariant metrics. By \emph{bi-invariant} we mean that if $(G, d_G) \in \G$, and $g, h, k \in G$, then
\[d_G(g, h) = d_G(kg, kh) = d_G(gk, hk).\]
We denote by $G(\ee)$ the closed ball of radius $\ee > 0$ around the identity. The groups $G$ can thus be seen as topological groups. 

\medskip

Most of the paper is concerned with uniform stability (Definition \ref{def:stab}). What follows is an equivalent characterisation which allows to make the statements more quantitative:

\begin{lemma}
\label{lem:quant}

The following are equivalent:
\begin{enumerate}
\item $\Gamma$ is uniformly $\G$-stable.
\item For all $\ee > 0$ there exists $\dd > 0$ such that whenever $\f \colon \Gamma \to G \in \G$ satisfies $\defe(\f) \leq \dd$, there exists a homomorphism $\ff \colon \Gamma \to G$ such that $\dist(\f, \ff) \leq \ee$.
\end{enumerate}
\end{lemma}

\begin{proof}
$2. \Rightarrow 1.$ Suppose that $2.$ holds, and let $(\f_n \colon \Gamma \to G_n \in \G)_{n \geq 1}$ be a uniform asymptotic homomorphism. For all $\f_n$ let $\ff_n \colon \Gamma \to G_n$ be a homomorphism that minimises $\dist(\f_n, \ff_n)$ up to $1/n$ (we cannot ask that a minimising one exists in such a general situation). We need to show that $\dist(\f_n, \ff_n) \to 0$, so let $\ee > 0$ and let $\dd > 0$ be as in $2.$ for $\ee/2$. Let $N$ be large enough so that $N \geq 2/\ee$ and $\defe(\f_n) \leq \dd$ for all $n \geq N$. Then $\dist(\f_n, \ff_n) \leq \ee/2 + 1/n \leq \ee$. 

\medskip

$1. \Rightarrow 2.$ Suppose that $2.$ does not hold. Then there exists $\ee > 0$ with the following property: for all $n \geq 1$ there exists $\f_n \colon \Gamma \to G_n \in \G$ such that $\defe(\f_n) \leq 1/n$ but for every homomorphism $\ff_n \colon \Gamma \to G_n$ we have $\dist(\f_n, \ff_n) > \ee$. The sequence $\f_n$ provides a counterexample to the uniform stability of $\Gamma$.
\end{proof}

This characterisation allows to talk quantitatively about stability. This quantitative approach to stability was initiated by Becker and Mosheiff in \cite{quant}, although it is already mentioned in \cite{BOT}, and is of greatest interest for applications to computer science (see also \cite{test}).

\begin{definition}
\label{def:quant}
Let $\Gamma$ be uniformly $\G$-stable. For each $\ee > 0$, we denote by $D_{\Gamma}^{\G}(\ee)$ the maximal $\dd > 0$ such that Item $2.$ of Lemma \ref{lem:quant} holds for $\ee$.
\begin{enumerate}
\item If $D_{\Gamma}^{\G}(\ee) \geq \ee$ for all small enough $\ee > 0$, we say that $\Gamma$ is uniformly-$\G$ stable \emph{with an optimal estimate}.
\item If there exists $C > 0$ such that $D_{\Gamma}^{\G}(\ee) \geq C \cdot \ee$ for all small enough $\ee > 0$, we say that $\Gamma$ is uniformly $\G$-stable \emph{with a linear estimate}.
\item If there exists $C > 0$ such that $D_{\Gamma}^{\G}(\ee) \geq C \cdot \ee^2$ for all small enough $\ee > 0$, we say that $\Gamma$ is uniformly $\G$-stable \emph{with a quadratic estimate}.
\end{enumerate}
\end{definition}

In other words, a group is uniformly $\G$-stable with an optimal estimate, if there exists $\ee_0 > 0$ such that whenever $\f \colon \Gamma \to G \in \G$ satisfies $\defe(\f) \leq \ee \leq \ee_0$, there exists a homomorphism $\ff \colon \Gamma \to G$ such that $\dist(\f, \ff) \leq \ee$.

\begin{remark}
Suppose that $\f \colon \Gamma \to G \in \G$ is such that there exists a homomorphism $\ff \colon \Gamma \to G$ with $\dist(\f, \ff) \leq \ee$. Then the most one can conclude about the defect of $\f$ is that $\defe(\f) \leq 3 \ee$. Indeed:
\begin{align*}
d_G(\f(gh), \f(g)\f(h)) &\leq d_G(\f(gh), \ff(gh)) + d_G(\ff(g)\ff(h), \f(g)\f(h)) \\
&\leq \ee + d_G(\ff(g)\ff(h), \f(g)\ff(h)) + d_G(\f(g)\ff(h), \f(g)\f(h)) \\
&\leq \ee + \ee + \ee = 3 \ee,
\end{align*}
where passing from the second to the third line we used both left and right invariance. So stability with an optimal estimate is not a natural notion for an arbitrary family $\G$. On the other hand, if $\G$ is \emph{ultrametric}, the same computation as above, using the ultrametric inequality instead of the triangle inequality, shows that $\defe(\f) \leq \ee$.
\end{remark}

Our positive results for stability will all be with an optimal, linear, or quadratic estimate, but of course one can analogously define stability with a polynomial estimate, with an exponential estimate, and so on.

\begin{remark}
Rephrasing stability in terms of presentations (see Proposition \ref{prop:stab_equiv}) allows to give a quantitative characterisation of pointwise stability of finitely presented groups, as in \cite{a:comm}. We will see that the two quantitative notions coincide (see Corollary \ref{cor:quant} and Remark \ref{rem:quant:fp}).
\end{remark}

The following observation, first made in \cite{a:comm} (see also \cite[Proposition 3]{GR}), is the key to connecting the notions of stability and approximation (Definitions \ref{def:stab} and \ref{def:approx}). Recall that given a group $\Gamma$ and a family of groups $\G$, we say that $\Gamma$ is \emph{residually-$\G$} if for all $1 \neq g \in \Gamma$ there exists a homomorphism $\f \colon \Gamma \to G \in \G$ such that $g \notin \ker(\f)$. It is \emph{fully residually-$\G$} if for every finite subset $K \subset \Gamma$ there exists a homomorphism $\f \colon \Gamma \to G \in \G$ such that $f|_K$ is injective. If the groups in $\G$ are residually finite, and $\Gamma$ is residually-$\G$, then $\Gamma$ is residually finite.

\begin{lemma}[Arzhantseva--P\u{a}unescu {\cite[Theorem 4.3]{a:comm}}]
\label{lem:GR}

Let $\Gamma$ be a group that is both $\G$-approximable and pointwise $\G$-stable. Then $\Gamma$ is fully residually-$\G$.
\end{lemma}


This is mostly useful for counterexamples. For instance let $\G$ be a family of residually finite groups; then if $\Gamma$ is not residually finite, it cannot be simultaneously approximable and pointwise stable. Such families include all families of finite groups, as well as the \emph{profinite families} that will be defined in Section \ref{s:fam}. Similarly, if $\G$ is a family of locally residually finite groups, then the same holds under the additional hypothesis that $\Gamma$ is finitely generated: such classes include all linear groups by a theorem of Malcev \cite{Malcev}. This is the way the authors in \cite{GR} provide the first non-examples of pointwise $(S_n, d_H)$-stable equations. It is also the approach suggested in \cite{a:comm} and successfully realised in \cite{GLT}, by which the authors provide the first example of a non-approximable group, with respect to $\{ (\U(n), \| \cdot \|_{Frob}) : n \geq 1 \}$. 

\medskip

Another useful equivalent characterisation of pointwise stability, due to Arzhantseva and P\u{a}unescu, is in terms of ultraproducts.
We will only use it in Section \ref{s:BC} and apply it to finitely presented groups, so we state it in this setting which is the one from \cite{a:comm}.
In the statement, $\prod\limits_{n \to \omega} G_n$ denotes the \emph{metric ultraproduct} of the $G_n$ with respect to the free (aka non-principal) ultrafilter $\omega$. That is, $\prod\limits_{n \to \omega} G_n$ is the quotient of the direct product by the normal subgroup $\{ (g_n)_{n \geq 1} : d_{G_n}(g_n, 1_{G_n}) \xrightarrow{n \to \omega} 0 \}$.

\begin{lemma}[Arzhantseva--P\u{a}unescu {\cite[Theorem 4.2]{a:comm}}]
\label{lem:stab_up}

Let $\Gamma$ be a countable group, and let $\G$ be a family of groups equipped with bi-invariant metrics. The following are equivalent:
\begin{enumerate}
\item $\Gamma$ is pointwise $\G$-stable.
\item For every free ultrafilter $\omega$ on $\mathbb{N}$ and every sequence $(G_n)_{n \geq 1} \subset \G$, every homomorphism $\Gamma \to \prod\limits_{n \to \omega} G_n$ lifts to a homomorphism $\Gamma \to \prod\limits_{n \geq 1} G_n$.
\end{enumerate}
\end{lemma}

\begin{remark}
Typically the metric ultraproduct is defined as a quotient of the subgroup of the direct product consisting of \emph{bounded} sequences of elements in $G_n$; that is, $\{ (g_n)_{n \geq 1} : \sup d_n(g_n, 1_{G_n}) < \infty \}$. However most of the times the diameter of the $G_n$ is uniformly bounded. This will be the case when we use Lemma \ref{lem:stab_up} in Section \ref{s:BC}: there $G_n = \GGL_{k_n}(\oo)$ has diameter $1$.
\end{remark}

\subsection{Residual properties and local embeddings}
\label{ss:rf}

See \cite[Chapters 2, 3, 7]{Cell} for more details. 

\medskip

Let $\C$ be a class of groups, that for simplicity we assume to be closed under taking subgroups. If the class $\C$ is closed under taking direct products, then a residually-$\C$ group is automatically fully residually-$\C$.
The following result is standard, so we include it here for reference and omit the proof:

\begin{lemma}
\label{lem:largest_rf}

Let $\C$ be a class of groups closed under taking subgroups, $\Gamma$ a group, $K$ the intersection of all normal subgroups of $\Gamma$ whose quotient belongs to $\C$, and $\Omega \coloneqq \Gamma / K$. Then $\Omega$ is the largest residually-$\C$ quotient of $\Gamma$; that is, $\Omega$ is residually-$\C$, and every homomorphism from $\Gamma$ to a residually-$\C$ group factors through $\Omega$.
\end{lemma}

The following will be common examples throughout this paper.

\begin{example}
\label{ex:sym0}

Let $\Sym(\mathbb{Z})$ be the group of permutations of the integers, let $\Sym_0(\mathbb{Z})$ be the subgroup of permutations with finite support, and $\Alt_0(\mathbb{Z})$ the subgroup of even permutations with finite support.
Let $T \colon \mathbb{Z} \to \mathbb{Z} : n \to (n + 1)$ be the translation. We denote by $H$ the group generated by $\Sym_0(\mathbb{Z})$ and $T$, and by $H^+$ the group generated by $\Alt_0(\mathbb{Z})$ and $T$, which has index $2$ in $H$. Then $H$ splits as a semidirect product $\Sym_0(\mathbb{Z}) \rtimes \langle T \rangle$, and similarly $H^+$ splits as a semidirect product $\Alt_0(\mathbb{Z}) \rtimes \langle T \rangle$. Since $\Alt_0(\mathbb{Z})$ has no non-trivial finite quotients, and $\mathbb{Z}$ is residually finite, the largest residually finite quotient of $H^+$ is $\mathbb{Z}$. Similarly the largest residually finite quotient of $H$ is $\mathbb{Z} \times \mathbb{Z}/2\mathbb{Z}$.
\end{example}

\begin{remark}
The group $H$ above was first considered by Malcev in \cite{Malcev}. It is commonly referred to as \emph{Houghton's second group}, in referece to \cite{Houghton}, and denoted by $H_2$. We will simply refer to it as Houghton's group, and denote it by $H$.
\end{remark}

\begin{example}
\label{ex:wr}

Let $\Gamma, \Lambda$ be groups, denote $\Sigma_\Lambda \Gamma = \bigoplus_{h \in \Lambda} \Gamma_h$, where each $\Gamma_h$ is an indexed copy of $\Gamma$, and consider the wreath product $\Gamma \wr \Lambda = \Sigma_\Lambda \Gamma \rtimes \Lambda$ where $\Lambda$ acts on the direct sum by shifting the coordinates. Common examples in combinatorial group theory are the \emph{lamplighter groups}, where $\Lambda = \mathbb{Z}$ and $\Gamma$ is finite (often the name is used to refer to the specific case of $\Gamma = \mathbb{Z}/2\mathbb{Z}$).

Suppose that $\Lambda$ is infinite. Then the projection of $\Gamma \wr \Lambda$ onto its largest residually finite quotient factors through $\Ab(\Gamma) \wr \Lambda$, where $\Ab(\Gamma)$ is the abelianisation of $\Gamma$. This latter is residually finite if both $\Ab(\Gamma)$ and $\Lambda$ are residually finite \cite[Theorem 3.2]{wr_rf} (we will mostly be interested in the first statement).
\end{example}

In Section \ref{s:approx} we are concerned with local embeddings, which sit between approximability in the sense of Definition \ref{def:approx} and the corresponding residual property. The notion of local embedding was formally introduced in \cite{LEF} although the idea goes back to Malcev \cite{Malcev}. We refer the reader to \cite[Chapter 7]{Cell} for details and proofs.

\begin{definition}[Malcev, Gordon--Vershik]
Let $\Gamma, C$ be groups, $K \subset \Gamma$ a finite subset. A map $f \colon \Gamma \to C$ is a \emph{$K$-local embedding} if $f|_K$ is injective and $f(g)f(h) = f(gh)$ whenever $g, h \in K$.

Let $\C$ be a class of groups (closed under taking subgroups). The group $\Gamma$ is \emph{locally embeddable into $\C$} if for every finite subset $K \subset \Gamma$ there exists a $K$-local embedding $f \colon \Gamma \to  C \in \C$. When $\C$ is the class of finite groups, $\Gamma$ is said to be \emph{LEF (Locally Embeddable in the class of Finite groups}).
\end{definition}

Here is an equivalent characterisation of local embeddability. In the statement, $\prod\limits_{n \to \omega} C_n$ denotes the \emph{set-theoretic ultraproduct} of the $C_n$ with respect to the free ultrafilter $\omega$. That is, $\prod\limits_{n \to \omega} C_n$ is the quotient of the direct product by the normal subgroup $\{ (g_n)_{n \geq 1} : \{n : g_n = 1_{C_n} \} \in \omega \}$ (equivalently, the metric ultraproduct where the $C_n$ are endowed with the discrete metric).

\begin{proposition}[Gordon--Vershik, see {\cite[Theorem 7.2.5]{Cell}}]
\label{prop:lef_up}

Let $\Gamma$ be a countable group, $\C$ a class of groups. Then $\Gamma$ is locally embeddable into $\C$ if and only if it embeds into $\prod\limits_{n \to \omega} C_n$ for some sequence $(C_n)_{n \geq 1} \subset \C$.
\end{proposition}

%

The properties of being residually-$\C$ and locally embeddable into $\C$ are related by the following result:

\begin{proposition}[Gordon--Vershik, see {\cite[Corollaries 7.1.14 and 7.1.21]{Cell}}]
\label{prop:lec_rc}

Let $\C$ be a class of groups closed under taking subgroups.
\begin{enumerate}
\item Every fully residually-$\C$ group is locally embeddable into $\C$.
\item Every finitely presented group that is locally embeddable into $\C$ is fully residually-$\C$.
\end{enumerate}
\end{proposition}

Importantly, Item $2.$ does not hold for general finitely generated groups.

\begin{example}
\label{ex:sym0_lef}

Let $H$ denote Houghton's group, and $H^+$ its index-$2$ subgroup (Example \ref{ex:sym0}). $H$ is finitely generated: by the transposition $(12)$ and the translation $T$. It is not residually finite: we computed its largest residually finite quotient in Example \ref{ex:sym0}. However $H$ is LEF \cite{LEF}. Similarly, $H^+$ is finitely generated (by the cycle $(123)$ and $T$, or because it has finite index in $H$), not residually finite (again by Example \ref{ex:sym0}), but it is LEF (since a subgroup of a LEF group is clearly LEF).
\end{example}

\begin{example}
\label{ex:wr_lef}

Let $\Gamma, \Lambda$ be finitely generated LEF groups. Then their wreath product $\Gamma \wr \Lambda$ is finitely generated and LEF \cite[Theorem 2.4 (ii)]{LEF}. However, if $\Gamma$ is non-abelian and $\Lambda$ is infinite, then $\Gamma \wr \Lambda$ is not residually finite (Example \ref{ex:wr}).
\end{example}

A natural framework in which to see these properties is that of \emph{marked groups}, introduced by Grigorchuk in \cite{grigor_mg}. We use the point of view of normal subgroups, as in \cite{Cell} (see also \cite{limit}). Given a countable group $\Gamma$, the set of \emph{$\Gamma$-marked groups} is the set of isomorphism classes of quotients of $\Gamma$ up to automorphism, which is identified with the set of normal subgroups $\mathcal{N}(\Gamma)$ of $\Gamma$. The space of marked groups is this set endowed with the subspace topology $\mathcal{N}(\Gamma) \subset \mathcal{P}(\Gamma) \cong \{ 0, 1 \}^\Gamma$. This topology is totally disconnected, compact and - since $\Gamma$ is assumed to be countable - metrizable. 

\medskip

Given a class $\C$ closed under taking subgroups, both the residual property and local embeddability admit characterisations in terms of the space of marked groups:

\begin{theorem}[Gordon--Vershik, see {\cite[Proposition 3.4.3 and Corollary 7.1.20]{Cell}}]
\label{thm:mg}

Let $\mathcal{C}$ be a class of groups closed under taking subgroups. Let $\Gamma = \langle S \mid R \rangle$ be a countable group, $N \coloneqq \normal{R} \leq F_S$. Then:
\begin{enumerate}
\item $\Gamma$ is fully residually-$\mathcal{C}$ if and only if there exists a sequence $N_k \to \{ 1 \} \in \mathcal{N}(\Gamma)$ such that $\Gamma/N_k \in \mathcal{C}$ for all $k$.
\item $\Gamma$ is locally embeddable into $\mathcal{C}$ if and only if there exists a sequence $N_k \to N \in \mathcal{N}(F_S)$ such that $F_S/N_k \in \mathcal{C}$ for all $k$.
\end{enumerate}
\end{theorem}

This point of view gives many more examples of finitely generated LEF groups that are not residually finite.

\begin{example}
\label{ex:small_canc}

Let $\Gamma = \langle S \mid R \rangle$, with $S$ finite and $R$ possibly infinite, be a presentation satisfying the $C'(1/6)$ small cancellation condition: we call $\Gamma$ a \emph{classical small cancellation group}. Letting as usual $N = \normal{R} \leq F_S$, there exists a sequence $N_k \to N \in \mathcal{N}(F_S)$, where the $N_k$ are normally generated by $k$ elements of $R$. Since the groups $F_S/N_k$ are defined by \emph{finite} $C'(1/6)$ presentation, they are hyperbolic \cite{Gromov2}, and moreover they are residually finite. The last statement follows by combining the following three deep results: finitely presented $C'(1/6)$ groups are cubulable \cite{Wise}, hyperbolic cubulable groups are virtually special-cubulable \cite{Agol}, and special-cubulable groups embed into RAAGs \cite{scc}. By Item $2.$ of Theorem \ref{thm:mg}, it follows that $\Gamma$ is LEF, being a marked limit of residually finite groups. Let us point out that this is a special property of the $C'(1/6)$ small cancellation condition: there exist finitely presented groups satisfying more relaxed small cancellation conditions that are not residually finite \cite{Wiserf}.

On the other hand there are many examples of (infinitely presented) classical small cancellation groups that are not residually finite. A classical example is a group constructed by Pride \cite{Pride}:
\[\Gamma = \langle a, b \mid a u_1, b v_1, a u_2, b v_2, \ldots \rangle,\]
where $u_n, v_n$ are well-chosen words in $a^n, b^n$ so that the presentation is $C'(1/6)$. For any such choice, this group is infinite and has no proper finite-index subgroup, so it is in particular not residually finite. For more examples see \cite[Section 2]{a:small} and the references therein; the authors also explain how to construct continuum-many isomorphism classes of non-residually finite classical small cancellation groups.
\end{example}

\subsection{Non-Archimedean fields}
\label{ss:nona}

See \cite{ANT, NFA} for more details. 

\medskip

Let $(\K, | \cdot |)$ be a normed field. If the group $|\K^\times| \leq \mathbb{R}^\times_{> 0}$ is discrete, it is either trivial, or of the form $r^\mathbb{Z}$, for some $0 < r < 1$, in the latter case we say that $\K$ is \emph{discretely valued}. The norm is \emph{non-Archimedean} if it satisfies the \emph{ultrametric inequality}, namely $|x + y| \leq \max\{ |x|, |y| \}$. Then the closed ball of radius $1$ is a local ring, called the \emph{ring of integers} and denoted by $\oo$, whose maximal ideal is the open ball of radius $1$, denoted by $\pp$. The quotient $\kk \coloneqq \oo / \pp$ is called the \emph{residue field} of $\K$, and its characteristic is called the \emph{residual characteristic} of $\K$. 

\medskip

If the induced topology is locally compact and non-discrete then $\K$ is called a \emph{local field}. Non-Archimedean local fields are precisely those that are discretely valued and have a finite residue field. It follows that $\oo$ is compact, and that the maximal ideal $\pp$ is principal, generated by an element $\uni$ such that $|\uni| = r$, using the notation above. Such an element is called a \emph{uniformiser}. A fundamental result is \emph{Ostrowski's Theorem} \cite[Proposition 5.2]{ANT}, which states that a non-Archimedean local field is either a finite extension of $\Qp$ (if it has characteristic $0$) or of $\Fp((X))$ (if it has characteristic $p$); in the latter case it is isomorphic to $\Fq((X))$, for $q$ a power of $p$.

\begin{example}
We write elements of $\Qp$ in the usual series form $x = \sum_{i \geq i_0} a_i p^i$, where $a_i \in \{0, \ldots, p-1\}$ and $a_{i_0} \neq 0$. Then $|x|_p \coloneqq p^{-i_0}$, so the norm takes values in $p^\mathbb{Z}$: in the previous notation we have $r = p^{-1}$. The ring of integers is $\oo = \Zp$ with uniformiser $\uni = p$ and maximal ideal $\pp = p \Zp$. The residue field is $\kk \cong \Fp$.
\end{example}

\begin{example}
$\Fq((X))$ is the field of Laurent series, so it consists of elements of the form $x = \sum_{i \geq i_0} a_i X^i$, where $a_i \in \Fq$ and $a_{i_0} \neq 0$. Then $|x|_q \coloneqq q^{-i_0}$, so the norm takes values in $q^\mathbb{Z}$: in the previous notation we have $r = q^{-1}$. The ring of integers is the ring of Laurent polynomials $\oo = \Fq[[X]]$ with uniformiser $\uni = X$ and maximal ideal $\pp = X \Fq[[X]]$. The residue field is $\kk \cong \Fq$.
\end{example}

\begin{example}
\label{ex:K|Qp}

Let $\K$ be a non-Archimedean local field, and let $\K' | \K$ be a finite extension. Then the norm on $\K$ extends uniquely to a norm on $\K'$, which it turn makes $\K'$ into a non-Archimedean local field \cite[Theorem 4.8]{ANT}. Since the extension is unique, the norm on $\K$ extends uniquely to a norm on the algebraic closure $\K^{alg}$ as well.
\end{example}

We next review the basics of functional analysis over the non-Archimedean local field $\K$, which are needed in Section \ref{s:BC}. A \emph{normed $\K$-vector space} is a $\K$-vector space $E$ endowed with a norm $\| \cdot \| \colon E \to \mathbb{R}_{\geq 0}$ that is positive-definite, $\K$-multiplicative, and satisfies the ultrametric inequality. If the induced metric is complete, we say that $E$ is a \emph{$\K$-Banach space}. The norm of $E$ need not take the same set of values as that of $\K$. If this is the case, we say that the norm on $E$ is \emph{solid}. In our case $\K$ is a local field, so it is in particular complete and discretely valued: over such fields, Banach spaces with a solid norm are isometrically classified \cite[Theorem 2.5.4]{NFA}.

Given two normed $\K$-vector spaces $E$ and $F$, a linear map $T \colon E \to F$ is continuous if and only if it is bounded, that is if and only if $\| T \|_{op} = \inf \{ C \geq 0 : \|Tx\|_F \leq C \| x \| \text{ for all } x \in E \} < \infty$. This applies in particular to the case when $F = \K$; then the space of bounded linear maps is called the continuous dual $E^*$, and is a normed $\K$-vector space endowed with the operator norm $\| T \|_{op}$. This is in general not equal to $\sup \{ |Tx|_{\K} : \|x\| \leq 1 \}$, but it is when the norm on $E$ is solid. Since $\K$ is complete, a dual space is always Banach.

\begin{definition}
A \emph{normed $\K[\Gamma]$-module} is a normed $\K$-vector space $E$ with a linear isometric action of the group $\Gamma$. If $E$ is Banach, we say that $E$ is a \emph{Banach $\K[\Gamma]$-module}. If $E = F^*$ is the dual of a $\K[\Gamma]$-module endowed with the operator norm and the contragradient action, we say that $E$ is a \emph{dual $\K[\Gamma]$-module}.
\end{definition}

The following basic theorem of functional analysis holds also in the non-Archimedean context:

\begin{theorem}[Open Mapping Theorem, see {\cite[Theorem 2.1.17]{NFA}}]
\label{thm:open}

Let $T \colon E \to F$ be a bounded surjective linear map between Banach spaces. Then $T$ is open.
\end{theorem}

\subsection{Splitting and lifting}
\label{ss:preli:split}

See \cite[Chapter IV]{Brown} and \cite[Chapter 7]{Rotman} for more details. These results will be relevant starting from Section \ref{s:vpropi}. All statements -- apart from Lemma \ref{lem:cohopk} -- are standard, but we recall them here since we will need some specific constructions in the sequel. 

\medskip

Let $1 \to N \to E \to Q \to 1$ be a group extension. A \emph{splitting} is a homomorphic section $ \sigma \colon Q \to E$: if one exists we say that the extension \emph{splits}, which is equivalent to $E \cong N \rtimes Q$, and the image of this section is a \emph{complement of $N$}. Consider splittings $\sigma_1, \sigma_2 \colon Q \to E$ with complements $Q_1, Q_2$. The complements are \emph{conjugate} if there exists $g \in E$ such that $gQ_1g^{-1} = Q_2$, which implies that there exists $g \in N$ such that $g \sigma_1 g^{-1} = \sigma_2$.
Because of this we may refer to the splittings themselves being conjugate, and the conjugating element may be assumed to lie in $N$.

Splitting problems are a special case of the more general \emph{lifting problems}: given a group $\Gamma$, and a homomorphism $\Gamma \xrightarrow{\f} G/N$, can this be lifted to a homomorphism $\Gamma \xrightarrow{\ff} G$? Are all lifts conjugate?
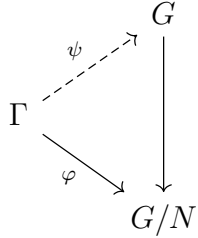
\[\begin{tikzcd}
	& G \\
	\Gamma \\
	& {G/N}
	\arrow["\f"', from=2-1, to=3-2]
	\arrow[from=1-2, to=3-2]
	\arrow["\ff", dashed, from=2-1, to=1-2]
\end{tikzcd}\]

A splitting problem is the special case in which $\Gamma = G/N$ and $\f$ is the identity. But every lifting problem can be reduced to a splitting problem as follows. Consider the pullback $G \times_{\f} \Gamma \coloneqq \{ (x, g) \in G \times \Gamma : xN = \f(g) \in G/N \}$, denote by $pr_{1, 2}$ the natural projections, and notice that there is a natural embedding $j \colon N \to G \times_{\f} \Gamma : n \mapsto (n, 1)$. Then we have the following commutative diagram with exact rows:
\[\begin{tikzcd}
	1 & N & {G \times_{\f} \Gamma} & \Gamma & 1 \\
	1 & N & G & {G/N} & 1.
	\arrow[from=1-1, to=1-2]
	\arrow["{pr_2}", from=1-3, to=1-4]
	\arrow[from=1-4, to=1-5]
	\arrow[from=2-1, to=2-2]
	\arrow[hook, from=2-2, to=2-3]
	\arrow[from=2-3, to=2-4]
	\arrow[from=2-4, to=2-5]
	\arrow["{=}"', from=1-2, to=2-2]
	\arrow["{pr_1}", from=1-3, to=2-3]
	\arrow["\f", from=1-4, to=2-4]
	\arrow["j", from=1-2, to=1-3]
\end{tikzcd}\]
The lift $\ff$ exists if and only if the exact sequence on the top row splits.
Moreover, two lifts are conjugate if and only if the corresponding two splittings are conjugate, and by the discussion above the conjugating element may be chosen to lie in $N$. Lifting problems will appear more naturally in this paper, so we will state our results in those terms.

The following will be the fundamental tool in Section \ref{s:vpropi} \cite[Theorem 7.41]{Rotman}:

\begin{theorem}[Schur--Zassenhaus, see {\cite[Theorem 7.41]{Rotman}}]
\label{thm:SZ}

Let $\Gamma, N$ be finite groups such that the orders of $\Gamma$ and $N$ are coprime. Then, whenever $N \unlhd G$, every homomorphism $\Gamma \to G/N$ lifts to a homomorphism $\Gamma \to G$, and any two lifts are $N$-conjugate.
\end{theorem}

The conjugacy statement admits a simple proof in the case in which either $N$ or $\Gamma$ is solvable. This will be the setting for our applications to $\GGL(\oo)$-stability: in fact $N$ will even be a $p$-group. The general case follows from the Odd Order Theorem \cite{FT}: if $N$ and $\Gamma$ have coprime order then at least one of them has odd order, and so is solvable. To the author's knowledge, no simpler proof is known.

The Schur--Zassenhaus Theorem alone will be enough for Section \ref{s:vpropi}. In Section \ref{s:char0} we will need a stronger lifting criterion, under additional hypotheses on $N$. Recall that given an extension $1 \to N \to E \to Q \to 1$ with $N$ abelian, the action of $E$ on $N$ by conjugacy induces an action of $Q$ on $N$, which defines cohomology groups $\HH^\bullet(Q; N)$. To this extension one associates a second cohomology class $[E] \in \HH^2(Q; N)$. A cocycle representing it can be defined as follows: given a (set-theoretic) section $\tilde{\sigma} \colon Q \to E$, define $z \colon Q \times Q \to N : (g, h) \mapsto z(g, h) = \tilde{\sigma}(gh)^{-1} \tilde{\sigma}(g) \tilde{\sigma}(h)$. If the extension splits, let us fix a splitting $E \cong N \rtimes Q$ (which as a set is just the cartesian product $N \times Q$). Then each splitting $\sigma \colon Q \to N \rtimes Q$ defines a first cohomology class $[\sigma] \in \HH^1(Q; N)$. A cocycle representing it can be defined by $c \colon Q \to N : g \mapsto c(g)$ where $\sigma(g) = (c(g), g) \in N \rtimes Q$.

\begin{theorem}[Schreier, see {\cite[Chapter IV, Proposition 2.3 and Theorem 3.1.2]{Brown}}]
The extension $1 \to N \to E \to Q \to 1$ (where $N$ is abelian) splits if and only if the associated cohomology class $[E] \in \HH^2(Q; N)$ vanishes. Two splittings $\sigma_1, \sigma_2$ are conjugate if and only if the associated cohomology classes $[\sigma_1], [\sigma_2] \in \HH^1(Q; N)$ coincide.
\end{theorem}

So cohomology vanishing gives rise to splitting and conjugacy results. The Schur--Zassenhaus Theorem is proved this way: one first reduces to the case where $N$ is an elementary abelian $p$-group, and then uses the fact that if $Q$ has order coprime to $p$, then $\HH^n(Q; N) = 0$ for all $n \geq 1$. The following lemma strengthens this last step:

\begin{lemma}
\label{lem:cohopk}

Let $\Gamma$ and $N$ be finite groups, with $N$ a $\Zpk$-module, and $\nu_p(|\Gamma|) \leq l \leq k$. Let $H \coloneqq p^l \cdot N$, which is a characteristic subgroup of $N$. Then, whenever $G$ is a group containing $N$ as a normal subgroup, for every homomorphism $\Gamma \xrightarrow{\f} G/H$, the homomorphism $\Gamma \xrightarrow{\of} G/N$, obtained by composing $\f$ with the quotient map, lifts to a homomorphism $\ff \colon \Gamma \to G$.
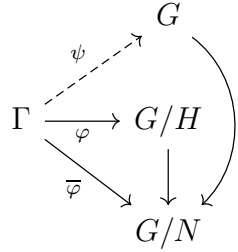
\[\begin{tikzcd}
	& G \\
	\Gamma & {G/H} \\
	& {G/N}
	\arrow["\ff", dashed, from=2-1, to=1-2]
	\arrow["{\overline{\f}}"', from=2-1, to=3-2]
	\arrow["\f"', from=2-1, to=2-2]
	\arrow[from=2-2, to=3-2]
	\arrow[from=1-2, to=3-2, bend left = 50]
\end{tikzcd}\]
Moreover, if they exist, any two lifts of $\f$ are $N$-conjugate.
\end{lemma}

\begin{remark}
Note that a lift of $\f$ is also a lift of $\overline{\f}$, but the converse need not hold. The first part of the lemma only guarantees the existence of the latter, and the second one is a statement about the former.
\end{remark}

\begin{proof}
A lift of $\of$ exists if and only if the extension $1 \to N \to G \times_{\of} \Gamma \to \Gamma \to 1$ splits, which in turn is equivalent to the vanishing of the corresponding cohomology class in $\HH^2(\Gamma; N)$. A cocycle representing it is given by $(g, h) \mapsto \tilde{\sigma}(gh)^{-1} \tilde{\sigma}(g) \tilde{\sigma}(h)$, where $\tilde{\sigma} \colon \Gamma \to G \times_{\of} \Gamma$ is some (set-theoretic) section. We can choose $\tilde{\sigma} \colon \Gamma \to G \times_{\f} \Gamma \leq G \times_{\of} \Gamma$, and then the corresponding cocycle will take values in $H$. Similarly, a lift of $\f$ defines a class in $\HH^1(\Gamma; N)$ such that the cocycle representing it takes values in $H$. Since $H = p^l \cdot N$, to show that these classes vanish it suffices to show that $p^l \cdot \HH^n(\Gamma; N) = 0$ for $n = 1, 2$. By \cite[Corollary III.10.2]{Brown}, we have $|\Gamma| \cdot \HH^n(\Gamma; N) = 0$ for all $n \geq 1$. Write $|\Gamma| = p^a m$ where $a \leq l$ and $(m, p) = 1$. The latter condition ensures that $m$ is a unit in $\Zpk$, and so $m \cdot \HH^n(\Gamma; N) = \HH^n(\Gamma; N)$. Therefore $p^a \cdot \HH^n(\Gamma; N) = 0$ and so $p^l \cdot \HH^n(\Gamma; N) = 0$, which concludes the proof.
\end{proof}

%
%

\pagebreak

\section{Ultrametric families}
\label{s:fam}

The subject of this paper is stability with respect to families $\G$ all of whose groups are equipped with bi-invariant ultrametrics. Before moving to stability in the Section \ref{s:ultra}, here we prove some basic facts about such families, and present several examples.

\subsection{Basic facts and terminology}

\begin{definition}
We say that the metric group $(G, d)$ is \emph{ultrametric} if the ultrametric inequality holds:
\[d(g, k) \leq \max\{ d(g, h), d(h, k) \} \text{ for all } g, h, k \in G.\]
We say that the family $\G$ is \emph{ultrametric} if every $G \in \G$ is ultrametric. If moreover the groups in $\G$ are compact, the family $\G$ is called \emph{profinite}.
\end{definition}

The most important general property of ultrametric groups is contained in the following lemma. Recall that, given a metric group $(G, d)$, the closed ball of radius $\ee > 0$ around the identity is denoted by $G(\ee)$.

\begin{lemma}
\label{lem:ultra_gp}

Let $(G, d)$ be a group equipped with a bi-invariant ultrametric. Then $G(\ee)$ is a clopen normal subgroup of $G$.
\end{lemma}

\begin{proof}
Closed balls in ultrametric spaces are automatically open. If $g, h \in G(\ee)$, then
\[d(gh^{-1}, 1) = d(g, h) \leq \max\{ d(g, 1), d(1, h) \} \leq \ee.\]
This shows that $G(\ee)$ is a subgroup, and it is normal because $d$ is conjugacy-invariant. 
\end{proof}

This allows to quotient out balls, leading to the following definition:

\begin{definition}
\label{def:mq}

Let $(G, d)$ be an ultrametric group. The quotients $G/G(\ee)$ are called \emph{metric quotients} of $G$. Given an ultrametric family $\G$, we denote by $\MQ(\G)$ the family of metric quotients of groups in $\G$ and all subgroups thereof.
\end{definition}

\begin{remark}
Including subgroups of metric quotients in $\MQ(\G)$ will allow us to treat homomorphisms $\Gamma \to G/G(\ee)$ (which will be shortly introduced in Lemma \ref{lem:mq} and be used throughout) with a unified language, without stressing whether they are surjective or not.
\end{remark}

A metric quotient of $G$ is a discrete group, since each $G(\ee)$ is open, and comes equipped with a quotient metric. If $G$ is moreover compact, then all metric quotients are finite; this is why we call such families profinite.

The possibility of quotienting out balls has very strong consequences for stability and approximation. These are based on the following lemma:

\begin{lemma}
\label{lem:mq}

Let $\Gamma$ be a group, $(G, d)$ an ultrametric group and $\ee > 0$. Let $\f, \ff \colon \Gamma \to G$. Then:
\begin{enumerate}
\item $\defe(\f) \leq \ee$ if and only if the map $\f(\ee) \colon \Gamma \xrightarrow{\f} G \to G/G(\ee)$ is a homomorphism.
\item If both $\defe(\f), \defe(\ff) \leq \ee$, then $\dist(\f, \ff) \leq \ee$ if and only if $\f(\ee) = \ff(\ee)$.
\end{enumerate}
\end{lemma}

\begin{proof}
$1.$ The map $\f(\ee)$ is a homomorphism if and only if $\f(gh)G(\ee) = \f(g)\f(h) G(\ee)$, which is equivalent to $\defe_{g, h}(\f) \leq \ee$, for all $g, h \in \Gamma$.

\medskip

$2.$ The homomorphisms $\f(\ee)$ and $\ff(\ee)$ coincide if and only if $\f(g) G(\ee) = \ff(g) G(\ee)$, which is equivalent to $\dist_g(\f, \ff) \leq \ee$, for all $g \in \Gamma$.
\end{proof}

\begin{definition}
With the notation of the previous lemma, we refer to the homomorphisms $\{ \f(\ee) : \ee \geq \defe(\f) \}$ as the \emph{homomorphisms induced by} $\f$.
\end{definition}

\subsection{Examples}

A trivial example of bi-invariant metrics falls in the ultrametric framework.

\begin{example}
\label{ex:discr}

Given a discrete group $G$ it is always possible to define a \emph{discrete metric} on it by setting $d(g, h) = 0$ if $g = h$ and $d(g, h) = 1$ otherwise. This is a bi-invariant ultrametric. A family $\G$ of discrete groups equipped with their discrete metrics will be called a \emph{discrete family}.
\end{example}

Probabilistic stability problems with respect to this metric are mostly used in property testing (see e.g. \cite{PT} and \cite{BChap}). In our deterministic setting, we will see that stability with respect to such families is less interesting (see Example \ref{ex:discr:stab}). 

\medskip

Next, we present two constructions that allow to put natural ultrametrics on groups, and we apply them to give examples of ultrametric families. Then we move on to the main example that will be treated in this paper, namely integral matrices over non-Archimedean local fields. For the rest of this subsection, fix a strictly decreasing sequence $\eee = (\ek)_{k \geq 0} \subset (0, 1]$ with $\ez = 1$ and $\ek \to 0$.

\subsubsection{Groups acting on filtered sets}

Let $\Omega$ be a set with a (possibly finite) filtration $(\Omega_k)_{k \geq 1}$; that is $\Omega_k \subseteq \Omega_{k+1}$, and $\Omega = \bigcup_{k \geq 1} \Omega_k$. Let $G$ be a group acting faithfully on $\Omega$ preserving each $\Omega_k$. For $g \neq h$, define $d_{\Omega}(g, h) \coloneqq \ek$, where $k$ is the maximal integer such that $g|_{\Omega_k} = h|_{\Omega_k}$, and $k = 0$ if no such integer exists. This is a bi-invariant ultrametric: that it is a left-invariant ultrametric is clear, and right-invariance follows from the fact that $G$ preserves each $\Omega_k$.

\begin{example}
\label{ex:UT}

Let $G = T_n(\mathbf{R})$ be the group of invertible upper-triangular $(n \times n)$ matrices over a commutative unital ring $\mathbf{R}$. Then we can set $\Omega = \mathbf{R}^n$ and $\Omega_k = \operatorname{span}\{e_1, \ldots, e_k\}$, so $d_\Omega(g, h) \leq \ek$ if and only if the first $k$ columns of $g$ and $h$ are identical. We denote by $T(\mathbf{R})$ the family $\{ (T_n(\mathbf{R}), d) : n \geq 1 \}$, or $T(\mathbf{R})(\eee)$ if we want to emphasise the choice of the sequence $\eee$. Similarly we can consider the subgroup $UT_n(\mathbf{R})$ of upper-triangular matrices with ones on the diagonal, and obtain the family $UT(\mathbf{R})$.

Given $\ee > 0$, let $k \geq 1$ be the maximal integer such that $\ek \geq \ee$. Then $G(\ee) = G(\ek)$ is the subgroup consisting of upper-triangular matrices with a copy of $I_k$ in the upper-left corner. It follows that the metric quotient $G/G(\ee)$ is isomorphic to $T_k(\mathbf{R})$, or $T_n(\mathbf{R})$ if $k > n$. In particular all metric quotients are solvable, and even nilpotent-by-abelian, so $\MQ(T(\mathbf{R}))$ (Definition \ref{def:mq}) is contained in the class of nilpotent-by-abelian groups. Similarly metric quotients of $UT_n(\mathbf{R})$ are isomorphic to $UT_k(\mathbf{R})$ for some $k$, in particular they are all nilpotent.
\end{example}

\begin{example}
\label{ex:aut:filt}

See \cite{selfsim} for more details. Let $X$ be a finite alphabet and $\Omega \coloneqq X^*$ the regular rooted tree of finite words on $X$. We denote by $\Omega_k$ the set of words of length at most $k$. Then $G = \Aut(\Omega)$, the group of rooted tree automorphisms, can be equipped with the metric induced by this filtration, so $d_\Omega(g, h) \leq \ek$ if and only if $g$ and $h$ act the same way on $\Omega_k$, or equivalently they act the same way on the $k$-th level of the rooted tree. We denote by $\Aut(X_{\bullet}^*)$ the family $\{ (\Aut(X_n^*), d) : n \geq 1 \}$, where $X_n = \{ 1, \ldots, n \}$, or $\Aut(X_{\bullet}^*)(\eee)$ if we want to emphasise the choice of the sequence $\eee$.

Given $\ee > 0$, let $k \geq 1$ be the maximal integer such that $\ek \geq \ee$. Then $G(\ee) = G(\ek)$ is the stabiliser of $\Omega_k$, equivalently the stabiliser of the $k$-th level of the rooted tree. It follows that the metric quotient $G/G(\ee)$ is a $k$-fold iterated wreath product of the symmetric group $\Sym(X)$. In particular all metric quotients are finite $\pi$-groups, where $\pi$ is the set of primes $p \leq |X|$. So $\MQ(\Aut(X_\bullet^*))$ is the class of all finite groups (recall that we are also including subgroups of metric quotients in our definition of $\MQ$).
\end{example}

\subsubsection{Projective limits}

Let $(A_k)_{k \geq 1}$ be an inverse system of discrete groups, indexed by the directed set $\mathbb{N}$, and let $G$ be the corresponding projective limit. Then we can define $d(g, h) = \ee_k$, where $k \geq 1$ is the maximal integer such that $g$ and $h$ have the same image in $A_k$. This is a bi-invariant ultrametric.

Let $G_k$ be the kernel of the quotient map $G \to A_k$. Then $G_k = G(\ek)$ and $d(g, h) \leq \ek$ if and only if $g G_k = h G_k$. Given $\ee > 0$, let $k \geq 1$ be the minimal integer such that $\ee \geq \ek$. Then $G(\ee) = G(\ek) = G_k$ and it follows that the metric quotient $G/G(\ee)$ is isomorphic to $A_k$. 

\medskip

This construction applies to more general projective limits, where the defining system is countable but not necessarily indexed by $\mathbb{N}$. More precisely, if $G$ is the projective limit of $(A_i)_{i \in I}$, where $I$ is a countable directed set, then we may choose a sequence $(i_k)_{k \geq 1}$ that is order-isomorphic to $\mathbb{N}$ such that the corresponding system $(A_{i_k})_{k \geq 1}$ gives back $G$.

Note that given an inverse system $(A_i)_{i \in I}$, the set $I$ is countable if and only if $G$ admits a countable neighbourhood basis of the identity, which for topological groups is equivalent to being first-countable. This shows that countability of the index set is a necessary and sufficient condition to put a metric on $G$, since a metric space is always first-countable.

\begin{example}
\label{ex:prof}

Let $G$ be a first-countable profinite group. Then there exists a strictly nested sequence $(G_k)_{k \geq 1}$ of open finite-index normal subgroups that intersect trivially. The metric defined by $d(g, h) = \ek$, where $k$ is the maximal integer such that $gG_k = hG_k$, is a bi-invariant ultrametric. The corresponding metric quotients are the finite groups $G/G_k$. We denote this metric by $d = d((G_k)_{k \geq 1}, \eee)$.
\end{example}

Here are two specific examples of first-countable profinite groups metrised this way. The first we have already seen from another point of view.

\begin{example}
\label{ex:aut:prof}

Using the same notation as in Example \ref{ex:aut:filt}: let $X$ be a finite alphabet, $G = \Aut(X^*)$ and let $G_k$ be the stabiliser of the $k$-th level. Then $(G_k)_{k \geq 1}$ is a strictly nested sequence of finite-index normal subgroups that intersect trivially. Letting $d = d((G_k)_{k \geq 1}, \eee)$ be the corresponding metric from Example \ref{ex:prof}, we obtain the same metric as in Example \ref{ex:aut:filt}.
\end{example}

The next example falls into the more general construction, where the set $I$ is countable but not necessarily order-isomorphic to $\mathbb{N}$. See \cite{Neuk} for more details.

\begin{example}
\label{ex:gal:prof}

Let $K$ be a field, $L/K$ a Galois extension and $\Gal(L/K)$ the corresponding Galois group: this is the projective limit of the Galois groups of all intermediate finite Galois extensions. Let $G$ be the absolute Galois group, that is, the Galois group of the separable closure $K^{sep}/K$ of $K$: for every other Galois extension $L/K$ the group $\Gal(L/K)$ is a quotient of $G$ by a closed normal subgroup. Suppose that $G$ is first-countable. Then we can choose a strictly nested sequence $(G_k)_{k \geq 1}$ of open finite-index normal subgroups intersecting trivially, and project $G_k$ to $\Gal(L/K)$ for every other Galois extension: this provides an associated nested sequence $(G_k^L)_{k \geq 1}$ of finite-index open normal subgroups of $\Gal(L/K)$. This sequence of subgroups, and the sequence $\eee$, together define a metric as in Example \ref{ex:prof} on each $\Gal(L/K)$.

We denote by $\Gal(K)$ the profinite family obtained this way, or  $\Gal(K)((G_k)_{k \geq 1}, \eee)$ if we want to emphasise the choices of the sequences $(G_k)_{k \geq 1}$ and $\eee$.
\end{example}

Here is a characterisation of fields whose absolute Galois group is first-countable:

\begin{lemma}
\label{lem:gal_firstcount}

Let $L/K$ be a Galois extension of a field $K$. Then the following are equivalent:
\begin{enumerate}
\item $\Gal(L/K)$ is first-countable;
\item $L$ is a countable (increasing) union of intermediate finite Galois extensions;
\item There exist only countably many intermediate finite Galois extensions.
\end{enumerate}
\end{lemma}

\begin{proof}
$1. \Leftrightarrow 2.$ A countable nested sequence of open finite-index normal subgroups corresponds to a countable increasing sequence of intermediate finite Galois extensions. If we choose the subgroups to form a basis, so as to intersect trivially, the corresponding increasing union of extensions gives all of $L$, and viceversa. Moreover the existence of a countable union implies the existence of an increasing one, since the compositum of finitely many finite Galois extensions is again a finite Galois extension. 

\medskip

$2. \Leftrightarrow 3.$ Write $L = \bigcup_{i \geq 1} K_i$, where each $K_i$ is a finite intermediate Galois extension. By the Primitive Element Theorem, each intermediate finite Galois extension is contained in some $K_i$. By Galois Correspondence, there are only finitely many Galois subextensions of each $K_i$. The other implication is clear.
\end{proof}


The first easy example is the following:

\begin{example}
Let $K$ be a countable field. Then $K$ admits only countably many finite Galois extensions, since $K^{sep}$ itself is countable. For instance the absolute Galois group of $\mathbb{Q}$, or a finite field, is first-countable.
\end{example}

The second one is more involved. It relies on Krasner's Lemma \cite[8.1.6]{Neuk}: let $K$ be a non-Archimedean complete normed field, $\alpha \in K^{sep}$ with conjugates $\alpha = \alpha_1, \ldots, \alpha_d$. If $\beta \in K^{alg}$ satisfies $|\beta - \alpha| < |\alpha - \alpha_i|$ for all $i = 2, \ldots, d$, then $K(\alpha) \subset K(\beta)$.

\begin{example}
Let $K$ be a non-Archimedean complete normed field whose topology is separable. Then $K$ admits only countably many finite Galois extensions. This applies to all non-Archimedean local fields.
\begin{proof}
By hypothesis there exists a countable dense subset $D \subset K$, which we may assume is a field. We claim that $K^{sep} = \bigcup_{\beta \in D^{sep}} K(\beta)$. Since $D^{sep}$ is countable, this allows to conclude thanks to Lemma \ref{lem:gal_firstcount}. So let $\alpha \in K^{sep}$ and let $f(X) \in K[X]$ be the minimal polynomial of $\alpha$, whose roots $\alpha = \alpha_1, \ldots, \alpha_d$ are the Galois conjugates of $\alpha$, which are all distinct since $\alpha$ is separable. We can choose a polynomial $g(X) \in D[X]$ which is arbitrarily close to $f(X)$, in fact we can choose $g$ so that $|g(\alpha_j)| = |g(\alpha_j) - f(\alpha_j)|$ is arbitrarily small, for all $j$. Now write $g(X) = \prod (X - \beta_i)$, where $\beta_i \in K^{alg}$: this implies that for all $j$ there exists $i$ such that $|\alpha_j - \beta_i|$ is small, say smaller than $|\alpha_j - \alpha_k|$ for all $j \neq k$. Since $g$ has the same degree as $f$, and the $\alpha_j$ are all different, the association $\alpha_j \leftrightarrow \beta_i$ is a bijection. This implies in particular that all $\beta_i$ are distinct, so $g$ is separable. Moreover, if $\beta$ is the root close to $\alpha$, we have $K(\alpha) \subset K(\beta)$ by Krasner's Lemma, and $\beta \in D^{sep}$.
\end{proof}
\end{example}

\subsubsection{Integral matrices}

Let $(\K, | \cdot |)$ be a non-Archimedean field with ring of integers $\oo$, maximal ideal $\pp$ and residue field $\kk$ of characteristic $p \geq 0$ (recall that if $\K$ is a non-Archimedean local field, then the residual characteristic is always positive). Then the matrix group $\GGL(\oo)$ comes equipped with the $\ell^\infty$-norm induced by the inclusion into $\MM_n(\K)$:
\[\| M \| \coloneqq \max \{ |M_{ij}| : 1 \leq i, j \leq n \}.\]
This norm induces a distance by $d(A, B) = \| A - B \|$, which we already mentioned in the introduction in the special case $\K = \Qp$. The following lemma gives some basic but useful facts about the norm and the distance, which will be used repeatedly throughout the rest of the paper, occasionally without mention.

\begin{lemma}
\label{lem:GL}

Let $\| \cdot \|$ and $d$ be as above. Then
\begin{enumerate}
\item $d$ is an ultrametric.
\item $\| \cdot \|$ is submultiplicative; that is $\| A B \| \leq \| A \| \cdot \| B \|$.
\item $\| \cdot \|$ is equal to the operator norm, where $\K^n$ is also equipped with the $\ell^\infty$-norm.
\item If $A \in \GGL_n(\oo)$, then $\| A \| = 1$.
\item If $A \in \GGL_n(\oo)$ and $\| A - B \| < 1$, then $B \in \GGL_n(\oo)$.
\item $\| \cdot \|$ is invariant under left or right multiplication by elements of $\GGL_n(\oo)$.
\end{enumerate}
\end{lemma}

\begin{proof}
$1.$ This follows directly from the fact that the norm on $\K$ induces an ultrametric. 

\medskip

$2.$ Let $A, B \in \MM_n(\K)$. Then
\begin{align*}
\| AB \| &= \max\limits_{1 \leq i, j \leq n} |(AB)_{ij}| = \max\limits_{1 \leq i, j \leq n} \left| \sum\limits_{k = 1}^n A_{ik} B_{kj} \right| \leq \max\limits_{1 \leq i, j, k \leq n} |A_{ik} B_{kj}| \\
&\leq \left( \max\limits_{1 \leq i, k \leq n} |A_{ik}| \right) \cdot \left( \max\limits_{1 \leq j, k \leq n} |B_{kj}| \right) = \|A\| \cdot \|B\|.
\end{align*}

$3.$ The submultiplicativity also holds for matrix-vector multiplication, with the same proof: $\|Ax\| \leq \| A \| \cdot \|x \|$ for all $x \in \K^n$, and so $\|A\|_{op} \leq \|A\|$. For the converse, suppose that $|A_{ij}|$ attains its maximum in the $i$-th column $A_i$. Then $\|A\|_{op} \geq \|A e_i\| = \|A_i\| = \|A\|$. 

\medskip

$4.$ Suppose that $A \in \GGL_n(\oo)$, then all coordinates are in $\mathfrak{o}$ which immediately implies $\|A\| \leq 1$. It implies moreover that $\det(A) \in \oo^\times$ so $|\det(A)| = 1$. Since $\det(A)$ is a polynomial on $A_{ij}$, this is not possible if $\|A\| < 1$. Note that the converse of $4.$ is not true: if $\| A \| = 1$, then $A$ is integral, but its determinant may lie in $\oo \, \backslash \, \oo^\times$, in which case $A^{-1}$ is not integral. 

\medskip

$5.$ Clearly $B \in \MM_n(\oo)$ so it suffices to show that $\det(B) \in \oo^\times$. This is because $A \equiv B \mod \pp$, so $\det(B \mod \pp) \neq 0 \in \kk$. 

\medskip

$6.$ Let $A \in \GGL_n(\oo)$ and $M \in \MM_n(\K)$. Then, using 2. and 4.:
\[\|M\| = \|A^{-1} A M \| \leq \|A^{-1}\| \cdot  \|AM\| = \|AM\| \leq \|A\| \cdot \|M\| = \|M\|.\]
Therefore $\|AM\| = \|M\|$. Similarly $\|MA\| = \|M\|$.
\end{proof}

The family of groups $\GGL_n(\oo)$ with the distance $d$ is thus an ultrametric family. We denote this family by $\GGL(\oo)$. A special case is the family $\GGL(\Zp)$ from the introduction. 

\medskip

In case $\K$ is a local field, then $\GGL(\oo)$ is a profinite family, and the ultrametric distance $d$ is also a special case of Example \ref{ex:prof}. Indeed, we can choose as a sequence $\ek \coloneqq |\uni|^k$ (where $\uni$ is a uniformiser), and $\GGL_n(\oo)_k$ to be the ball of radius $\ek$ around the identity. This can be explicitly identified as the \emph{congruence subgroup}:
\[\GGL_n(\oo)_k \coloneqq \{ I + \uni^k M : M \in \MM_n(\oo) \}.\]
The metric quotients are the finite matrix groups $\GGL_n(\oo / \pp^k)$. There is no restriction on the order of these groups: indeed every finite group embeds into $\GGL_n(\oo/\pp) = \GGL_n(\kk)$ for $n$ large enough. However there is a restriction on the order of the metric quotients of the \emph{principal congruence subgroup} $\GGL_n(\oo)_1$, something that will be crucial in Sections \ref{s:vpropi} and \ref{s:char0}.

\begin{lemma}
\label{lem:vprop}

The principal congruence subgroup $\GGL_n(\oo)_1$ is pro-$p$. More precisely, for all $k \geq 1$ there is an isomorphism:
\[\GGL_n(\oo)_k / \GGL_n(\oo)_{2k} \to \left( \MM_n( \ok ), + \right).\]
\end{lemma}

\begin{proof}
Define the map $\GGL_n(\oo)_k \to \MM_n(\ok) : (I + \uni^k M) \mapsto M \mod \uni^k$. This is a homomorphism:
\[(I + \uni^k M)(I + \uni^k N) = I + \uni^k(M + N + \uni^k MN),\]
and the kernel is precisely $\GGL_n(\oo)_{2k}$.
\end{proof}

\pagebreak

\section{Ultrametric stability}
\label{s:ultra}

Let $\G$ be an ultrametric family. This section is concerned with stability with respect to $\G$, without additional assumptions, the main goals being Theorems \ref{intro:thm:pw_un} and \ref{intro:thm:rf}. The main tool for the proof of Theorem \ref{intro:thm:pw_un} will be to rephrase stability in terms of presentations, following \cite{a:comm}: this is well-known in the general setting, but the ultrametric framework gives better quantitative estimates that we will use in the rest of the paper, so we go through the arguments in detail. The proof of Theorem \ref{intro:thm:rf} will be quite direct thanks to Lemma \ref{lem:mq}. We will end the section by giving complete solutions to uniform stability problems with respect to some of the families introduced in Section \ref{s:fam}.

\subsection{Lifting and inducing asymptotic homomorphisms}
\label{ss:free}

Let $\Gamma = \langle S \mid R \rangle$ be a presentation of $\Gamma$: for the moment we do not impose any finiteness condition. We define defects and distances for homomorphisms on the free group $F_S$ that are close to satisfying the relations; this is analogous to how one can look at homomorphisms of $\Gamma$ as homomorphisms defined of $F_S$ that satisfy the relations. The following definitions are due to Arzhantseva--P\u{a}unescu in the finitely presented case \cite{a:comm}, and to Becker--Lubotzky--Thom in the finitely generated case \cite{IRS}.

\begin{definition}
Given a map $\hf \colon F_S \to G \in \G$ we define the \emph{defect of $\hf$ at $r \in \normal{R}$} and the \emph{defect of $\hf$} to be
\[\defe_r(\hf) \coloneqq d_G(\hf(r), 1_G); \,\,\,\,\,\,\, \defe(\hf) \coloneqq \sup\limits_{r \in R} \defe_r(\hf).\]
Given two maps $\hf, \hff \colon F_S \to G \in \G$ we define their \emph{distance at $w \in F_S$} and their \emph{distance} to be
\[\dist_w(\hf, \hff) \coloneqq d_G(\hf(w), \hff(w)); \,\,\,\,\,\,\, \dist(\hf, \hff) \coloneqq \sup\limits_{s \in S} \dist_s(\hf, \hff).\]
\end{definition}

\begin{remark}
Note that $\defe(\hf) = 0$ if and only if $\hf$ descends to a homomorphism of $\Gamma$.
\end{remark}

We will show in Lemma \ref{lem:ultra_lift} the correspondence between these notions ``at the level of $F_S$'' and those ``at the level of $\Gamma$'' that we defined in the introduction. A good feature of these definitions is that they allow to give a unique quantity for the notions of defect (for finitely presented groups) and of distance (for finitely generated groups), even when dealing with pointwise asymptotic homomorphisms. For dealing with uniform asymptotic homomorphisms, one would instead have to define the defect with a supremum over \emph{all} relations $r \in \normal{R}$, and the distance with a supremum over \emph{all} words $w \in F_S$. It turns out that this is not necessary under the ultrametric assumption:

\begin{lemma}
\label{lem:ultra_free}

Let $\Gamma = \langle S \mid R \rangle$ and let $\hf, \hff \colon F_S \to G \in \G$ be homomorphisms. Then:
\begin{enumerate}
\item For every $r \in \normal{R}$ there exists a finite set $\{r_1, \ldots, r_k\} \subset R$ (independent of $\hf$) such that
\[\defe_r(\hf) \leq \max\limits_{i} \defe_{r_i}(\hf).\]
In particular $\sup\limits_{r \in \normal{R}} \defe_r(\hf) = \sup\limits_{r \in R} \defe_r(\hf) = \defe(\hf)$.
\item For every $w \in F_S$ there exists a finite set $\{ s_1, \ldots, s_k \} \subset S$ (independent of $\hf$ and $\hff$) such that
\[\dist_w(\hf, \hff) \leq \max\limits_i \dist_{s_i}(\hf, \hff).\]
In particular $\sup\limits_{w \in F_S} \dist_w(\hf, \hff) = \sup\limits_{s \in S} \dist_s(\hf, \hff) = \dist(\hf, \hff)$.
\end{enumerate}
\end{lemma}

\begin{proof}
$1.$ Let $r \in \normal{R}$, and write $r = (w_1 r_1^{\sigma_1} w_1^{-1}) \cdots (w_k r_k^{\sigma_k} w_k^{-1})$ for $w_i \in F_S, r_i \in R$ and $\sigma_i \in \{ \pm 1 \}$. Then:
\begin{align*}
\defe_r(\hf) &= d_G(\hf(r), 1_G) = d_G(\hf(w_1 r_1^{\sigma_1} w_1^{-1}) \cdots \hf(w_k r_k^{\sigma_k} w_k^{-1}), 1_G) \\
&\leq \max \{ d_G(\hf(w_2 r_2^{\sigma_2} w_2^{-1}) \cdots \hf(w_k r_k^{\sigma_k} w_k^{-1}), 1_G), d_G(\hf(w_1 r_1^{\sigma_1} w_1^{-1}), 1_G) \} \\
&\leq \cdots \leq \max\limits_{i} d_G(\hf(w_i r_i^{\sigma_i} w_i^{-1}), 1_G) = \max\limits_{i} d_G(\hf(r_i^{\sigma_i}), 1_G) = \max\limits_i \defe_{r_i}(\hf).
\end{align*}

$2.$ Let $w = s_1^{\sigma_1} \cdots s_k^{\sigma_k} \in F_S$. Then, similarly:
\begin{align*}
\dist_w(\hf, \hff) &= d_G(\hf(w), \hff(w)) = d_G(\hf(s_1^{\sigma_1}) \cdots \hf(s_k^{\sigma_k}) \hff(s_k^{\sigma_k})^{-1} \cdots \hff(s_1^{\sigma_1})^{-1}, 1_G) \\
&\leq \max\limits_{i} d_G(\hf(s_i^{\sigma_i}) \hff(s_i^{\sigma_i})^{-1}, 1_G) = \max\limits_{i} d_G(\hf(s_i^{\sigma_i}), \hff(s_i^{\sigma_i})) = \max_i \dist_{s_i}(\hf, \hff). \qedhere
\end{align*}
\end{proof}

Let us now make the connection between such maps and those defined at the level of $\Gamma$:

\begin{lemma}
\label{lem:ultra_lift}
Let $\Gamma = \langle S \mid R \rangle$ and denote by $F_S \to \Gamma : w \mapsto \overline{w}$ the projection.
\begin{enumerate}
\item Let $\f \colon \Gamma \to G \in \G$ be a map. Let $\hf \colon F_S \to G$ be the unique homomorphism coinciding with $\f$ on $S$. Then for every word $w \in F_S$ there exists a finite set $\{(g_1, h_1), \ldots, (g_k, h_k) \} \subset \Gamma^2$ (independent of $\f$) such that
\[d_G(\f(\overline{w}), \hf(w)) \leq \max\limits_i \defe_{g_i, h_i}(\f).\]
In particular, if $r \in \normal{R}$ is a relation, then $\defe_r(\hf) \leq \max\limits_i \defe_{g_i, h_i}(\f)$ and $\defe(\hf) \leq \defe(\f)$.
\item Let $\hf \colon F_S \to G \in \G$ be a homomorphism. Choose a (set-theoretic) section $\tau \colon \Gamma \to F_S$ and define $\f \coloneqq \hf \circ \tau$. Then for every $(g, h) \in \Gamma^2$ there exists a finite set $\{ r_1, \ldots, r_k \} \subset R$ (independent of $\hf$ but depending on $\tau$) such that
\[\defe_{g, h}(\f) \leq \max\limits_i \defe_{r_i}(\hf).\]
In particular $\defe(\f) \leq \defe(\hf)$.
\end{enumerate}
\end{lemma}

\begin{remark}
\label{rem:2unst}
In the proof that follows, the notation could be simplified by removing all indices $\sigma_i$, if we could reduce to the case in which $\f(s^{-1}) = \f(s)^{-1}$. However this assumption is equivalent to stability of the group $\mathbb{Z}/2 \mathbb{Z}$, which cannot be established in such great generality (see Example \ref{ex:2unst}).
\end{remark}

\begin{proof}[Proof of Lemma \ref{lem:ultra_lift}]
$1.$ Fix a word $w = s_1^{\sigma_1} \cdots s_k^{\sigma_k}$ written as a reduced product of elements of $S \sqcup S^{-1}$: here $s_i \in S$ and $\sigma_i \in \{ \pm 1 \}$. Let
\[\dd \coloneqq \max\limits_i d_G(\f(s_1^{\sigma_1} \cdots s_i^{\sigma_i}), \f(s_1^{\sigma_1} \cdots s_{i-1}^{\sigma_{i-1}}) \f(s_i)^{\sigma_i}).\]
Notice that this is bounded by finitely many terms of the form $\defe_{g, h}(\f)$ for some $(g, h) \in \Gamma^2$: here we are using that $d_G(\f(s^{-1}), \f(s)^{-1}) \leq \max\{ \defe_{s, s^{-1}}(\f), \defe_{1, 1}(\f) \}$. Now
\[d_G(\f(s_1^{\sigma_1}), \f(s_1)^{\sigma_1}) \leq \dd\]
and by induction
\[
d_G(\f(s_1^{\sigma_1} \cdots s_i^{\sigma_i}), \f(s_1)^{\sigma_1} \cdots \f(s_i)^{\sigma_i}) \leq \max
\begin{cases}
d_G(\f(s_1^{\sigma_1} \cdots s_i^{\sigma_i}), \f(s_1^{\sigma_1} \cdots s_{i-1}^{\sigma_{i-1}})\f(s_i)^{\sigma_i}) \\
d_G(\f(s_1^{\sigma_1} \cdots s_{i-1}^{\sigma_{i-1}}), \f(s_1)^{\sigma_1} \cdots \f(s_{i-1})^{\sigma_{i-1}})
\end{cases}
\leq \dd.
\]
Therefore
\[d_G(\f(\overline{w}), \hf(w)) = d_G(\f(s_1^{\sigma_1} \cdots s_k^{\sigma_k}), \f(s_1)^{\sigma_1} \cdots \f(s_k)^{\sigma_k}) \leq \delta.\]
The last statement follows by taking $w$ to be a relation, so $\f(\overline{w}) = \f(1_\Gamma)$ is close to $1_G$; more precisely $d_G(\f(1), 1) = \defe_{1, 1}(\f)$. 

\medskip

$2.$ Fix $g, h \in \Gamma$; and let $r \coloneqq \tau(gh) (\tau(g) \tau(h))^{-1} \in \normal{R}$. Then
\[\defe_{g, h}(\f) = d_G(\f(gh), \f(g)\f(h)) = d_G(\hf(\tau(gh)), \hf(\tau(g))\hf(\tau(h)) ) = d_G(\hf(r), 1_G) = \defe_r(\hf).\]
The result then follows from Item $1.$ of Lemma \ref{lem:ultra_free}.
\end{proof}

This implies the desired equivalent characterisation of stability:

\begin{proposition}
\label{prop:stab_equiv}

Let $\Gamma = \langle S \mid R \rangle$.
\begin{enumerate}
\item $\Gamma$ is pointwise $\G$-stable if and only if for every sequence of homomorphisms $(\hf_n \colon F_S \to G_n \in \G)_{n \geq 1}$ such that $\defe_r(\hf_n) \to 0$ for all $r \in R$, there exists a sequence of homomorphisms $(\hff_n \colon F_S \to G_n)_{n \geq 1}$ that descend to $\Gamma$ and such that $\dist_s(\hf_n, \hff_n) \to 0$ for all $s \in S$.
\item $\Gamma$ is uniformly $\G$-stable if and only if for every sequence of homomorphisms $(\hf_n \colon F_S \to G_n \in \G)_{n \geq 1}$ such that $\defe(\hf_n) \to 0$, there exists a sequence of homomorphisms $(\hff_n \colon F_S \to G_n)_{n \geq 1}$ that descend to $\Gamma$ and such that $\dist(\hf_n, \hff_n) \to 0$.
\end{enumerate}
\end{proposition}

\begin{proof}
Suppose that $\Gamma$ is pointwise $\G$-stable and let $\hf_n$ be a sequence as in the statement of $1.$ Item $2.$ of Lemma \ref{lem:ultra_lift} gives a sequence $(\f_n \colon \Gamma \to G_n)_{n \geq 1}$ such that $\defe_{g, h}(\f_n) \to 0$ for all $(g, h) \in \Gamma^2$. By the hypothesis of pointwise stability, the asymptotic homomorphism $(\f_n)_{n \geq 1}$ is pointwise asymptotically close to a sequence of homomorphisms $(\ff_n \colon \Gamma \to G_n)_{n \geq 1}$. This lifts to a sequence $(\hff_n \colon F_S \to G_n)_{n \geq 1}$ such that $\defe(\hff_n) = 0$ and $\dist_s(\hf_n, \hff_n) \to 0$ for all $s \in S$. 

The converse is similar, using Item $1.$ of Lemma \ref{lem:ultra_lift} instead, and the proof for the uniform case is the same.
\end{proof}

So one can take the notions of stability to be defined by maps at the level of the free group that almost descend to $\Gamma$. Then Proposition \ref{prop:stab_equiv} becomes less obvious than it looks, since it says that this notion does not depend on the choice of the presentation. For finitely presented groups, this fact can also be proved directly by noticing that this notion of stability of a presentation is not affected by Tietze transformations \cite[Section 3]{a:comm}. 

\medskip

One last thing we need to understand is when closeness on the generators implies closeness everywhere. The following lemma showcases how useful moving up and down from $\Gamma$ to the corresponding free group can be.

\begin{lemma}
\label{lem:ultra_close}

Let $\Gamma$ be generated by a set $S$, and consider two maps $\f, \ff \colon \Gamma \to G \in \G$. Then for all $g \in \Gamma$ there exist finite sets $\{ (x_1, y_1), \ldots, (x_k, y_k) \} \subset \Gamma^2$ and $\{ s_1, \ldots, s_k \} \subset S$ (independent of $\f, \ff$) such that
\[\dist_g(\f, \ff) \leq \max\limits_i \{ \defe_{x_i, y_i}(\f), \defe_{x_i, y_i}(\ff), \dist_{s_i}(\f, \ff) \}.\]
In particular
\[\dist(\f, \ff) \leq \max \{ \defe(\f), \defe(\ff), \sup\limits_{s \in S} \dist_s(\f, \ff) \}.\]
\end{lemma}

\begin{proof}
Let $\hf, \hff \colon F_S \to G$ be the homomorphisms obtained via Item $1.$ of Lemma \ref{lem:ultra_lift}. That statement allows us to work with $\hf, \hff$ instead, up to finitely many defects of $\f$ and $\ff$. Then Item $2.$ of Lemma \ref{lem:ultra_free} implies that the distance at $g$ is bounded in terms of the distance at finitely many generators.
\end{proof}

We can now use Lemmas \ref{lem:ultra_lift} and \ref{lem:ultra_close} to give a quantitative version of Item $2.$ of Proposition \ref{prop:stab_equiv}. Namely, we provide a quantitative characterisation of uniform stability, as in Lemma \ref{lem:quant}, in terms of presentations:

\begin{corollary}
\label{cor:quant}

Let $\Gamma = \langle S \mid R \rangle$. The following are equivalent:
\begin{enumerate}
\item $\Gamma$ is uniformly $\G$-stable.
\item For all $\ee > 0$ there exists $\dd > 0$ such that whenever $\hf \colon F_S \to G \in \G$ satisfies $\defe(\hf) \leq \dd$, there exists $\hff \colon F_S \to G$ that descends to $\Gamma$ and satisfies $\dist(\hf, \hff) \leq \ee$.
\end{enumerate}
Moreover, for each $\ee > 0$, let $D^{\G}_{\langle S \mid R \rangle}(\ee)$ be the maximal $\dd > 0$ such that Item $2.$ holds for $\ee$. Then
\[D_{\langle S \mid R \rangle}^{\G}(\ee) \geq D_{\Gamma}^{\G}(\ee) \geq \min\{ D_{\langle S \mid R \rangle}^{\G}(\ee), \ee\}.\]
\end{corollary}

In particular, the quality of the uniform stability estimate (optimal, linear, quadratic... see Definition \ref{def:quant}) can also be computed by means of the characterisation in Item $2.$ of Corollary \ref{cor:quant}. We will discuss this more in Remark \ref{rem:quant:fp}.

\begin{proof}
Suppose that $\Gamma$ is uniformly $\G$-stable. By Lemma \ref{lem:quant}, for every $\ee > 0$ there exists $\dd \coloneqq D_{\Gamma}^{\G}(\ee) > 0$ such that whenever $\f \colon \Gamma \to G \in \G$ satisfies $\defe(\f) \leq \dd$ there exists a homomorphism $\ff \colon \Gamma \to G$ such that $\dist(\f, \ff) \leq \ee$. Now let $\hf \colon F_S \to G \in \G$ be such that $\defe(\hf) \leq \dd$. Item $2.$ of Lemma \ref{lem:ultra_lift} gives a map $\f \colon \Gamma \to G$ such that $\defe(\f) \leq \defe(\hf) \leq \dd$. Therefore there exists a homomorphism $\ff \colon \Gamma \to G$ such that $\dist(\f, \ff) \leq \ee$. This lifts to a homomorphism $\hff \colon F_S \to G$ that by construction descends to $\Gamma$, and such that $\dist(\hf, \hff) \leq \dist(\f, \ff) \leq \ee$. This shows that $2.$ holds with $\dd = D_{\Gamma}^{\G}(\ee)$ for this $\ee$, so it gives the implication $1. \Rightarrow 2.$ and the inequality $D_{\langle S \mid R \rangle}^{\G}(\ee) \geq D_{\Gamma}^{\G}(\ee)$.

Now suppose that $2.$ holds. Let $\ee > 0$ and let $\dd \coloneqq \min\{ D_{\langle S \mid R \rangle}^{\G}(\ee), \ee \} > 0$. Let $\f \colon \Gamma \to G \in \G$ be such that $\defe(\f) \leq \dd$. Item $1.$ of Lemma \ref{lem:ultra_lift} gives a homomorphism $\hf \colon F_S \to G$ such that $\defe(\hf) \leq \defe(\f) \leq \dd$. Therefore there exists a homomorphism $\hff \colon F_S \to G$ which descends to $\ff \colon \Gamma \to G$ and satisfies $\dist(\hf, \hff) \leq \ee$. Now $\hf$ coincides with $\f$ on $S$, and similarly $\hff$ coincides with $\ff$ on $S$; therefore $\dist_s(\f, \ff) \leq \ee$ for all $s \in S$. It then follows from Lemma \ref{lem:ultra_close} that
\[\dist(\f, \ff) \leq \max \{ \defe(\f), \defe(\ff), \sup\limits_{s \in S} \dist_s(\f, \ff) \} \leq \max\{ \dd, 0, \ee \} = \ee,\]
since we chose $\dd \leq \ee$. This shows that the characterisation of uniform $\G$-stability from Lemma \ref{lem:quant} holds with $\dd = \min\{ D_{\langle S \mid R \rangle}^{\G}(\ee), \ee\}$ for this $\ee$, so it gives the implication $2. \Rightarrow 1.$ and the inequality $D^{\G}_{\Gamma}(\ee) \geq \min\{ D_{\langle S \mid R \rangle}^{\G}(\ee), \ee\}$.
\end{proof}

\subsection{Finiteness conditions}

Now we add finiteness conditions on the presentation. The following proposition gives general properties of asymptotic homomorphisms under such hypotheses:

\begin{proposition}
\label{prop:ultra_asy}

\begin{enumerate}
\item Suppose that $\Gamma$ is generated by the finite set $S$, and that $(\f_n, \ff_n \colon \Gamma \to G_n \in \G)_{n \geq 1}$ satisfy $\dist_s(\f_n, \ff_n) \to 0$ for all $s \in S$. If $\f_n, \ff_n$ are pointwise (respectively, uniform) asymptotic homomorphisms, then they are pointwise (respectively, uniformly) asymptotically close.
\item Suppose that $\Gamma$ is finitely presented. Then every pointwise asymptotic homomorphism is pointwise asymptotically close to a uniform asymptotic homomorphism.
\end{enumerate}
\end{proposition}

\begin{proof}
$1.$ This follows directly from Lemma \ref{lem:ultra_close}. 

\medskip

$2.$ Let $(\f_n \colon \Gamma \to G_n \in \G)_{n \geq 1}$ be a pointwise asymptotic homomorphism. Using Item $1.$ of Lemma \ref{lem:ultra_lift}, we can lift $\f_n$ to a sequence of homomorphisms $(\hf_n \colon F_S \to G_n)_{n \geq 1}$ such that $\defe_r(\hf_n)$ is bounded in terms of finitely many defects of $\f_n$, for every relator $r \in R$. Since $R$ is finite, the same holds for $\defe(\hf_n)$. Now $\hf_n$ induces a map $\ff_n$ on $\Gamma$ using Item $2.$ of Lemma \ref{lem:ultra_lift}, thus we obtain a sequence $(\ff_n \colon \Gamma \to G_n)_{n \geq 1}$ such that $\defe(\ff_n) \leq \defe(\hf_n)$. It follows that $\ff_n$ is a uniform asymptotic homomorphism. Moreover, $\ff_n$ coincides with $\f_n$ on $S$, and so $(\ff_n)_{\geq 1}$ is pointwise asymptotically close to $(\f_n)_{n \geq 1}$ by the previous item.
\end{proof}

Item $2.$ of Proposition \ref{prop:ultra_asy} does not say that pointwise asymptotic homomorphisms are automatically uniform: this is false in general, as the next example shows.

\begin{example}
Consider the map
\[\f_n \colon \mathbb{Z} \to \GGL_2(\Zp) : k \mapsto 
\begin{cases}
\begin{pmatrix}
1 & k \\
0 & 1
\end{pmatrix} & \text{if } |k| \leq n \\
I_2 & \text{otherwise.}
\end{cases}\]
Then $\defe(\f_n) = 1$ for all $n$, while $\defe_{g, h}(\f_n) \to 0$ for all $(g, h) \in \mathbb{Z}^2$.
\end{example}

We can now prove Theorem \ref{intro:thm:pw_un} from the introduction:

\begin{theorem}
\label{thm:pw_un}
Let $\Gamma$ be finitely generated and pointwise $\G$-stable. Then $\Gamma$ is uniformly $\G$-stable. If moreover $\Gamma$ is finitely presented, then the converse holds.
\end{theorem}

\begin{proof}
Let $(\f_n \colon \Gamma \to G_n \in \G)_{n \geq 1}$ be a uniform asymptotic homomorphism. Since $\Gamma$ is pointwise $\G$-stable, $(\f_n)_{n \geq 1}$ is pointwise asymptotically close to a sequence of homomorphisms $(\ff_n \colon \Gamma \to G_n)_{n \geq 1}$. By Item $1.$ of Proposition \ref{prop:ultra_asy}, since $(\f_n)_{n \geq 1}$ and $(\ff_n)_{n \geq 1}$ are both uniform asymptotic homomorphisms, they are uniformly asymptotically close. 

\medskip

Now suppose that $\Gamma$ is finitely presented and uniformly $\G$-stable. Let $(\f_n \colon \Gamma \to G_n \in \G)_{n \geq 1}$ be a pointwise asymptotic homomorphism. By Item $2.$ of Proposition \ref{prop:ultra_asy}, $(\varphi_n)_{n \geq 1}$ is pointwise asymptotically close to a uniform asymptotic homomorphism, which in turn is uniformly (thus pointwise) asymptotically close to a sequence of homomorphisms, by uniform $\G$-stability.
\end{proof}

\begin{remark}
We will see that there exist finitely generated groups that are uniformly but not pointwise $\GGL(\oo)$-stable (Proposition \ref{prop:sharp}), and that there exist countable groups that are pointwise but not uniformly $\GGL(\oo)$-stable, at least in characteristic $0$ (Proposition \ref{prop:countable:unst}).
\end{remark}

This theorem allows to unambiguously talk about \emph{$\G$-stability} for finitely presented groups, since pointwise and uniform stability are equivalent.

\begin{example}
\label{ex:free}

A free group of finite rank $F$ is $\G$-stable. This follows immediately by using the characterisation of pointwise stability in Proposition \ref{prop:stab_equiv}. Using the estimates in Lemma \ref{lem:ultra_close}, we see that moreover the estimate for uniform stability is optimal. In fact, if $\f \colon F \to G \in \G$ is such that $\defe(\f) \leq \ee$, then setting $\ff \colon F \to G$ to be the unique homomorphism that coincides with $\f$ on a free basis, we have $\dist(\f, \ff) \leq \ee$.
\end{example}

Although pointwise stability of free groups holds for every family $\G$, as we remarked in the introduction, uniform stability is peculiar to the ultrametric setting. Indeed, free groups are not uniformly $\G$-stable, for $\G = \{ (\U(n), \| \cdot \|) : n \geq 1 \}$, where $\| \cdot \|$ is any $\U(n)$-invariant norm on $\MM_n(\mathbb{C})$ \cite{Rolli, BOT}, or for $\G = \{ (S_n, d_H) : n \geq 1 \}$ \cite{BChap}.
Example \ref{ex:free} also has the following notable consequence:

\begin{remark}
\label{rem:quotients}

If there exists a finitely generated group that is not uniformly $\G$-stable, then Example \ref{ex:free} implies that quotients of uniformly $\G$-stable groups need not be uniformly $\G$-stable. This holds for instance for $\G = \GGL(\oo)$ by Theorem \ref{intro:thm:unst} (Theorem \ref{thm:fg:unst}). Given how the norm on $\GGL(\oo)$ coincides with the operator norm (Lemma \ref{lem:GL}), this is in contrast with the Archimedean case, where Ulam stability (i.e. uniform stability with respect to unitary matrices endowed with the operator norm) is preserved under taking quotients \cite[Lemma 2.2]{BOT}. Indeed, the proof of \cite[Lemma 2.2]{BOT} makes crucial use of the Archimedean property.
\end{remark}

Here is a very general example.

\begin{example}
\label{ex:discr:stab}

Let $\G$ be a discrete family (Example \ref{ex:discr}). Then every group is uniformly $\G$-stable: if $\f \colon \Gamma \to G \in \G$ satisfies $\defe(\f) < 1$, then $\f$ is already a homomorphism. Theorem \ref{thm:pw_un} then implies that all finitely presented groups are pointwise $\G$-stable.
\end{example}

\begin{remark}
Pointwise $\G$-stability need not hold for arbitrary finitely generated groups: if $\G$ is the discrete family of all finite groups and $\Gamma$ is LEF but not residually finite, then $\Gamma$ is not pointwise stable (Corollary \ref{cor:cex_stab}).
\end{remark}

Let us end this subsection with a remark about quantitative notions of stability for finitely presented groups, which will be useful in Section \ref{s:BC}:

\begin{remark}
\label{rem:quant:fp}

Let $\Gamma = \langle S \mid R \rangle$ be a finitely presented group. Then for a general (not necessarily ultrametric) family $\G$, the analogue of Proposition \ref{prop:stab_equiv} holds, and thus one can give a quantitative characterisation of pointwise stability. Namely, $\Gamma$ is pointwise $\G$-stable if and only if for all $\ee > 0$ there exists $\dd > 0$ such that whenever $\hf \colon F_S \to G \in \G$ satisfies $\defe(\hf) \leq \dd$, then there exists $\hff \colon F_S \to G$ that descends to $\Gamma$ and satisfies $\dist(\hf, \hff) \leq \ee$ \cite{a:comm}. Note that $\defe$ is defined as a supremum over $R$ and $\dist$ as a supremum over $S$, so both supremums are attained and finite, even if $\G$ is not necessarily ultrametric.

This characterisation allows to talk quantitatively about pointwise stability of finitely presented groups, an approach initiated in \cite{quant} (see also \cite{a:quant}). That is, defining the function $D_{\langle S \mid R\rangle}^{\G}(\ee)$ analogously to Definition \ref{def:quant} with the above characterisation of pointwise stability, one can talk about pointwise stability with a linear estimate, polynomial estimate, and so on. We use the subscript $\langle S \mid R \rangle$ instead of $\Gamma$ to emphasise that this quantity depends on the presentation of $\Gamma$.

In the ultrametric setting, $D_{\langle S \mid R\rangle}^{\G}(\ee)$ is precisely the function appearing in Corollary \ref{cor:quant}, which was about uniform stability. Thus, for finitely presented groups we obtained a quantitative version of Theorem \ref{thm:pw_un}: the quantitative notions of pointwise and uniform $\G$-stability also coincide, whenever $\G$ is an ultrametric family. For this reason, in this paper we limit ourselves to the quantitative study of uniform stability, which by this remark is strictly more general.

It is also worth pointing out how little $D_{\langle S \mid R\rangle}^{\G}(\ee)$ depends on the presentation, as shown in the estimates in Corollary \ref{cor:quant}. In other cases, the dependence on the presentation can be more important, and forces one to study the function $D_{\langle S \mid R\rangle}^{\G}(\ee)$ up to a suitable equivalence relation, which for instance does not distinguish between the identity and linear functions \cite[Proposition A.5]{quant}. In our setting, the dependence is so weak that we can still distinguish between optimal and linear stability estimates, and we will do so throughout this paper.
\end{remark}

\subsection{Homomorphisms to metric quotients}

We now move to the proof of Theorem \ref{intro:thm:rf}. The main tool is given by Lemma \ref{lem:mq}, which relates asymptotic homomorphisms to $\G$ with true homomorphisms to $\MQ(\G)$, the family of metric quotients of $\G$ and subgroups thereof (Definition \ref{def:mq}).

\begin{notation}
For the rest of this subsection, we fix an ultrametric family $\G$, and a class of groups $\C$ closed under taking subgroups such that all metric quotients of $\G$ are contained in $\C$. Since $\C$ is assumed to be closed under taking subgroups, we have $\MQ(\G) \subset \C$.
By Lemma \ref{lem:largest_rf}, there exists a largest residually-$\C$ quotient of $\Gamma$, which we denote by $\Omega$. That is, $\Omega \coloneqq \Gamma/K$ where $K$ is the intersection of all normal subgroups of $\Gamma$ whose quotient belongs to $\C$. For instance, if $\G = \GGL(\oo)$, then $\MQ(\G)$ is the class of all finite groups, and so $\Omega$ can be taken to be the largest residually finite quotient of $\Gamma$. We also denote by $\pi_\Omega \colon \Gamma \to \Omega$ the quotient map.
\end{notation}

We start by proving a consequence of Lemma \ref{lem:mq} which essentially gives one direction of Theorem \ref{intro:thm:rf}. Intuitively, it states that maps to $G$ with small defect almost factor through $\Omega$.

\begin{lemma}
\label{lem:factor_R}

Let $\f \colon \Gamma \to G \in \G$ be such that $\defe(\f) \leq \ee$. Then there exists a map $\overline{\f} \colon \Omega \to G$ such that $\defe(\overline{\f}) \leq \ee$ and $\dist(\f, \overline{\f} \circ \pi_\Omega) \leq \ee$.
\end{lemma}

\begin{proof}
By Item $1.$ of Lemma \ref{lem:mq} we can consider the homomorphisms $\f(\ee) \colon \Gamma \to G/G(\ee) \in \C$ induced by $\f$. By the universal property of $\Omega$ (Lemma \ref{lem:largest_rf}) $\f(\ee)$ factors through a homomorphism $\phi \colon \Omega \to G/G(\ee)$. Take $\overline{\f}$ to be a (set-theoretic) lift of $\phi$ to a map $\Omega \to G$. Then by construction $\phi \colon \Omega \xrightarrow{\overline{\f}} G \to G/G(\ee)$ is a homomorphism, so by Item $1.$ of Lemma \ref{lem:mq} again $\defe(\overline{\f}) \leq \ee$. Moreover, the induced homomorphisms $\f(\ee), (\overline{\f}(\ee) \circ \pi_\Omega) \colon \Gamma \to G/G(\ee)$ both coincide with $\phi \circ \pi_\Omega$, so by Item $2.$ of Lemma \ref{lem:mq} we have $\dist(\f, \overline{\f} \circ \pi_\Omega) \leq \ee$.
\end{proof}

We are now ready to prove Theorem \ref{intro:thm:rf} from the introduction:

\begin{theorem}
\label{thm:rf}

Let $\G, \Gamma, \Omega$ be as above. Then $\Gamma$ is uniformly $\G$-stable if and only if $\Omega$ is, and the estimates coincide. If $\Gamma$ is pointwise $\G$-stable, then so is $\Omega$.
\end{theorem}

\begin{proof}
We use the characterisation from Lemma \ref{lem:quant}: a group $\Gamma$ is uniformly $\G$-stable if and only if for all $\ee > 0$ there exists $\dd > 0$ such that whenever $\f \colon \Gamma \to G \in \G$ is such that $\defe(\f) \leq \dd$, there exists a homomorphism $\ff \colon \Gamma \to G$ such that $\dist(\f, \ff) \leq \ee$.

Suppose that $\Omega$ is uniformly $\G$-stable. Fix $\ee > 0$ and let $\dd > 0$ be such that the previous statement holds; we may assume that $\dd \leq \ee$. Let $\f \colon \Gamma \to G \in \G$ be a map such that $\defe(\f) \leq \dd$. By Lemma \ref{lem:factor_R} there exists a map $\overline{\f} \colon \Omega \to G$ such that $\defe(\overline{\f}) \leq \dd$ and $\dist(\f, \overline{\f} \circ \pi_\Omega) \leq \dd$. By the choice of $\dd$, there exists a homomorphism $\overline{\ff} \colon \Omega \to G$ such that $\dist(\overline{\f}, \overline{\ff}) \leq \ee$. Then $\ff \coloneqq \overline{\ff} \circ \pi_\Omega \colon \Gamma \to G$ is a homomorphism and
\[\dist(\f, \ff) \leq \max\{ \dist(\f, \overline{\f} \circ \pi_\Omega), \dist(\overline{\f} \circ \pi_\Omega, \overline{\ff} \circ \pi_\Omega) \} \leq \max\{ \dd, \ee \} \leq \ee.\]
Therefore $\Gamma$ is uniformly $\G$-stable, and the estimate for $\Gamma$ is at least as good as the estimate for $\Omega$. 

\medskip

Suppose that $\Gamma$ is pointwise $\G$-stable, and let $(\f_n \colon \Omega \to G_n \in \G)_{n \geq 1}$ be a pointwise asymptotic homomorphism. Then $(\f_n \circ \pi_\Omega \colon \Gamma \to G_n)_{n \geq 1}$ is also a pointwise asymptotic homomorphism: indeed $\defe_{g, h}(\f_n \circ \pi_\Omega) = \defe_{\pi_\Omega(g), \pi_\Omega(h)}(\f_n)$ for all $(g, h) \in \Gamma^2$. By pointwise stability of $\Gamma$ there exists a sequence of homomorphisms $(\widetilde{\ff}_n \colon \Gamma \to G_n)_{n \geq 1}$ that is pointwise asymptotically close to $(\f_n \circ \pi_\Omega)_{n \geq 1}$. Since $G$ is residually-$\C$, we have that $\widetilde{\ff}_n$ factors through $\Omega$, and so there exists a homomorphism $\ff_n \colon \Omega \to G$ such that $\widetilde{\ff}_n = \ff_n \circ \pi_\Omega$. Moreover, $\f_n$ and $\ff_n$ are pointwise asymptotically close: indeed $\dist_{\pi_\Omega(g)}(\f_n, \ff_n) = \dist_g(\f_n \circ \pi_\Omega, \widetilde{\ff}_n)$.

A similar argument shows that if $\Gamma$ is uniformly $\G$-stable, then so is $\Omega$. Going through the proof in quantitative terms, using Lemma \ref{lem:quant}, also proves that the estimate for $\Omega$ is at least as good as the estimate for $\Gamma$.
\end{proof}

\begin{example}
\label{ex:fq_free}

Let $\Gamma$ be a group without non-trivial quotients in $\MQ(\G)$. Then $\Gamma$ is uniformly $\G$-stable, with an optimal estimate. More precisely, if $\f \colon \Gamma \to G \in \G$ and $\defe(\f) \leq \ee$, then $\dist(\f, \mathbbm{1}) \leq \ee$. For instance, if $\Gamma$ is a simple group, then it can only be unstable if it belongs to $\MQ(\G)$.

If $\G$ is profinite, then $\MQ(\G)$ consists of finite groups, and so an infinite group without non-trivial finite quotients is uniformly $\G$-stable. Examples include Pride's group (Example \ref{ex:small_canc}), as well as certain finitely presented groups such as Higman's group \cite{Higman}, Thompson's groups \cite{CFP} or Burger--Mozes groups \cite{BM}. Recently, more examples have been found in \cite{isomh3} among discrete subgroups of $\mathrm{Isom}(\mathbb{H}^3)$.

If $\MQ(\G)$ consists of finite $\pi$-groups, for $\pi$ a set of primes, then this statement is a weaker version of Proposition \ref{prop:pifree}.
\end{example}


\begin{example}
\label{ex:sym0_ust}

Let $\G$ be a profinite family. The group $H^+$ from Example \ref{ex:sym0}, which has index $2$ in Houghton's group $H$, has $\mathbb{Z}$ as largest residually finite quotient, which is uniformly $\G$-stable with an optimal estimate, by Example \ref{ex:free}. Therefore $H^+$ is uniformly $\G$-stable, with an optimal estimate. Similarly the largest residually finite quotient of $H$ is $\mathbb{Z} \times \mathbb{Z}/2\mathbb{Z}$, which as we will see in Corollaries \ref{cor:vfree_p} and \ref{cor:vfree} is uniformly $\GGL(\oo)$-stable, whenever $\K$ has characteristic other than $2$ (the quality of the estimate will be optimal or linear depending on $\K$). Thus, in these cases, $H$ is uniformly $\GGL(\oo)$-stable.
\end{example}

\begin{example}
\label{ex:wr_ust}

Let $\G$ be a profinite family. Let $\Gamma, \Lambda$ be finitely generated, with $\Lambda$ infinite and residually finite. Then by Example \ref{ex:wr} the largest residually finite quotient of $\Gamma \wr \Lambda$ is $\Ab(\Gamma) \wr \Lambda$. In particular, if $\Gamma$ is perfect (that is, $\Ab(\Gamma) = \{ 1 \}$) and $\Lambda$ is uniformly $\G$-stable, then $\Gamma \wr \Lambda$ is uniformly $\G$-stable, with the same estimate. For example the lamplighter group $\Gamma \wr \mathbb{Z}$ is uniformly $\G$-stable, with an optimal estimate, for every non-abelian finite simple group $\Gamma$ (Example \ref{ex:free}). We will prove a stronger statement for virtually pro-$\pi$ families in Corollary \ref{cor:wr}.
\end{example}

Note that in the proof of Theorem \ref{thm:rf}, as well as that of Lemma \ref{lem:factor_R}, we only used that every homomorphism from $\Gamma$ to a residually-$\MQ(\G)$ group factors through $\Omega$. The same holds for every intermediate quotient, and so we obtain the following generalisation of Theorem \ref{thm:rf}:

\begin{corollary}
\label{cor:rf}

Let $K \leq \Gamma$ be a group that is contained in the kernel of the quotient $\Gamma \to \Omega$. Then $\Gamma$ is uniformly stable if and only if $\Gamma/K$ is, and the estimates coincide. If $\Gamma$ is pointwise $\G$-stable, then so is $\Gamma/K$.
\end{corollary}

Another interesting consequence of Theorem \ref{thm:rf} is that the equivalence of Theorem \ref{thm:pw_un} can also be extended to some infinitely presented residually-$\MQ(\G)$ groups. We will give explicit examples in Corollaries \ref{cor:rfBS_p} and \ref{cor:rfBS}.

\begin{corollary}
\label{cor:fgrf}

Let $\G, \C$ be as above. Let $\Gamma$ be a (finitely generated, residually-$\C$) group, that can be expressed as the largest residually-$\C$ quotient of some finitely presented group. Then $\Gamma$ is pointwise $\G$-stable if and only if it is uniformly $\G$-stable.
\end{corollary}

\begin{proof}
By Theorem \ref{thm:pw_un} we only need to show that uniform stability implies pointwise stability. Let $\widehat{\Gamma}$ be a finitely presented such that $\Gamma$ is the largest residually $\C$-quotient of $\widehat{\Gamma}$. If $\Gamma$ is uniformly stable, then $\widehat{\Gamma}$ is uniformly stable by Theorem \ref{thm:rf}; being finitely presented $\widehat{\Gamma}$ is also pointwise stable by Theorem \ref{thm:pw_un}, and so $\Gamma$ is pointwise stable again by Theorem \ref{thm:rf}.
\end{proof}

In case $\G$ is profinite and $\C$ is the class of all finite groups, it is tempting to conjecture that all finitely generated residually finite groups satisfy the hypotheses of Corollary \ref{cor:fgrf}, and so the equivalence of Theorem \ref{thm:pw_un} applies to all of them. This is not the case. Indeed, let $\Gamma$ be a finitely generated residually finite group. Suppose that we know:
\begin{enumerate}
\item Whenever $C$ is finitely presented and $C \to \Gamma$ is a surjective homomorphism, $C$ is \emph{large}, that is, it virtually surjects onto $F_2$.
\item There exist finite groups onto which $\Gamma$ does \emph{not} virtually surject.
\end{enumerate}
Then $\Gamma$ cannot be the largest residually finite quotient of a finitely presented group. The paper \cite{covers} contains many examples of such groups.

\begin{example}
A finitely generated group that surjects onto $\mathbb{Z}$ with locally finite kernel satisfies Property $1.$ above \cite[Post-Scriptum]{covers}. Thus the Lamplighter group $\mathbb{Z}/2\mathbb{Z} \wr \mathbb{Z}$ (which satisfies Property $2.$, for example because it is metabelian) cannot be the largest residually finite quotient of a finitely presented group.
\end{example}

\begin{example}
There are also torsion-free examples. The Basilica group, introduced in \cite{basilica}, is finitely generated, residually finite, torsion-free, and every proper quotient thereof is solvable, so it satisfies Property $2.$ above. It also satisfies Property $1.$ \cite[Section 2]{covers}, so it cannot be the largest residually finite quotient of a finitely presented group.
\end{example}

Compare this with \cite[Corollary 6.9]{limit}: every finitely generated residually free group is the largest residually free quotient of a finitely presented group.

Nevertheless, we still do not know if the equivalence of Theorem \ref{thm:pw_un} holds for all finitely generated residually finite groups (see Question \ref{q:fgrf:unnonpw}).

\subsection{Solution to some stability problems}
\label{ss:sol}

We present here the complete solution to two uniform stability problems, with respect to certain families introduced in Section \ref{s:fam}.

\begin{proposition}
\label{prop:aut_stab}

Let $\G$ be an ultrametric family with the following property: for every $G \in \G$ and every $\ee > 0$, the extension $1 \to G(\ee) \to G \to G/G(\ee) \to 1$ splits. Then all groups are uniformly $\G$-stable, with an optimal estimate. In particular, all groups are uniformly $T(\mathbf{R})$-stable and $\Aut(X^*_\bullet)$-stable.
\end{proposition}

\begin{proof}
We use the characterisation of uniform stability from Lemma \ref{lem:quant}. By Lemma \ref{lem:mq} a map $\f \colon \Gamma \to G$ of defect $\ee$ induces a homomorphism $\Gamma \to G/G(\ee)$. Composing the latter with a homomorphic section $G/G(\ee) \to G$, we obtain a homomorphism $\ff \colon \Gamma \to G$ such that $\dist(\f, \ff) \leq \ee = \defe(\f)$.

Both $T(\mathbf{R})$ and $\Aut(X^*_\bullet)$ satisfy the hypothesis. The metric quotients of $T_n(\mathbf{R})$ are isomorphic to $T_{n-k}(\mathbf{R})$ for some $1 \leq k \leq n$, and a homomorphic section of the quotient map is given by the inclusion in the upper-left corner. The metric quotients of $\Aut(X^*_n)$ are isomorphic to the group of automorphisms of words of a given finite length, and a homomorphic section of the quotient map is given by letting these elements act on the prefix of the appropriate length, and trivially on the rest of the word.
\end{proof}

\begin{remark}
We will see that there exist (finitely generated) groups that are not pointwise $\Aut(X^*_\bullet)$-stable (Corollary \ref{cor:cex_stab}).
\end{remark}

Our next goal is to prove a stability result for finite families $\G$. When studying uniform stability in the general setting, even looking at families consisting of a single group can be interesting. For instance, when $\G = \{  (\U(1), \| \cdot \|_{op}) \}$, then non-abelian free groups are not uniformly $\G$-stable \cite{Rolli}. However, in the ultrametric setting, such a situation cannot occur:

\begin{proposition}
\label{prop:unfin}

Let $\G$ be a finite profinite family. Then every finitely generated group is uniformly $\G$-stable.
\end{proposition}

This is a uniform version of \cite[Proposition 6]{a:quant}, in the ultrametric setting. We will later see that the hypothesis of finite generation is necessary (Remark \ref{rem:unfin}).

For the proof we need the following equivalent characterisation of uniform stability in terms of ultrafilters, which is a uniform version of \cite[Theorem 4.2]{a:comm} (see Lemma \ref{lem:stab_up}):

\begin{lemma}
\label{lem:stab_uf}

Let $\G$ be an ultrametric family and let $\Gamma = \langle S \mid R \rangle$ be a countable group. The following are equivalent:
\begin{enumerate}
\item $\Gamma$ is uniformly $\G$-stable.
\item For every free ultrafilter $\omega \subset \mathcal{P}(\mathbb{N})$ the following holds: for every sequence $(\hf_n \colon F_S \to G_n \in \G)_{n \geq 1}$ such that $\defe(\hf_n) \xrightarrow{n \to \omega} 0$, there exists a sequence $(\hff_n \colon F_S \to G_n)_{n \geq 1}$ of homomorphisms that descend to $\Gamma$ such that $\dist(\hf_n, \hff_n) \xrightarrow{n \to \omega} 0$.
\end{enumerate}
\end{lemma}

\begin{proof}
We use the characterisation of uniform $\G$-stability from Corollary \ref{cor:quant}. 

\medskip

$1. \Rightarrow 2.$ Fix a free ultrafilter $\omega$ and let $(\hf_n \colon F_S \to G_n \in \G)_{n \geq 1}$ be such that $\defe(\hf_n) \xrightarrow{n \to \omega} 0$. For all $\hf_n$, let $\hff_n \colon F_S \to G_n$ be a homomorphism that descends to $\Gamma$ and minimises $\dist(\f_n, \ff_n)$ up to $1/n$. We need to show that $\dist(\hf_n, \hff_n) \xrightarrow{n \to \omega} 0$, so let $\ee > 0$ and let $\dd > 0$ be as in Corollary \ref{cor:quant} for $\ee/2 > 0$: this means that if $\defe(\hf_n) \leq \dd$ then $\dist(\hf_n, \hff_n) \leq \ee/2 + 1/n$. Let $N \geq 2/\ee$. Then if $\defe(\hf_n) \leq \dd$ and $n \geq N$ we have $\dist(\hf_n, \hff_n) \leq \ee$. Therefore
\[\{ n \geq 1 : \dist(\hf_n, \hff_n) \leq \ee \} \supset \{ n \geq N \} \cap \{ n \geq 1 : \defe(\hf_n) \leq \dd \}.\]
The smaller set belongs to $\omega$ because $\defe(\hf_n) \xrightarrow{n \to \omega} 0$ and $\omega$ is free. Thus the larger set also belongs to $\omega$, and we conclude. 

\medskip

$2. \Rightarrow 1.$ Suppose that $\Gamma$ is not uniformly $\G$-stable. By Corollary \ref{cor:quant} there exists $\ee > 0$ and a sequence $(\hf_n \colon F_S \to G_n \in \G)_{n \geq 1}$ such that $\defe(\hf_n) \leq 1/n$ but for every sequence of homomorphisms $(\hff_n \colon F_S \to G_n)_{n \geq 1}$ descending to $\Gamma$, it holds $\dist(\hf_n, \hff_n) \geq \ee$. This implies that $\defe(\hf_n) \xrightarrow{n \to \omega} 0$ for every free ultrafilter $\omega$, while $\lim\limits_{n \to \omega} \dist(\hf_n, \hff_n) \geq \ee > 0$ for every sequence $(\hff_n)_{n \geq 1}$ of homomorphisms that descend to $\Gamma$. So Item $2.$ does not hold.
\end{proof}

\begin{proof}[Proof of Proposition \ref{prop:unfin}]
Let $\Gamma = \langle S \mid R \rangle$ be a finitely generated group. Fix a free ultrafilter $\omega \subset \mathcal{P}(\mathbb{N})$ and let $(\hf_n \colon F_S \to G_n \in \G)_{n \geq 1}$ be a sequence such that $\defe(\hf_n) \xrightarrow{n \to \omega} 0$. Since $\G$ is finite, up to restricting to a subset belonging to $\omega$ we may assume that $G_n = G$ is a fixed group for all $n \geq 1$. Since $G$ is a compact metric space, for all $s \in S$ the sequence $\hf_n(s)$ admits an $\omega$-limit, which we denote by $\hff(s) \in G$. Let $\hff \colon F_S \to G$ be the corresponding homomorphism. Then $\hff$ descends to $\Gamma$: indeed, for all $r \in R$ we have
\[d_G(\hff(r), 1_G) = \lim\limits_{n \to \omega} d_G(\hf_n(r), 1_G) = \lim\limits_{n \to \omega} \defe_r(\hf_n) = 0.\]
Moreover, by definition $\dist_s(\hf_n, \hff) \xrightarrow{n \to \omega} 0$, and so, since $S$ is finite, $\dist(\hf_n, \hff) \xrightarrow{n \to \omega} 0$. We conclude by Lemma \ref{lem:stab_uf}.
\end{proof}

\begin{remark}
\label{rem:galois:problem}

In a previous version of this paper, we ended this section with a complete solution to the uniform stability problem with respect to the family $\Gal(K)$ from Example \ref{ex:gal:prof}. We first showed that a group is uniformly $\Gal(K)$-stable if and only if it is uniformly $\{ \Gal(K^{sep}/K) \}$-stable, and then concluded via Proposition \ref{prop:unfin} that all finitely generated groups are uniformly $\Gal(K)$-stable. Unfortunately that proof had a gap, we thank the anonymous referee for identifying it.
\end{remark}

\pagebreak

\section{Ultrametric approximation and pointwise stability}
\label{s:approx}

This section constitutes an interlude, in that we leave the question of stability to focus on the related approximation problem. The main goal is to prove Theorem \ref{intro:thm:approx} for an ultrametric family $\G$. Combining this with Lemma \ref{lem:GR} will produce several counterexamples to pointwise stability. 

\medskip

Recall from Definition \ref{def:approx} that $\Gamma$ is \emph{$\G$-approximable} if there exists a \emph{$\G$-approximation}, namely an asymptotically injective pointwise asymptotic homomorphism $(\f_n \colon \Gamma \to G_n \in \G)_{n \geq 1}$. The following similar notion appears quite naturally in this context:

\begin{definition}
We say that a $\G$-approximation is \emph{uniform} if it is moreover a uniform asymptotic homomorphism. If $\Gamma$ admits a uniform $\G$-approximation, it is said to be \emph{uniformly $\G$-approximable}.
\end{definition}

The results from Section \ref{s:ultra} imply that this is not stronger for finitely presented groups:

\begin{lemma}
\label{lem:approx_fp}

Let $\G$ be an ultrametric family, $\Gamma$ a finitely presented group. If $\Gamma$ is $\G$-approximable, then it is uniformly $\G$-approximable.
\end{lemma}

\begin{proof}
By Item $2.$ of Proposition \ref{prop:ultra_asy}, every pointwise asymptotic homomorphism of $\Gamma$ is pointwise asymptotically close to a uniform one. Applying this to a $\G$-approximation gives a uniform asymptotic homomorphism that is still asymptotically injective: a uniform $\G$-approximation.
\end{proof}

\begin{remark}
Lemma \ref{lem:approx_fp} only holds with the ultrametric assumption. For example, let $\G$ be the family of finite-dimensional unitary groups equipped with the operator norm. Then every amenable group is $\G$-approximable \cite{amenableMF}; however there exist finitely presented amenable groups which are not uniformly $\G$-approximable. Indeed, every amenable group is uniformly $\G$-stable \cite{amenst}. But if a finitely generated group is uniformly $\G$-approximable and uniformly $\G$-stable, then it is automatically residually finite, by the same argument as in Lemma \ref{lem:GR}. Therefore a finitely presented amenable group which is not residually finite (see e.g. \cite{fpamenrf}) cannot be uniformly $\G$-approximable. See \cite[Section 7]{Bharatandi} for more examples of groups that are not uniformly $\G$-approximable, for more general families $\G$.
\end{remark}

\subsection{From approximations to local embeddings}

Well-studied approximation properties such as soficity or hyperlinearity are much weaker than residual finiteness, or local embeddability into finite groups. In this subsection we prove that $\G$-approximability, when $\G$ is a profinite family, is stronger. This is essentially a reinterpretation of the interplay between local embeddability and convergence in the space of marked groups (Theorem \ref{thm:mg}). Recall that $\MQ(\G)$ denotes the class of metric quotients of $\G$ and their subgroups (Definition \ref{def:mq}).

\begin{proposition}
\label{prop:approx_1}

Let $\Gamma = \langle S \mid R \rangle$ be a countable group. If $\Gamma$ is $\G$-approximable, then $\Gamma$ is locally embeddable into $\MQ(\G)$. If $\Gamma$ is uniformly $\G$-approximable, then $\Gamma$ is fully residually-$\MQ(\G)$. In particular, if $\Gamma$ is $\G$-approximable and finitely presented, then $\Gamma$ is fully residually-$\MQ(\G)$.
\end{proposition}

\begin{remark}
The last statement follows form the general fact that finitely presented groups that are locally embeddable into $\MQ(\G)$ are also fully residually-$\MQ(\G)$ (Proposition \ref{prop:lec_rc}). However it also follows by combining the rest of the proposition with Lemma \ref{lem:approx_fp}.
\end{remark}

\begin{proof}
Let $(\f_n \colon \Gamma \to G_n \in \G)_{n \geq 1}$ be a pointwise asymptotic homomorphism, which we lift to a sequence of homomorphisms $(\hf_n \colon F_S \to G_n)_{n \geq 1}$ using Lemma \ref{lem:ultra_lift}. Fix an enumeration of $N = \normal{R}$, denote by $N(k)$ the first $k$ elements, and fix a strictly decreasing sequence $\ek \to 0$. Up to subsequence, we may assume that $\defe_r(\hf_n) \leq \en$ for all $r \in N(n)$, and we look at the induced homomorphism $f_n \colon F_S \to G/G(\en)$. We get a sequence of $F_S$-marked groups with kernel $N_n = \{ w \in F_S : d(\hat{\varphi}_n(w), 1) \leq \en \}$. Up to subsequence, this converges in $\mathcal{N}(F_S)$ to:
\begin{align*}
&\{ w \in F_S : d(\hf_n(w), 1) \leq \en \text{ for infinitely many } n \} \\
= \; & \{ w \in F_S : d(\hf_n(w), 1) \leq \en \text{ for all but finitely many } n \}.
\end{align*}
Now $N = \normal{R}$ is contained in the left-hand side by choice of the subsequence. If moreover $\f_n$ is asymptotically injective, then $N$ contains the right-hand side, and so $N_n \to N \in \mathcal{N}(F_S)$, which implies that $\Gamma$ is locally embeddable into $\MQ(\G)$ by Item $2.$ of Theorem \ref{thm:mg}. 

\medskip

Now assume that $(\varphi_n \colon \Gamma \to G_n)_{n \geq 1}$ is a uniform asymptotic homomorphism. We apply a similar argument over the sequence of $\Gamma$-marked groups with kernels $N_n \coloneqq \{ g \in \Gamma : d(\varphi_n(g), 1) \leq \defe(\varphi_n) \}$; again, all quotients belong to $\MQ(\G)$. Assuming that $\varphi_n$ is asymptotically injective, we deduce that a subsequence of $N_n$ converges to $\{ 1 \} \in \mathcal{N}(\Gamma)$, and so $\Gamma$ is fully residually $\MQ(\G)$ by Item $1.$ of Theorem \ref{thm:mg}.
\end{proof}

\begin{example}
A group is called \emph{weakly hyperlinear} if it is approximable with respect to some family of compact metric groups (see \cite{a:quant}, where this notion is attributed to Gismatullin). Similarly, a group is called \emph{weakly sofic} if it is approximable with respect to some family of finite metric groups \cite{weakly}. Proposition \ref{prop:approx_1} shows that, if we add the hypothesis that the approximating families are ultrametric, then all such groups are LEF.
\end{example}

According to the properties of the class $\MQ(\G)$, the conclusion of Proposition \ref{prop:approx_1} can be strengthened.
We look at two examples: the family $T(\mathbf{R})$ (Example \ref{ex:UT}) where $\mathbf{R}$ is a finite ring, and the family $\Gal(\mathbb{F})$ (Example \ref{ex:gal:prof}), where $\mathbb{F}$ is a finite field.

\begin{corollary}
\label{cor:approxTR}

Let $\mathbf{R}$ be a commutative unital ring, and $\Gamma$ a countable $T(\mathbf{R})$-approximable group. Then there exists a normal subgroup $\Gamma_0 \leq \Gamma$ such that $\Gamma/\Gamma_0$ embeds into $(\mathbf{R}^\times)^{\mathbb{N}}$ and $\Gamma_0$ is locally embeddable into $UT(\mathbf{R})$, in particular $\Gamma_0$ is locally embeddable into the class of nilpotent groups.

If moreover $\mathbf{R}$ is finite and $\Ab(\Gamma)$ is finitely generated, then $\Gamma$ and $\Gamma_0$ have non-trivial abelian quotients, and $\Gamma/\Gamma_0$ embeds into $(\mathbf{R}^\times)^n$ for some $n \geq 1$.
\end{corollary}

\begin{proof}
We will use the equivalent characterisation of local embeddability in terms of ultraproducts (Proposition \ref{prop:lef_up}) repeatedly throughout the proof.

By Proposition \ref{prop:approx_1}, and since all metric quotients of $T_n(\mathbf{R})$ are of the form $T_k(\mathbf{R})$ (Example \ref{ex:UT}), we know that $\Gamma$ is locally embeddable into $T(\mathbf{R})$. Then $\Gamma$ embeds into an ultraproduct $\prod\limits_{n \to \omega} T_n(\mathbf{R})$. This gives a homomorphism
\[\Gamma \to \prod\limits_{n \to \omega} T_n(\mathbf{R}) \to \prod\limits_{n \to \omega} \Ab(T_n(\mathbf{R})) \cong \prod\limits_{n \to \omega} (\mathbf{R}^\times)^n,\]
let $\Gamma_0$ be its kernel. Then $\Gamma_0$ embeds into $\prod\limits_{n \to \omega} UT_n(\mathbf{R})$, and so it is locally embeddable into $UT(\mathbf{R})$.

Now $\Gamma/\Gamma_0$ embeds into $\prod\limits_{n \to \omega} (\mathbf{R}^\times)^n$, so it is locally embeddable into $\C \coloneqq \{ (\mathbf{R}^\times)^n : n \geq 1 \}$. Since $\Gamma/\Gamma_0$ is abelian, every finitely generated subgroup of it is finitely presented and locally embeddable into $\C$, so residually-$\C$ by Item $2.$ of Proposition \ref{prop:lec_rc}, and so it embeds into $(\mathbf{R}^\times)^{\mathbb{N}}$. Since $\Gamma$ is countable, $\Gamma/\Gamma_0$ embeds into $(\mathbf{R}^\times)^{\mathbb{N}}$, too. If now $\mathbf{R}^\times$ is finite, then $\mathbb{Z}$ cannot be residually-$\C$, since there is a bound on the order of cyclic subgroups of $(\mathbf{R}^\times)^n$. So $\Gamma/\Gamma_0$ is a torsion group; if moreover $\Ab(\Gamma)$ is finitely generated, then $\Gamma/\Gamma_0$ is finite, and being residually-$\C$ it embeds into $(\mathbf{R}^\times)^n$ for some $n \geq 1$. The statement about $\Gamma$ and $\Gamma_0$ having non-trivial abelian quotients is a consequence of \cite{Sol}, where it is proved that this holds for all finitely generated groups that are approximable in the class of finite solvable groups.
\end{proof}

For the family $\Gal(\mathbb{F})$, we can give a full characterisation.

\begin{corollary}
\label{cor:galfin_approx_1}

Let $\mathbb{F}$ be a finite field, and $\Gamma$ a finitely generated group. Then $\Gamma$ is $\Gal(\mathbb{F})$-approximable if and only if it is of the form $\mathbb{Z}^r \times C$ for some finite cyclic group $C$.
\end{corollary}

\begin{proof}
The absolute Galois group $\Gal(\mathbb{F}^{sep}/\mathbb{F})$ is isomorphic to the profinite completion $\widehat{\mathbb{Z}}$ of $\mathbb{Z}$, which in turn is isomorphic to the direct product of $\Zp$, where $p$ goes through all primes \cite[VII.5]{Neuk}. It follows that every Galois extension of $\mathbb{F}$ has pro-cyclic Galois group, so by Proposition \ref{prop:approx_1}, if $\Gamma$ is $\Gal(\mathbb{F})$-approximable, then it locally embeddable in the class of finite cyclic groups. This implies two things: first, all finite subgroups of $\Gamma$ are cyclic, since a finite group that is locally embeddable in a class automatically belongs to that class. Secondly, $\Gamma$ is abelian; more generally, a group that is locally embeddable in the class of abelian groups is abelian: this follows directly from the ultraproduct characterisation (Proposition \ref{prop:lef_up}). 

\medskip

We next show that $\Gamma = \mathbb{Z}^r \times C$ is approximable. We construct the approximations inside the absolute Galois group of $\mathbb{F}$, which we identify with $\widehat{\mathbb{Z}} \cong \prod \Zp$. This group is metrised using a nested sequence $G_k$ of open finite-index normal subgroups and a sequence of positive reals $\ek \to 0$, as in Example \ref{ex:gal:prof}. Let us make another reduction: if we are able to approximate $C = \mathbb{Z}/p^n\mathbb{Z}$ via maps that take values in $\Zp \leq \widehat{\mathbb{Z}}$, then we are done. Indeed, we can write every finite cyclic group as a direct product of such groups, for a finite set $\{p_1, \ldots, p_i\}$ of distinct primes, and take the direct product of these approximations into $\prod_i \mathbb{Z}_{p_i} \leq \widehat{\mathbb{Z}}$. Finally, we can embed $\mathbb{Z}^r$ into a direct product of $\Zp$ for $r$ distinct primes that we did not use yet. 

So we are left to show that we can approximate $C = \mathbb{Z}/p^n\mathbb{Z}$ for a given $n \geq 1$ with an approximation taking values in $\Zp \leq \widehat{\mathbb{Z}}$. Given $k \geq 1$ let $m \geq 1$ be such that $\Zp \cap G_k = p^m \Zp$; since the $G_k$ are ordered by reverse inclusion, $m \to \infty$ as $k \to \infty$. Let $k(n)$ be the smallest integer such that the corresponding $m$ is larger than $n$, and let $k \geq k(n)$. Then we can embed $C$ into $\mathbb{Z}/p^m \mathbb{Z}$, compose with a section into $\Zp$, and finally include $\Zp$ in $\widehat{\mathbb{Z}}$. This gives a map $\f \colon C \to \widehat{\mathbb{Z}}$ that projects to an injective homomorphism into $\widehat{\mathbb{Z}}/G_k$. It follows from the definition of the metric (Example \ref{ex:gal:prof}) and Lemma \ref{lem:mq} that $\defe(\f) \leq \ek$, and moreover $d(\f(x), 1) \geq \ee_{k(n)}$ for all $x \neq 1$.
\end{proof}

\begin{remark}
Note that the proof actually shows that these groups are $\{ \Gal(\mathbb{F}^{sep}/\mathbb{F}) \}$-approximable. In the finitely generated case, the approximations can be taken uniform by Lemma \ref{lem:approx_fp}.
\end{remark}

\subsection{From local embeddings to approximations}

The first easy examples of sofic and hyperlinear groups are residually finite, and more generally LEF groups. In this subsection we prove that these are also approximable with respect to certain ultrametric families. The precise statement will involve the following concept:

\begin{definition}
\label{def:capprox}

Let $\C$ be a class of groups. We say that $\C$ is \emph{$\G$-approximable} if there exists $\ee > 0$ such that for all $C \in \C$ and for all $\dd > 0$ there exists a map $\eta \colon C \to G \in \G$ such that $\defe(\eta) \leq \dd$ and $d_G(\eta(x), 1_G) > \ee$ for all $1 \neq x \in C$.
\end{definition}

So a class $\C$ is $\G$-approximable if every group in $C$ is uniformly $\G$-approximable, and moreover the injectivity gap can be taken uniformly for the entire class. For several profinite families $\G$, the class of all finite groups is $\G$-approximable:

\begin{example}
\label{ex:approx_gl}

The class of all finite groups is $\GGL(\oo)$-approximable, where $\oo$ is the ring of integers of a non-Archimedean local field. Indeed every finite group $C$ can be embedded into $\GGL_n(\oo)$, for $n$ large enough, using permutation matrices. If $\eta$ denotes this embedding, then $\defe(\eta) = 0$ and $d(\eta(x), I_n) = 1$ for all $1 \neq x \in C$.
\end{example}

\begin{example}
\label{ex:approx_aut}

The class of all finite groups is $\Aut(X_\bullet^*)$-approximable. Indeed, every finite group $C$ can be embedded into $S_n$, for $n$ large enough, which in turn can be embedded into $\Aut(X_n^*)$ by acting on the first letter of a word. If $\eta$ denotes this embedding, then $\defe(\eta) = 0$ and $d(\eta(x), id_{X_n^*}) = 1$ for all $1 \neq x \in C$.
\end{example}

For other families, this property can be quite restrictive:

\begin{example}
\label{ex:galfin_approx_2}

Let $K$ be a field whose absolute Galois group is first countable, and consider the family $\Gal(K) = \Gal(K)((G_k)_{k \geq 1}, \eee)$ of Galois groups of Galois extensions of $K$: see Example \ref{ex:gal:prof} for the notation. We claim that an infinite class of finite groups cannot be $\Gal(K)$-approximable.

Suppose that $\C$ is a $\Gal(K)$-approximable class of finite groups. Let $\ee > 0$ be the uniform injectivity gap as in the definition, we may assume that $\ee = \ee_k$ for some $k \geq 1$. Then for all $C \in \C$ there exists $G/N \in \Gal(K)$ and a map $\eta \colon C \to G/N$ such that $\defe(\eta) \leq \ee_k$ and $d(\eta(x), 1) > \ee_k$ for all $1 \neq x \in C$. The first inequality implies that $\eta$ induces a homomorphism $C \to (G/N) / (G_k N / N) \cong G / G_k N$, and the second one implies that this homomorphism is injective. Therefore $|C| \leq [G : G_k N] \leq [G : G_k]$. Since the inequality holds for every $C$, this implies that $\C$ is finite.
\end{example}

By Proposition \ref{prop:approx_1}, in order to prove that all finite groups are $\G$-approximable, some restriction on $\G$ is necessary. For instance if a finite group does not belong to $\MQ(\G)$, then it cannot be $\G$-approximable. When the local embeddings take place in a $\G$-approximable class, then we can prove a converse to Proposition \ref{prop:approx_1}:

\begin{proposition}
\label{prop:approx_2}

Let $\C$ be a $\G$-approximable class of groups and $\Gamma$ a countable group. If $\Gamma$ is locally embeddable into $\C$, then $\Gamma$ is $\G$-approximable. If $\Gamma$ is fully residually-$\C$, then $\Gamma$ is uniformly $\G$-approximable.
\end{proposition}

\begin{proof}
Since $\Gamma$ is countable, we can write it as an increasing union of finite sets $(K_n)_{n \geq 1}$. If $\Gamma$ is locally embeddable into $\C$, then for all $n$ we can choose a $K_n$-local embedding $f_n \colon \Gamma \to C_n \in \C$. Since $\C$ is $\G$-approximable, there exists $\ee > 0$ such that: for all $n$ there exists a map $\eta_n \colon C_n \to G_n \in \G$ with $\defe(\eta_n) \leq 1/n$ and $d_n(\eta_n(x), 1_{G_n}) > \ee$ for all $1 \neq x \in C_n$. Then $\f_n \colon \Gamma \xrightarrow{f_n} C_n \xrightarrow{\eta_n} G_n$ satisfies: $\defe_{g, h}(\f_n) \leq 1/n$ for all $(g, h)^2 \in K_n^2$, and $d_n(\f_n(g), 1_{G_n}) > \ee$ for all $1 \neq g \in K_n$. It follows that $(\f_n)$ is a $\G$-approximation.

If $\Gamma$ is fully residually-$\C$, then the $K_n$-local embeddings $f_n \colon \Gamma \to C_n$ may be chosen to be homomorphisms that restrict to injective maps on $K_n$. Then the resulting $\G$-approximation is uniform.
\end{proof}

Before moving on to the main application, namely examples of groups that are not pointwise stable, we make a few remarks about the previous results.

\begin{remark}
The requirement that the class $\C$ be $\G$-approximable allows to deduce approximability from the existence of local embeddings of $\Gamma$ into $\C$, without requiring any knowledge of the local embeddings themselves. But this is not the only way to produce approximations. For instance, by Example \ref{ex:galfin_approx_2}, when $\G = \Gal(K)$ we can only apply Proposition \ref{prop:approx_2} to groups that are locally embeddable into some finite class $\C$ of finite groups, which are necessarily finite. But in Corollary \ref{cor:galfin_approx_1} we saw that a $\Gal(\mathbb{F})$-approximable group can be infinite.
\end{remark}

\begin{remark}
The proof of Proposition \ref{prop:approx_2} implies something (a priori) slightly stronger than approximation: namely, such groups are $\G$-approximable with a uniform injectivity gap:
\[\inf\limits_{g \in \Gamma} (\liminf\limits_{n \to \infty} d_n(\f_n(g), 1_{G_n})) > 0.\]
This is taken to be the definition of $\G$-approximation in some of the literature (see e.g. \cite{a:quant}). This ambiguity is due to the fact that in several approximation problems the two notions coincide: for instance, for sofic groups, this follows from a well-known \emph{amplification trick} due to Elek and Szab\'o \cite{ampl}, while in other context it may be harder to prove \cite{a:lin, a:lin:gap}.
\end{remark}

\begin{remark}
Putting together Proposition \ref{prop:approx_1}, Examples \ref{ex:approx_gl} and \ref{ex:approx_aut}, and Proposition \ref{prop:approx_2}, we obtain that a group is LEF if and only if it is $\GGL(\oo)$-approximable (Theorem \ref{intro:thm:approx}), or $\Aut(X^*_\bullet)$-approximable. In other words, the notions of approximability with respect to these two families coincides with that of approximability with respect to the family of all finite groups equipped with the discrete metric (see Example \ref{ex:discr}). However the \emph{quantitative} versions of approximability are potentially distinct: for instance to embed a finite group into $\GGL_n(\oo)$ one does not always need the degree $n$ to be equal to the order. The quantitative study of approximation properties was initiated in \cite{a:quant}; the case of local embeddings into finite groups was recently studied in more detail in \cite{Bradford}.
\end{remark}

\subsection{Examples of pointwise instability}
\label{ss:unst}

We now apply the results in this section to give counterexamples to pointwise stability. These all stem from the following corollary, which applies to both $\G = \GGL(\oo)$ and $\G = \Aut(X^*_\bullet)$.

\begin{corollary}
\label{cor:cex_stab}

Let $\G$ be a profinite family, such that the class of finite groups is $\G$-approximable. Then if $\Gamma$ is LEF but not residually finite, it is not pointwise $\G$-stable.
\end{corollary}

\begin{proof}
This is just a combination of Lemma \ref{lem:GR} and Proposition \ref{prop:approx_2}.
\end{proof}

It is worth noticing that there is no hypothesis of finite generation in this statement, in contrast to the analogous statement for families of unitary groups. This is because the groups $\GGL_n(\oo)$ are not only locally residually finite, as guaranteed for all linear groups by a theorem of Malcev \cite{Malcev}, but they are themselves residually finite, being profinite. 

\medskip

In the examples below, $\GGL(\oo)$ may be replaced with any profinite family satisfying the hypotheses of the corollary, for instance $\Aut(X^*_\bullet)$, or the discrete family of all finite groups (Example \ref{ex:discr}).

\begin{example}
\label{ex:cex}

Let $\Gamma$ be a classical small cancellation group that is not residually finite, for instance Pride's group (Example \ref{ex:small_canc}). Then $\Gamma$ is not pointwise $\GGL(\oo)$-stable. So even the simplest example of a finitely generated uniformly stable group -- a group without non-trivial finite quotients (see Example \ref{ex:fq_free}) -- need not be pointwise $\GGL(\oo)$-stable.
\end{example}

\begin{example}
\label{ex:cex2}

Let $H$ be Houghton's group and $H^+$ its index-$2$ subgroup (see Example \ref{ex:sym0}). These groups are finitely generated, and the largest residually finite quotient is $\mathbb{Z} \times \mathbb{Z}/2 \mathbb{Z}$ - for $H$ - or $\mathbb{Z}$ - for $H^+$. In particular neither $H$ not $H^+$ is residually finite. But by Example \ref{ex:sym0_lef} both $H$ and $H^+$ are LEF, therefore they are not pointwise $\GGL(\oo)$-stable by Corollary \ref{cor:cex_stab}. On the other hand, by Example \ref{ex:sym0_ust} the group $H^+$ is uniformly $\GGL(\oo)$-stable (and when $\K$ does not have characteristic $2$, so is $H$, by Example \ref{ex:pfree} and Proposition \ref{prop:vpfree}). The same result holds for lamplighter groups by Examples \ref{ex:wr}, \ref{ex:wr_lef} and \ref{ex:wr_ust}, for instance if $\Gamma$ is a non-abelian finite simple group, then $\Gamma \wr \mathbb{Z}$ is finitely generated, LEF, but not residually finite, so it is not pointwise $\GGL(\oo)$-stable, even though it is uniformly $\GGL(\oo)$-stable.
\end{example}

These examples show that two of our main results are sharp (this is part of Proposition \ref{intro:prop:sharp} from the introduction):

\begin{proposition}
\label{prop:sharp}

Let $\G$ be a profinite family. Then Theorems \ref{thm:pw_un} and \ref{thm:rf} are sharp, in the following sense:
\begin{enumerate}
\item There exists a finitely generated group that is uniformly but not pointwise $\G$-stable.
\item There exists a finitely generated group that is not pointwise $\G$-stable, but whose largest residually finite quotient is.
\end{enumerate}
\end{proposition}

\pagebreak

\section{Virtually pro-$\pi$ stability}
\label{s:vpropi}

We now specialise our study of stability to profinite families $\G$ whose metric quotients have restricted orders. We will prove some general criteria, which will allow us to provide various examples of uniformly stable groups with respect to such families, some of which are listed in Theorem \ref{intro:thm:pifree}.

The basic idea of this approach can be traced back to \cite{simGLp}, where the author uses the conjugacy part of Hall's Theorem on solvable groups \cite[Theorem 5.28]{Rotman} to prove a kind of stability result for the conjugacy relation, under some coprimality assumption on the order of the elements. Here we will go much further, and this is made possible by the use of the more general Schur--Zassenhaus Theorem (Theorem \ref{thm:SZ}).

\begin{definition}
\label{def:vprop}

Let $\G$ be a profinite family. Given a class $\C$ of finite groups, we say that $\G$ is \emph{virtually pro-$\C$} if there exists some $\ee_0 > 0$ such that $G(\ee_0)$ is pro-$\C$ for all $G \in \G$. In this section we focus on the class $\C$ of $\pi$-groups, where $\pi$ is a fixed set of primes: we say that $\G$ is \emph{virtually pro-$\pi$}.
\end{definition}

The condition that $G(\ee_0)$ be pro-$\C$ is equivalent to all metric quotients of $G(\ee_0)$ being in $\C$. Notice that we are asking that the $\ee_0 > 0$ be uniform for the whole family.

\begin{example}
Let $\oo$ be the ring of integers of a non-Archimedean local field of residual characteristic $p$. Then $\GGL(\oo)$ is virtually pro-$p$: indeed by Lemma \ref{lem:vprop} the principal congruence subgroups $\GGL_n(\oo)_1$ are pro-$p$, and we can take $\ee_0 = |\uni|$, where $\uni \in \oo$ is a uniformiser.
\end{example}

The key to this approach is the interpretation of stability in terms of the following lifting property. A map $\f \colon \Gamma \to G$ with $\defe(\f) \leq \dd$ induces a homomorphism $\f(\dd) \colon \Gamma \to G/G(\dd)$, and thus also a homomorphism $\f(\ee) \colon \Gamma \to G/G(\ee)$ for all $\ee \geq \dd$. Then a homomorphism $\ff \colon \Gamma \to G$ satisfies $\dist(\f, \ff) \leq \ee$, where $\dd \leq \ee$, if and only if it is a lift of the induced homomorphism $\f(\ee) \colon \Gamma \to G/G(\ee)$. This is just a rephrasing of Lemma \ref{lem:mq}.

\subsection{$\pi$-free groups}
\label{ss:pifree}

In this subsection we use the lifting part of the Schur--Zassenhaus Theorem to prove stability with respect to virtually pro-$\pi$ families of groups whose finite quotients have restricted orders.

\begin{definition}
\label{def:pifree}

A group $\Gamma$ is \emph{$\pi$-free} if all of its finite quotients are $\pi'$-groups; that is, if all of its finite quotients have order divisible only by primes not in $\pi$.
\end{definition}

Equivalently, a group is $\pi$-free if it has no finite virtual $p$-quotients, for any $p \in \pi$. This is the phrasing used in Theorem \ref{intro:thm:pifree}.

\medskip

Clearly a group is $\pi$-free if and only if it is $p$-free for all $p \in \pi$. This terminology is inspired from the terminology in \cite{Schik, mio} introduced by Schikhof: indeed a group is $p$-free according to Definition \ref{def:pifree} if and only if its profinite completion is $p$-free according to Schikhof's definition. This class of groups clearly contains groups without finite quotients, whose uniform stability was already established in Example \ref{ex:fq_free}. But there are more examples, including residually finite ones.

\begin{example}
Finite $\pi'$-groups are $\pi$-free.
\end{example}

\begin{example}
\label{ex:period}

More generally, let $\Gamma$ be a periodic group without elements of order $p$ for all $p \in \pi$. Then $\Gamma$ is $\pi$-free. Indeed, the order of an element in a finite quotient of $\Gamma$ must divide the order of a preimage of it. For example Grigorchuk's group \cite{grigor_1} is a finitely generated periodic residually finite $2$-group, so it is $2'$-free.
\end{example}

\begin{example}
\label{ex:CSP}

Let $X$ be a finite alphabet, and let $\Aut(X^*)$ be the corresponding group of rooted tree automorphisms (see Example \ref{ex:aut:filt}). Following \cite{branch}, we say that a subgroup $\Gamma \leq \Aut(X^*)$ -- which is necessarily residually finite -- has the \emph{congruence subgroup property} if every finite-index subgroup of $\Gamma$ contains some level stabiliser $\Aut(X^*)_k$.

Let $\pi$ be the set of primes $p \leq |X|$. Then $\Aut(X^*)$ is pro-$\pi$, and so every subgroup $\Gamma \leq \Aut(X^*)$ with the congruence subgroup property is $\pi'$-free. This gives many examples of $\pi$-free groups, among which are many branch groups \cite{branch}, and in particular all Grigorchuk--Guptda--Sidki groups with non-constant defining vector \cite{GGS1, GGS2}. See \cite{GGS3, EGS} for more examples.
\end{example}

\begin{example}
Building on the previous example, for every odd prime $p$ there exist finitely generated residually finite torsion-free groups that are $p'$-free \cite{GGS2}.
\end{example}

We now prove stability of such groups. We first prove a quantitative lemma, and deduce the stability statement.

\begin{lemma}
\label{lem:pifree}

Let $\G$ be a virtually pro-$\pi$ family, and let $\ee_0$ be such that $G(\ez)$ is pro-$\pi$ for all $G \in \G$. Let $\f \colon \Gamma \to G \in \G$ be such that $\defe(\f) \leq \ee \leq \ez$, and suppose that the image $C$ of $\Gamma$ in $G/G(\ee)$ is a $\pi'$-group. Then there exists a homomorphism $\ff \colon \Gamma \to G$ such that $\dist(\f, \ff) \leq \ee$. Moreover, $\ff(\Gamma) \leq G$ is a finite group isomorphic to $C$, in particular it is a $\pi'$-group.
\end{lemma}

\begin{proof}
Consider the induced homomorphism $\f(\ee) \colon \Gamma \to G/G(\ee)$. By hypothesis $C = \f(\ee)(\Gamma)$ is a $\pi'$-group. Given $\dd \leq \ee$, we have the following lifting problem
\[\begin{tikzcd}
	&& {G/G(\dd)} \\
	\Gamma & C \\
	&& {G/G(\ee) \cong (G/G(\dd))/(G(\ee)/G(\dd))}
	\arrow["{\f(\ee)}", from=2-1, to=2-2]
	\arrow[dashed, from=2-2, to=1-3]
	\arrow[from=2-2, to=3-3]
	\arrow[from=1-3, to=3-3]
\end{tikzcd}\]
Since $G(\ee)/G(\dd)$ is a finite $\pi$-group, by the Schur--Zassenhaus Theorem there exists a lift: a homomorphism $\f(\dd) \colon \Gamma \to G/G(\dd)$ such that the projection onto $G/G(\ee)$ gives back $\f(\ee)$. Moreover $\f(\dd)(\Gamma) \cong C$: indeed $\f(\dd)$ factors through $C$, and it is a lift of $\f(\ee)$, which surjects onto $C$. Repeating this process by induction on a sequence $\ee \geq \dd_i \to 0$ gives a sequence of homomorphisms $\f(\dd_i) \colon \Gamma \to G/G(\dd_i)$ that are all compatible with the projections, and such that all images are groups isomorphic to $C$. Since $G$ is the projective limit of the groups $G/G(\dd_i)$, this sequence of compatible homomorphisms induces a homomorphism $\ff \colon \Gamma \to G$ such that $\ff(\ee) \colon \Gamma \to G/G(\ee)$ coincides with $\f(\ee)$ and $\ff(\Gamma)$ is a finite group isomorphic to $C$. Therefore $\dist(\f, \ff) \leq \ee$ by Lemma \ref{lem:mq}.
\end{proof}

\begin{proposition}
\label{prop:pifree}

Let $\G$ be a virtually pro-$\pi$ family and $\Gamma$ a $\pi$-free group. Then $\Gamma$ is uniformly $\G$-stable, with an optimal estimate.
\end{proposition}

\begin{proof}
Since $\Gamma$ is $\pi$-free, all finite quotients of $\Gamma$ are $\pi'$-groups. Therefore given a map $\f \colon \Gamma \to G$ with small enough defect, the previous lemma applies and $\f$ is close to a homomorphism. We conclude by Lemma \ref{lem:quant}. Moreover, Lemma \ref{lem:pifree} shows that the estimate for stability is optimal: there exists $\ez = \ez(\G) > 0$ such that if $\f \colon \Gamma \to G$ satisfies $\defe(\f) \leq \ee \leq \ez$, then there exists a homomorphism $\ff \colon \Gamma \to G$ such that $\dist(\f, \ff) \leq \ee$.
\end{proof}

\begin{remark}
We will see in Example \ref{ex:2unst} that the hypothesis of $\pi$-freeness is necessary, even for finite groups.
\end{remark}

\begin{example}
\label{ex:pfree}

Let $\Gamma$ be a $p$-free group, $\oo$ the ring of integers of a non-Archimedean local field of residual characteristic $p$. Then $\Gamma$ is uniformly $\GGL(\oo)$-stable, with an optimal estimate. Explicitly, for every map $\f: \Gamma \to \GGL_n(\oo)$ such that $\defe(\f) \leq \ee \leq |\uni|$ (where $\uni$ is a uniformiser), there exists a homomorphism $\ff \colon \Gamma \to \GGL_n(\oo)$ such that $\dist(\f, \ff) \leq \ee$.
\end{example}



The following example is a hint at the relation with bounded cohomology that will be explored in Section \ref{s:BC}.

\begin{example}
\label{ex:npa}

Let $\K$ be a non-Archimedean local field of characteristic $p$, and let $\Gamma$ be a normed $\K$-amenable group \cite[Definition 1.1]{mio}. Then $\Gamma$ is uniformly $\GGL(\oo)$-stable with a linear estimate: indeed such groups are characterised as being locally finite (thus periodic) and without elements of order $p$ \cite[Theorem 6.2]{mio}, so they are $p$-free by Example \ref{ex:period}.
\end{example}

\subsection{Graphs of groups}
\label{ss:gog}

In this subsection we exploit the conjugacy part of the Schur--Zassenhaus Theorem to prove stability with respect to virtually pro-$\pi$ families of several fundamental groups of graphs of groups.

\begin{remark}
This part of the Schur--Zassenhaus Theorem depends on the Odd Order Theorem \cite{FT}, but this can be avoided if we assume that either the kernel or the quotient of the extension to which the theorem is being applied is solvable (see Subsection \ref{ss:preli:split}). As in the proof of Proposition \ref{prop:pifree}, the extensions to which we apply the Schur--Zassenhaus Theorem are with a $\pi$-kernel and a $\pi'$-quotient, so if $\pi = \{ p \}$ then the kernel is solvable. Similarly if $\pi = \{ p, q \}$ then the kernel is solvable by Burnside's Theorem \cite[Chapter 31]{reptheory}. The same kind of statements can be given for the quotient. It may also be possible that we know for other reasons that $\G$ is virtually prosolvable: for instance this is the case for $\Gal(\mathbb{F})$ when $\mathbb{F}$ is a finite field (see Corollary \ref{cor:galfin_approx_1}). For the general case, however, we need the full power of the Schur--Zassenhaus Theorem, and so the general statements in this subsection depend on the Odd Order Theorem.
\end{remark}

Let us fix the definitions and notation (see \cite{Serre} for more details). Let $X = (V, E)$ be a connected graph with vertex set $V$ and edge set $E$, maps $\pm \colon E \to V : e \mapsto e^{\pm}$ giving the source and target of an edge, and a fix-point free involution $E \to E : e \mapsto \overline{e}$ reversing the orientation of each edge: that is $(\overline{e})^{\pm} = e^{\mp}$. A \emph{graph of groups} is composed by the following data: a connected graph $X = (V, E)$, groups $\Gamma_v$ for all $v \in V$ and $\Gamma_e$ for all $e \in E$ such that $\Gamma_e = \Gamma_{\overline{e}}$, and injective morphisms $\iota_e^{\pm} \colon \Gamma_e \to \Gamma_{e^\pm}$. By abuse of notation we use $X$ to denote both the graph of groups and the underlying abstract graph.

Let $T$ be a spanning tree of $X$. The \emph{fundamental group} of the graph of groups (with respect to $T$) is the group generated by all vertex groups, together with an element $t_e$ for each $e \in E$, with the additional relations:
\[t_{\overline{e}} = t_e^{-1}; \qquad t_e = 1 \text{ if } e \in T;\]
\[t_e \iota_e^-(x) t_e^{-1} = \iota_e^+(x) \text{ for all } x \in \Gamma_e \text{ (a generating set of } \Gamma_e \text{ suffices)}.\]
The isomorphism type of the fundamental group is independent of the choice of $T$. A presentation of the fundamental group is thus given by
\[\langle \{ S_v : v \in V \} \cup \{ t_e : e \in E \} \mid \{ R_v : v \in V \} \cup \{ R_e : e \in E \} \rangle,\]
where $\langle S_v \mid R_v \rangle$ is a presentation of $\Gamma_v$, and $R_e$ are the relations describing the identification $t_{\overline{e}} = t_e^{-1}$, the relation $t_e = 1$ if $e \in T$, and the effect of conjugacy by $t_e$ on $\iota_e^-(\Gamma_e)$. 

\medskip

As in the previous subsection, we first prove a quantitative lemma, and then deduce two stability statements. Since the fundamental group is defined in terms of a presentation, the most natural approach is by working in terms of it, which is possible by Proposition \ref{prop:stab_equiv} and Corollary \ref{cor:quant}.

\begin{lemma}
\label{lem:pifree_gog}

Let $\G$ be a virtually pro-$\pi$ family, and let $\ez$ be such that $G(\ez)$ is pro-$\pi$ for all $G \in \G$. Let $X$ be a connected graph of groups with vertex groups $\Gamma_v$, edge groups $\Gamma_e$ and edge inclusions $\iota_e^{\pm} \colon \Gamma_e \to \Gamma_{e^{\pm}}$. Let $\Gamma$ be the fundamental group of $X$ with respect to a spanning tree $T$, with the standard presentation $\langle S \mid R \rangle = \langle S_v, t_e \mid R_v, R_e \rangle$ as above.

Let $\hf \colon F_S \to G \in \G$ be a map with $\defe(\hf) \leq \ee \leq \ez$. Suppose further that for all $v \in V$ the restriction of $\hf$ to $F_{S_v}$ descends to a homomorphism $\f_v \colon \Gamma_v \to G$ such that, if $e^{\pm} = v$, then the image of $\f_v(\iota_e^{\pm}(\Gamma_e))$ in $G/G(\dd)$ is a $\pi'$-group, for all $\dd \leq \ee$.
Then there exists a homomorphism $\hff \colon F_S \to G$ such that $\dist(\hf, \hff) \leq \ee$ and $\hff$ descends to a homomorphism of $\Gamma$.
\end{lemma}

Before proceeding with the proof, let us comment on how this lemma is of interest independently of our applications. Indeed, it shows that such fundamentals groups of graphs of groups are examples of two notions related to stability, introduced recently in the literature and of which few examples are known so far.

\begin{remark}
\label{rem:constraint:epi}

First, Lemma \ref{lem:pifree_gog} is a statement about \emph{constraint stability}, a notion introduced by Arzhantseva and P\u{a}unescu in \cite{a:const}. Given a group $\Gamma$ and a subgroup $\Lambda \leq \Gamma$, let us say that $\Gamma$ is \emph{constraint $\G$-stable} with respect to $\Lambda$, if for every asymptotic homomorphism $(\f_n \colon \Gamma \to G_n \in \G)_{n \geq 1}$ such that its restriction to $\Lambda$ is close to a sequence of homomorphisms $(\ff_n \colon \Lambda \to G_n)_{n \geq 1}$, we can extend $(\ff_n)_{n \geq 1}$ to a homomorphism of $\Gamma$ that is close to $(\f_n)_{n \geq 1}$. As usual, this can be formalised to a pointwise notion and a uniform one. Similarly we can talk of $\Gamma$ being stable with respect to a set of subgroups. Then Lemma \ref{lem:pifree_gog} is a statement about constrant stability of $\Gamma$ with respect to the set of vertex subgroups.

Secondly, Lemma \ref{lem:pifree_gog} is a statement about \emph{stability of an epimorphism}, a notion introduced by Lazarovich and Levit in \cite{stepi}. We say that an epimorphism $\overline{\Gamma} \to \Gamma$ is \emph{$\G$-stable} if every asymptotic homomorphism of $\overline{\Gamma}$ that almost descends to $\Gamma$ (where ``almost descends'' is meant as in Proposition \ref{prop:stab_equiv}) is close to a sequence of homomorphisms of $\overline{\Gamma}$ that descend to $\Gamma$. Again, this leads to a pointwise and a uniform notion. Then Lemma \ref{lem:pifree_gog} is a statement about stability of the epimorphism of $((*_v \Gamma_v) * (*_e \langle t_e \rangle))$ onto $\Gamma$. Interestingly, this is precisely the setting in \cite{stepi}, where the authors prove stability of the same epimorphism in the case of virtually free groups, to deduce that all virtually free groups are stable in permutations.
\end{remark}

We proceed with the proof.

\begin{proof}[Proof of Lemma \ref{lem:pifree_gog}]
By hypothesis $\hf$ already satisfies all relations $R_v$. Our goal is to modify $\hf$ step by step so that it keeps this property, changes by at most $\ee$, and it also satisfies all relations $R_e$. For the rest of this proof, we denote the reduction map $G \to G/G(\dd)$ by $(\cdot \mod \dd)$. So if $A \leq G$, its image in $G/G(\dd)$ is denoted by $A \mod \dd$.

First of all, we can choose a set $E_+$ of positively oriented edges, and replace $\hf(t_e)$ by $\hf(t_{\overline{e}})^{-1}$ for all $e \notin E_+$. Since $\defe(\hf) \leq \ee$, these two elements are at a distance at most $\ee$, and so this substitution does not affect the other relations being satisfied in $G/G(\ee)$. Similarly we can replace $\hf(t_e)$ by $1$ for all $e \in T$. This leaves us with the conjugacy relations. 

\medskip

We start with the conjugacy relations given by the edges in $T$, which amalagamate edge groups in the adjacent vertex groups. We will modify $\hf$ at the vertex groups so that it still restricts to a homomorphism on each vertex group, but it also satisfies the amalgamations. Fix a vertex $v_0$ in $X$, with neighbours $v_1, \ldots, v_r$ and edges $e_1, \ldots, e_r \in T$, where $v_0 = e_i^-$ and $v_i = e_i^+$ for $i = 1, \ldots, r$. Let $A_i^- \coloneqq \f_{v_0}(\iota_{e_i}^-(\Gamma_{e_i})) \leq \f_{v_0}(\Gamma_{v_0})$ and $A_i^+ \coloneqq \f_{v_i}(\iota_{e_i}^+(\Gamma_{e_i}))  \leq \f_{v_i}(\Gamma_{v_i})$: recall that each $\f_v$ is a homomorphism, and so $A_i^{\pm}$ are groups. By hypothesis, for all $\dd \leq \ee$ the reduction modulo $\dd$ of both $A_i^-$ and $A_i^+$ is a $\pi'$-group. In particular, for all $\dd \leq \ee$ the projection map $A_i^{\pm} \mod \dd \to A_i^{\pm} \mod \ee$ is an isomorphism: indeed the kernel is contained in $G(\ee)/G(\dd)$ which is a finite $\pi$-group since $G(\ee) \leq G(\ez)$ is pro-$\pi$.

Now $\defe(\hf) \leq \ee$, so $\hf$ induces a homomorphism $\Gamma \to G/G(\ee)$. Moreover the copies of $\Gamma_{e_i}$ in $\Gamma_{v_0}$ and $\Gamma_{v_i}$ are amalgamated in every quotient of $\Gamma$. Thus, the reduction modulo $\ee$ of the homomorphisms $\f_{v_0} \circ \iota_{e_i}^-, \f_{v_i} \circ \iota_{e_i}^+ \colon \Gamma_{e_i} \to A_i^{\pm}$ is the same, we denote it by $f(\ee)$. Therefore the reductions modulo $\dd \leq \ee$ of the same homomorphisms are two (a priori distinct) solutions to the following lifting problem:
\[\begin{tikzcd}
	&& {G/G(\dd)} \\
	{\Gamma_{e_i}} & {f(\ee)(\Gamma_{e_i})} \\
	&& {G/G(\ee) \cong (G/G(\dd))/(G(\ee)/G(\dd))}
	\arrow[dashed, from=2-2, to=1-3]
	\arrow[from=2-2, to=3-3]
	\arrow[from=1-3, to=3-3]
	\arrow[from=2-1, to=2-2]
\end{tikzcd}\]
Since $G(\ee)/G(\dd)$ is a finite $\pi$-group, the Schur--Zassenhaus Theorem implies that these two lifts are $G(\ee)/G(\dd)$-conjugate. Let $t \in G(\ee)$ be a lift of a conjugating element. We can then replace $\f_{v_i}$ by $x \mapsto t \f_{v_i}(x) t^{-1}$, which is still a homomorphism of $\Gamma_{v_i}$, and remains $\ee$-close to $\f_{v_i}$ since $t \in G(\ee)$. But now the groups $A_i^- \coloneqq \f_{v_0}(\iota^-_{e_i}(\Gamma_{e_i}))$ and $A_i^+ \coloneqq t \f_{v_i}(\iota^+_{e_i}(\Gamma_{e_i})) t^{-1}$ are amalgamated in $G/G(\dd)$. We can repeat this process inductively on a sequence $\ee \geq \dd_k \to 0$: at each step we modify $\f_{v_i}$ conjugating it by an element of $G(\dd_k)$, so that the groups $A_i^{\pm}$ become amalgamated in $G/G(\dd_{k+1})$.
By compactness, a subsequence of these conjugating elements converges to an element $t \in G(\ee)$, and conjugating by $t$ we have modified $\f_{v_i}$ so that it is still $\ee$-close to $\varphi$, but it moreover satisfies the amalgamation $A_i^- = A_i^+$.

Doing this for all $i$, we get the desired relation for each edge $e_i$. These modifications are compatible, since they only affect $\f_{v_i}$ for $i \geq 1$ and not for $i = 0$. We can now apply the same argument to all neighbours of $v_i$ connected by edges of $T$ other than $e_0$, conjugating by elements of $G(\ee)$ so that the corresponding relations are satisfied in $G$, without affecting the behaviour of $\f_{v_i}$. Since $T$ is a tree, it is possible to iterate this procedure until all edges of $T$ are covered. This produces a homomorphism $\hf \colon F_S \to G$ with the same properties as before, which moreover satisfies all relations $\{ R_e : e \in T \}$. 

\medskip

We next move to the conjugacy relations given by edges not in $T$. For such an edge $e$, let $A^{\pm} \coloneqq \f_{e^{\pm}}(\iota_e^{\pm}(\Gamma_e)) \leq G$ as before, so that $A^{\pm} \mod \ee$ is a $\pi'$-group. Since $(\hf \mod \ee)$ descends to a homomorphism of $\Gamma$, we know that $A^+$ and $A^-$ are conjugate modulo $\ee$ by $\hf(t_e)$. So $A^+$ and $\hf(t_e)A^-\hf(t_e)^{-1}$ satisfy the same hypotheses as in the previous step. By the same argument, there exists $t \in G(\ee)$ such that $t \hf(t_e) A^- \hf(t_e)^{-1} t^{-1} = A^+$. We can thus replace $\hf(t_e)$ by $t \hf(t_e)$, which is congruent to $\hf(t_e)$ modulo $\ee$. In this way the conjugacy relations given by edges not in $T$ are now satisfied. Note that we have only modified the images of the edge elements, so this does not affect the definition of $\hf$ at the vertex groups, or at the other edge groups. In particular, the procedure at this step does not affect the modifications done at the previous steps. 

\medskip

We are left with a homomorphism $\hff \colon F_S \to G$ such that $\dist(\hf|_{F_{S_v}}, \hff|_{F_{S_v}}) \leq \ee$ for every vertex $v$, and $d(\hf(t_e), \hff(t_e)) \leq \ee$ for every edge $e$, and moreover $\hff$ satisfies all of the defining relations of $\Gamma$. Thus $\dist(\hf, \hff) \leq \ee$ and $\hff$ descends to a homomorphism $\Gamma \to G$.
\end{proof}

Here is our first stability result for graphs of groups:

\begin{proposition}
\label{prop:pifree_gog1}

Let $\G$ be a virtually pro-$\pi$ family and $\Gamma$ the fundamental group of a graph of groups such that all vertex groups $\Gamma_v$ are uniformly $\G$-stable, and such that $D_{\Gamma_v}^{\G}(\ee)$ is uniformly bounded away from $0$, for every $\ee > 0$. Suppose that for every edge $e$ adjacent to a vertex $v$, the image of $\Gamma_e$ in every finite quotient of $\Gamma_v$ is a $\pi'$-group. Then $\Gamma$ is uniformly $\G$-stable, and
\[D_{\Gamma}^{\G}(\ee) \geq \inf\limits_{v \in V} D_{\Gamma_v}^{\G}(\ee).\]
\end{proposition}

\begin{remark}
Note that the condition on $D_{\Gamma_v}^{\G}$ is automatically satisfied if the graph is finite.
\end{remark}

\begin{proof}
We use the characterisation of uniform stability from Corollary \ref{cor:quant}. Let $\langle S \mid R \rangle$ be the standard presentation of $\Gamma$. Fix $0 < \ee \leq \ez$. Let $\hf \colon F_S \to G \in \G$ be a homomorphism with $\defe(\hf) \leq \dd$, where $\dd = \inf_v D_{\Gamma}^{\G}(\ee)$, which is positive by hypothesis. The hypothesis of stability allows to modify $\hf$ by at most $\ee$ on the vertex generators $S_v$ so that $\hf|_{F_{S_v}}$ descends to a homomorphism of $\Gamma_v$. Now $\hf$ satisfies the hypotheses of Lemma \ref{lem:pifree_gog}, and so there exists $\hff \colon F_S \to G$ such that $\dist(\hf, \hff) \leq \ee$ and $\hff$ descends to a homomorphism of $\Gamma$.
\end{proof}

This next slightly different result combines Lemma \ref{lem:pifree_gog} with the last part of Lemma \ref{lem:pifree}:

\begin{proposition}
\label{prop:pifree_gog2}

Let $\G$ be a virtually pro-$\pi$ family and $\Gamma$ the fundamental group of a graph of groups such that for every vertex $v$ the image of $\Gamma_v$ in every finite quotient of $\Gamma$ is a $\pi'$-group. Then $\Gamma$ is uniformly $\G$-stable, with an optimal estimate.
\end{proposition}

\begin{proof}
Let $\langle S \mid R \rangle$ be the standard presentation of $\Gamma$. Fix $0 < \ee \leq \ez$ and let $\hf \colon F_S \to G \in \G$ be a homomorphism with $\defe(\hf) \leq \ee$. It induces a homomorphism $\Gamma \to G/G(\ee)$, whose restriction to $\Gamma_v$ is a $\pi'$-group. Now Lemma \ref{lem:pifree} allows to modify $\hf$ by at most $\ee$ on the vertex generators $S_v$ so that $\hf|_{F_{S_v}}$ descends to a homomorphism $\Gamma_v \to G$ whose image is again a finite $\pi'$-group. Then we apply Lemma \ref{lem:pifree_gog} and conclude as in Proposition \ref{prop:pifree_gog1}.
\end{proof}

\subsection{First corollaries}
\label{ss:vpropi:cor}

We now apply Propositions \ref{prop:pifree_gog1} and \ref{prop:pifree_gog2} to obtain examples of uniformly $\G$-stable groups. 

\medskip

The following is a direct consequence of Propositions \ref{prop:pifree_gog1} and \ref{prop:pifree_gog2}:

\begin{corollary}
\label{cor:pifree_gog}

Let $\G$ be a virtually pro-$\pi$ family. The following groups are uniformly $\G$-stable, with an optimal estimate:

\begin{enumerate}
\item Fundamental groups of connected graphs of groups, with $\pi$-free vertex groups.
\item Fundamental groups of finite, connected graphs of groups, with uniformly $\G$-stable vertex groups and $\pi$-free edge groups.
\end{enumerate}
\end{corollary}

The next corollary relies on Dunwoody's characterisation of groups of cohomological dimension (denoted $\cd$) at most $1$ \cite{cd1}:

\begin{corollary}
\label{cor:vfree_p}

Let $\G$ be a virtually pro-$\pi$ family. If $\cd_{\mathbb{F}_p}(\Gamma) \leq 1$ for all $p \in \pi$, then $\Gamma$ is uniformly $\G$-stable, with an optimal estimate. In particular, finitely generated virtually free groups without elements of order $p$, for any $p \in \pi$, are uniformly (and thus pointwise) $\G$-stable, with an optimal estimate.
\end{corollary}

\begin{proof}
By \cite{cd1}, a group has $\Fp$-cohomological dimension at most $1$ if and only if it is the fundamental group of a connected graph of groups whose vertex groups are finite and $p$-free. Even if the underlying graph is infinite, the estimate for uniform $\G$-stability of each vertex group is optimal by Lemma \ref{lem:pifree}, and so we can apply Proposition \ref{prop:pifree_gog1}. The statement about virtually free groups is the finitely generated case of \cite{cd1}, but it also follows more directly from Stallings's Theorem on groups with infinitely many ends, without going through cohomological dimension (see \cite{Stallings1} for the torsion-free case and \cite[5.A.9]{Stallings2} for the general case). Since finitely generated virtually free groups are finitely presented, pointwise stability also follows from Theorem \ref{thm:pw_un}.
\end{proof}

Recall from Example \ref{ex:wr_ust} that if $\Gamma$ is perfect, then $\Gamma \wr \mathbb{Z}$ is uniformly $\G$-stable for every profinite family $\G$. The following corollary of Proposition \ref{prop:pifree_gog2} strengthens this:

\begin{corollary}
\label{cor:wr}

Let $\G$ be a virtually pro-$\pi$ family. If $\Gamma$ does not surject onto $\mathbb{Z}/p\mathbb{Z}$, for any $p \in \pi$, then $\Gamma \wr \mathbb{Z}$ is uniformly $\G$-stable, with an optimal estimate.
\end{corollary}

\begin{proof}
We use the notation from Example \ref{ex:wr}. Note that $\Gamma \wr \mathbb{Z}$ is the fundamental group of a loop with vertex group and edge group $\Sigma_{\mathbb{Z}} \Gamma$, where the edge inclusions are the identity and the shift. So to apply Proposition \ref{prop:pifree_gog2} it suffices to show that in any finite quotient of $\Gamma \wr \mathbb{Z}$ the image of $\Sigma_{\mathbb{Z}} \Gamma$ has order coprime to $p$. By Example \ref{ex:wr}, this image must be abelian, and a finite abelian group of order divisible by $p$ surjects onto $\mathbb{Z}/p\mathbb{Z}$. This is ruled out by the hypothesis.
\end{proof}



\subsection{GBS groups}
\label{ss:GBS}

We now apply Proposition \ref{prop:pifree_gog2} to many Generalised Baumslag--Solitar (from now on: GBS) groups. We refer the reader to \cite{GBS} for more details on GBS groups. 

\medskip

Recall that the \emph{Baumslag--Solitar group} $\BS(m, n)$ is defined by the presentation $\langle s, t \mid ts^nt^{-1} = s^m \rangle$, so it is the fundamental group of a loop with vertex group and edge group $\mathbb{Z}$, where the edge inclusions are $\mathbb{Z} \to \mathbb{Z} : 1 \mapsto n, m$. More generally, a \emph{GBS} group is the fundamental group of a finite connected graph of groups $X = (V, E)$ all of whose vertex and edge groups are infinite cyclic. The information on the edge inclusions can be summarised in two \emph{weight functions} $w_{\pm} \colon E \to \mathbb{Z} \, \backslash \, \{ 0 \}$, that is $\iota_e^{\pm} \colon \Gamma_e \cong \mathbb{Z} \to \Gamma_{e^{\pm}} \cong \mathbb{Z} : 1 \mapsto w_{\pm}(e)$. Note that it suffices to know $w_+$ on all edges, or to know $w_\pm$ on a set of positively oriented edges, in order to recover all the information, since $\iota_e^- = \iota_{\overline{e}}^+$. We denote the graph of groups associated to a GBS group by $(X, w)$, where $X$ is the underlying graph and $w = (w_-, w_+)$ are the weight functions.

It will be convenient to extend the weight functions from oriented edges to oriented paths. So given an oriented path $P : v_1 \xrightarrow{e_1} v_2 \to \cdots \to v_k \xrightarrow{e_k} v_{k+1}$, we denote by $w_{\pm}(P) \coloneqq \prod_i w_{\pm}(e_i)$. If $p$ is a prime, we have $\nu_p(w_{\pm}(P)) = \sum_i \nu_p(w_{\pm}(e_i))$. In particular $\nu_p(w_{\pm}(P)) = 0$ if and only if $\nu_p(w_{\pm}(e_i)) = 0$ for all $i = 1, \ldots, k$.

The following corollary to Proposition \ref{prop:pifree_gog2} gives a combinatorial and arithmetic criterion that ensures that a GBS group is $\G$-stable.

\begin{corollary}
\label{cor:GBS_p}

Let $\G$ be a virtually pro-$\pi$ family, and let $\Gamma$ be a GBS group with weighted graph $(X, w)$. Suppose that for all $p \in \pi$ there exists a set $\mathbf{C}$ of oriented cycles such that every $C \in \mathbf{C}$ satisfies $\nu_p(w_-(C)) = 0 < \nu_p(w_+(C))$, and that for every vertex $y$ there exists a vertex $x$ belonging to some $C \in \mathbf{C}$, and a path $x \xrightarrow{P} y$ with $\nu_p(w_+(P)) = 0$. Then $\Gamma$ is uniformly (and thus pointwise) $\G$-stable, with an optimal estimate.
\end{corollary}

We will focus on uniform stability; since GBS groups are finitely presented, the pointwise stability follows from Theorem \ref{thm:pw_un}.
Note that the condition only requires that such cycles and paths exist for any given $p \in \pi$: we are allowed to choose different ones for each prime in $\pi$. The simplest example is that of Baumslag--Solitar groups, which features in Theorem \ref{intro:thm:pifree}. Here there is only one vertex so the condition on the existence of special paths is not needed, and the special cycle must be the loop.

\begin{corollary}
\label{cor:BS_p}

Let $\G$ be a virtually pro-$\pi$ family and suppose that each $p \in \pi$ divides exactly one of $m, n$. Then $\BS(m, n)$ is uniformly (and thus pointwise) $\G$-stable, with an optimal estimate.
\end{corollary}

Here is a more complex example of graph $(X, w)$ that satisfies the conditions of Corollary \ref{cor:GBS_p} where $\pi = \{ p \}$ is a single prime.

\begin{example}
We draw a set of positively oriented edges $e$, labeled by the pair of weights $(w_-(e), w_+(e))$. Each weight labeled $p'$ may be replaced by any integer coprime to $p$, each weight labeled $p$ by any non-zero multiple of $p$, and each weight labeled $*$ by any non-zero integer.

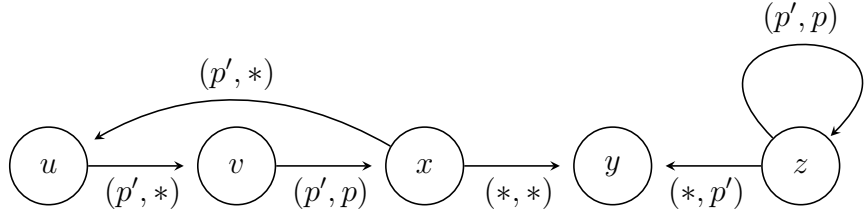
\begin{center}
\begin{tikzpicture}[->,>=stealth,
	shorten  >=5pt ,
	node  distance =2.5cm,
	semithick]

\node[state]	(U)                 {$u$};
\node[state]	(V) [right of=U]    {$v$};
\node[state]	(X) [right of=V] 	{$x$};
\node[state]	(Y) [right of=X] 	{$y$};
\node[state]	(Z) [right of=Y] 	{$z$};

\path	(U)	edge [below]               	node  {$(p', *)$} (V)
		(V)	edge [below]               	node  {$(p', p)$} (X)
		(X)	edge [below]               	node  {$(*, *)$} (Y)
			edge [bend right, above]   	node  {$(p', *)$} (U)
		(Z)	edge [in=45,out=135,looseness=8,loop, above]         	node  {$(p', p)$} (Z)
			edge [below]				node  {$(*, p')$} (Y);
\end{tikzpicture}
\end{center}

If $C : u \to v \to x \to u$, then $\nu_p(w_-(C)) = \nu_p(p') + \nu_p(p') + \nu_p(p') = 0$, while $\nu_p(w_+(C)) \geq \nu_p(p) > 0$, so it satisfies the hypothesis. Similarly, the loop at $z$ satisfies the hypothesis: $\nu_p(p') = 0 < \nu_p(p)$. The only vertex left to check is $y$, and for this we use the path $P : z \to y$, that satisfies $\nu_p(w_+(P)) = \nu_p(p') = 0$.
\end{example}

This example also clarifies that although the condition is stated in notation-heavy terms, it is quite easy to check, and there is no need to precisely compute $\nu_p(w_{\pm}(C, P))$. For instance $\nu_p(w_-(C)) = 0 < \nu_p(w_+(C))$ simply means that the negative weights along $C$ are all coprime to $p$, and that there is at least one positive weight that is divisible by $p$. Similarly $\nu_p(w_+(P)) = 0$ simply means that the positive weights along $P$ are all coprime to $p$.

\begin{proof}[Proof of Corollary \ref{cor:GBS_p}]
The proof will be split in a sequence of technical lemmas, some of which will be used in the proof of Corollary \ref{cor:GBS} in the next section.

Fix a GBS group $\Gamma$ with weighted graph $(X, w)$ and $p \in \pi$. Given a vertex $x$ we denote by $s_x$ the corresponding generator. We want to show that the graph of groups defining $\Gamma$ satisfies the hypotheses of Proposition \ref{prop:pifree_gog2}, so we need to show that for every vertex $x$, the image of $s_x$ has order coprime to $p$ in every finite quotient of $\Gamma$. For the sake of brevity, let us say that a vertex $x$ with the desired property is \emph{$p$-free}. The first lemma shows that the condition on the existence of special paths reduces the question to the vertices belonging to special cycles:

\begin{lemma}
\label{lem:GBS_pfree}

Suppose that $x$ is $p$-free, and let $x \xrightarrow{P} y$ be an oriented path with $\nu_p(w_+(P)) = 0$. Then $y$ is also $p$-free.
\end{lemma}

\begin{proof}
By induction on the length of the path, it suffices to show this for paths of length $1$. So suppose that $x \xrightarrow{e} y$ and let $(w_-(e), w_+(e)) = (m, n)$. By definition of the fundamental group, $s_x^m$ is conjugate to $s_y^n$ in $\Gamma$, and by hypothesis $\nu_p(n) = 0$. Let $f \colon \Gamma \to K$ be a finite quotient of $\Gamma$, let $o_x$ be the order of $f(s_x)$ and $o_y$ the order of $f(s_y)$. Conjugacy implies that $f(s_x)^m$ and $f(s_y)^n$ have the same order, that is $o_x/\mathrm{gcd}(o_x, m) = o_y/\mathrm{gcd}(o_y, n)$. Since $x$ is $p$-free, $\nu_p(o_x) = 0$, so $\nu_p(o_y) = \nu_p(\mathrm{gcd}(o_y, n)) \leq \nu_p(n) = 0$. Since $K$ was arbitrary, we conclude that $y$ is $p$-free.
\end{proof}

So we only need to show that if $C$ is a cycle such that $\nu_p(w_-(C)) = 0 < \nu_p(w_+(C))$, then every vertex of $C$ is $p$-free. Here is a sufficient condition for a vertex to be $p$-free:

\begin{lemma}
\label{lem:GBS_omn}

Let $x$ be a vertex such that $s_x^m$ is conjugate to $s_x^n$ inside $\Gamma$. If $\nu_p(m) = 0 < \nu_p(n)$, then $x$ is $p$-free.
\end{lemma}

\begin{proof}
Let $f \colon \Gamma \to K$ be a finite quotient of $\Gamma$, and let $o$ be the order of $f(s_x)$ in $K$. As in the previous lemma, the conjugacy implies that $o/\mathrm{gcd}(o, m) = o/\mathrm{gcd}(o, n)$ and so $\mathrm{gcd}(o, m) = \mathrm{gcd}(o, n)$. Since $p$ divides $n$ but not $m$, this is only possible if $p$ does not divide $o$.
\end{proof}

Note that this lemma alone is enough to conclude the proof in the case in which all cycles are loops, in particular it concludes the proof in the case of the Baumslag--Solitar group (which does not even need Lemma \ref{lem:GBS_pfree}). For more general cycles, we use instead the next lemma, which will also be used in the proof of Corollary \ref{cor:GBS}:

\begin{lemma}
\label{lem:GBS_conjP}

Let $x \xrightarrow{P} y$ be an oriented path. Then $s_x^{w_-(P)}$ is conjugate to $s_y^{w_+(P)}$.
\end{lemma}

\begin{proof}
We prove the statement by induction on the length of $P$. It is clear if $P$ has length $1$. Now suppose that the statement is true for $x \xrightarrow{P} y$ and let us prove it for $x \xrightarrow{P} y \xrightarrow{e} z$ for some edge $e$. By induction hypothesis $s_x^{w_-(P)}$ is conjugate to $s_y^{w_+(P)}$, and so $(s_x^{w_-(P)})^{w_-(e)}$ is conjugate to $(s_y^{w_+(P)})^{w_-(e)}$. Since by definition $s_y^{w_-(e)}$ is conjugate to $s_z^{w_+(e)}$, this implies that $(s_y^{w_-(e)})^{w_+(P)}$ is conjugate to $(s_z^{w_+(e)})^{w_+(P)}$. Thus $s_x^{w_-(P) \cdot w_-(e)}$ is conjugate to $s_z^{w_+(P) \cdot w_+(e)}$.
\end{proof}

In the case in which $P = C$ is a cycle such that $\nu_p(w_-(C)) = 0 < \nu_p(w_+(C))$, Lemma \ref{lem:GBS_conjP} implies that every vertex in $C$ satisfies a relation as in Lemma \ref{lem:GBS_omn}, and so is $p$-free. Lemma \ref{lem:GBS_pfree} allowed to reduce to looking at the vertices belonging to cycles in $\mathbf{C}$, hence this concludes the proof of Corollary \ref{cor:GBS_p}.
\end{proof}

Corollary \ref{cor:BS_p} has a further consequence. We know from Theorem \ref{thm:pw_un} that for finitely presented groups the notions of pointwise and uniform stability coincide. For finitely generated infinitely presented groups, we have mostly seen examples of \emph{uniform} stability, and non-examples of pointwise stability (Example \ref{ex:cex}).
But we know from Theorem \ref{thm:rf} that the largest residually finite quotient of a pointwise stable group is pointwise stable. We will apply this to the largest residually finite quotient of the non-residually finite Baumslag--Solitar groups, which were identified by Moldavanskii in \cite{moldavanskii}. Thus we obtain:

\begin{corollary}
\label{cor:rfBS_p}

Let $\G$ be a virtually pro-$\pi$ family and suppose that each $p \in \pi$ divides exactly one of $m, n$. Let $d \coloneqq \mathrm{gcd}(m, n)$, and suppose that $|m|, |n|$ are distinct from each other and from $1$. Then the group
\[\Gamma = \langle a, b_i : i \in \mathbb{Z} \mid [b_i^d, b_j] = 1, b_i^m = b_{i+1}^n, ab_ia^{-1} = b_{i+1} : i \in \mathbb{Z} \rangle\]
is finitely generated, infinitely presented, and pointwise $\G$-stable.

Writing $m = du, n = dv$, this group fits into an extension
\[1 \to \mathbb{Z} \left[ \frac{1}{uv} \right] \to \Gamma \to (\mathbb{Z}/d\mathbb{Z} * \mathbb{Z}) \to 1,\]
and in case $d = 1$ it is isomorphic to $\mathbb{Z}\left[\frac{1}{mn}\right] \rtimes_{\frac{m}{n}} \mathbb{Z}$.
\end{corollary}

\begin{proof}
By Corollary \ref{cor:BS_p} the group $\BS(m, n)$ is $\G$-stable. The condition on $|m|, |n|$ is equivalent to $\BS(m, n)$ being not residually finite \cite{BS:Hopf, Alex}, and the group above is its largest residually finite quotient \cite[Equation (1)]{moldavanskii}, which is pointwise $\G$-stable by Theorem \ref{thm:rf}. It is infinitely presented by \cite[Theorem 2]{moldavanskii}.

The description of the group is in \cite[Proposition 3]{moldavanskii}, except the author uses the presentation $C = \langle e_k : k > 0 \mid e_k = e_{k+1}^{uv} \rangle$ for the kernel of the extension \cite[Proposition 4]{moldavanskii}. This is isomorphic to $\mathbb{Z}[1/uv]$ under the isomorphism $\varphi \colon C \to \mathbb{Z}[1/uv] : e_k \mapsto (uv)^{-k}$, with inverse $\varphi^{-1} \colon \mathbb{Z}[1/uv] \to C : a(uv)^{-k} \mapsto e_k^a$.
\end{proof}

\begin{example}
The group $\mathbb{Z}\left[ \frac{1}{6} \right] \rtimes_{\frac{2}{3}} \mathbb{Z}$ is finitely generated, infinitely presented, residually finite and pointwise $\GGL(\oo)$-stable, whenever $\K$ has residual characteristic $2$ or $3$.
\end{example}

\pagebreak

\section{$\GGL(\oo)$-stability}
\label{s:char0}

In this section we focus on the family $\GGL(\oo)$, where $\oo$ is the ring of integers of a non-Archimedean local field $\K$ \textbf{of characteristic $0$}. By Ostrowski's Theorem, $\K$ is a finite extension of $\Qp$, where $p$ is the residual characteristic of $\K$. The stability results will be similar to the ones in Section \ref{s:vpropi}, but more flexible: this will be achieved by applying Lemma \ref{lem:cohopk} (instead of the Schur--Zassenhaus Theorem) to ensure existence or conjugacy in the lifting problems that occur. The more general results will allow to expand the class of examples that we presented in Section \ref{s:vpropi}. Comparing Theorem \ref{intro:thm:pifree} and Theorem \ref{intro:thm:vpfree}, let us stress what the methods of this section can achieve that the ones in the previous section could not.
\begin{itemize}
\item We saw that $p$-free groups are uniformly $\GGL(\oo)$-stable (Proposition \ref{prop:pifree}). Here we will see that \emph{virtually} $p$-free groups are uniformly $\GGL(\oo)$-stable (Proposition \ref{prop:vpfree}).
\item We saw that finitely generated virtually free groups without elements of order $p$ are uniformly $\GGL(\oo)$-stable (Corollary \ref{cor:vfree_p}). Here we will see that \emph{all} finitely generated virtually free groups are uniformly $\GGL(\oo)$-stable (Corollary \ref{cor:vfree}).
\item We saw that $\mathrm{BS}(m, n)$ is uniformly $\GGL(\oo)$-stable when $p$ divides exactly one of $m$ and $n$ (Corollary \ref{cor:BS_p}). Here we will see that $\mathrm{BS}(m, n)$ is $\GGL(\oo)$-stable if $\nu_p(m) \neq \nu_p(n)$ (Corollary \ref{cor:BS}).
\end{itemize}

\begin{notation}
We fix the following notation for the rest of the section (see Subsection \ref{ss:nona}): $\K$ is a finite extension of $\Qp$, with the unique norm $| \cdot |$ that restricts to the $p$-adic norm of $\Qp$. Let $\oo$ be the ring of integers, $\uni$ a uniformiser, so the maximal ideal of $\oo$ is $\pp = \uni \oo$, and $|\uni| = r < 1$, so that $|\K^\times| = r^{\mathbb{Z}}$. Since $|p| = |p|_p = p^{-1}$, there exists $a \geq 1$ such that $r^a = p^{-1} = |p|$, that is $p \in \pp^a$. Whenever $n$ does not vary, we denote $G \coloneqq \GGL_n(\oo)$ and the congruence subgroups by $G_k \coloneqq \GGL_n(\oo)_k$.
\end{notation}

The last Subsection \ref{ss:poschar} switches to the case in which $\K$ has characteristic $p$: we will prove stability of $\mathbb{Z}/2\mathbb{Z}$ in characteristic $2$ (Proposition \ref{prop:z2z}) and discuss why the same method of proof does not work in odd characteristic.

\subsection{Virtually $p$-free groups}

In this subsection we use the lifting part of Lemma \ref{lem:cohopk} to prove stability of groups that are only required to be \emph{virtually} $p$-free. This covers in particular all finite groups. Here is a characterisation:

\begin{lemma}
\label{lem:vpfree_char}

Let $\Gamma$ be a group and $p$ be a prime. Then $\Gamma$ is virtually $p$-free if and only if
\[\sup \{ \nu_p(|C|) : C \text{ is a finite quotient of } \Gamma \} < \infty.\]
\end{lemma}

\begin{proof}
Suppose that $\Gamma$ is virtually $p$-free, and let $H$ be a $p$-free finite-index subgroup. A finite-index subgroup of a $p$-free group is $p$-free, so we may assume that $H$ is normal. We claim that the supremum is achieved at $\Gamma/H$. Indeed, let $K$ be any other finite-index normal subgroup of $\Gamma$. Then
\[\nu_p(|\Gamma/K|) \leq \nu_p(|\Gamma/(K \cap H)|) = \nu_p(|\Gamma/H|) + \nu_p(|H/(K \cap H)|) = \nu_p(|\Gamma/H|),\]
where the last equality uses that $H$ is $p$-free.

Conversely, suppose that the supremum is achieved at $C = \Gamma/H$, where $H$ is a finite-index normal subgroup of $\Gamma$. Then $H$ is $p$-free. Indeed, if $K$ is a finite-index normal subgroup of $H$, let $N \leq K$ be a finite-index normal subgroup of $\Gamma$; then
\[\nu_p(|H/K|) \leq \nu_p(|H/N|) = \nu_p(|\Gamma/N|)/\nu_p(|\Gamma/H|) = 1,\]
where the last equality uses that $\nu_p(|\Gamma/N|) \geq \nu_p(|\Gamma/H|)$, and the latter value is assumed to be maximal.
\end{proof}

In other words, a group is virtually $p$-free if and only if there is a bound on the order of its finite virtual $p$-quotients. This is the phrasing used in Theorem \ref{intro:thm:vpfree}.

\begin{example}
\label{ex:locfin:vpfree}

Let $\Gamma$ be a locally finite group with a bound on the order of its finite $p$-subgroups. Say $\Gamma$ has no subgroup of order $p^k$ (and so no subgroup of order $p^l$ for $l \geq k$). Then it cannot admit a group of order $p^k$ as a virtual quotient. Lemma \ref{lem:vpfree_char} implies that $\Gamma$ is virtually $p$-free.
\end{example}

We now prove the analogues of Lemma \ref{lem:pifree} and Proposition \ref{prop:pifree}. The proofs are essentially the same, but they use Lemma \ref{lem:cohopk} instead of the Schur--Zassenhaus Theorem.

\begin{lemma}
\label{lem:vpfree}

Let $\f \colon \Gamma \to G \in \GGL(\oo)$ be such that $\defe(\f) \leq r^{ak} = p^{-k}$ for some $k \geq 1$, and suppose that the image of $\Gamma$ in $G/G_{ak}$ is a group $C$ with $\nu_p(|C|) \leq l < k/2$. Then there exists a homomorphism $\ff \colon \Gamma \to G$ such that $\dist(\f, \ff) \leq r^{a(k - l)} = p^l \cdot p^{-k}$. Moreover, $\ff(\Gamma) \leq G$ is a finite group of isomorphic to a quotient of $C$.
\end{lemma}

\begin{proof}
Let $\f_k = \f(r^{ak}) \colon \Gamma \to G/G_{ak}$ denote the induced homomorphism; we also get a homomorphism $\f_{k - l} \colon \Gamma \to G/G_{a(k-l)}$. So we have the lifting problem
\[\begin{tikzcd}
	&&& {G/G_{2a(k-l)}} \\
	\Gamma & C && {G/G_{ak}} \\
	&&& {G/G_{a(k-l)}}
	\arrow["{\varphi_k}"', from=2-2, to=2-4]
	\arrow["{\varphi_{k-l}}"', from=2-2, to=3-4]
	\arrow[dashed, from=2-2, to=1-4]
	\arrow[from=2-4, to=3-4]
	\arrow[from=1-4, to=3-4, bend left = 50]
	\arrow[from=2-1, to=2-2]
\end{tikzcd}\]
Now $G/G_{ak} \cong (G/G_{2a(k-l)})/(G_{ak}/G_{2a(k - l)})$ and $G/G_{a(k-l)} \cong (G/G_{2a(k-l)})/(G_{a(k-l)}/G_{2a(k - l)})$. Moreover by Lemma \ref{lem:vprop} we have an isomorphism
\[G_{a(k-l)}/G_{2a(k-l)} = \GGL_n(\oo)_{a(k-l)} / \GGL_n(\oo)_{2a(k-l)} \to \MM_n(\oo/\pp^{a(k-l)}).\]
So $G_{a(k-l)}/G_{2a(k-l)}$ is a $\mathbb{Z}/p^{(k - l)} \mathbb{Z}$-module (because $\oo/\pp^{a(k-l)}$ is) and the image under multiplication by $p^l$ is $G_{ak}/G_{2a(k-l)}$. Since $\nu_p(|C|) \leq l$ we are in the situation of Lemma \ref{lem:cohopk}, which shows that the lift exists. Lifting this in turn to a map $\ff_1 \colon \Gamma \to G$, we have $\defe(\ff_1) \leq r^{2a(k - l)}$ and $\dist(\f, \ff_1) \leq r^{a(k-l)}$.

The hypothesis $k > 2l$ implies that $\defe(\ff_1) \leq r \defe(\f)$, so we can apply the above procedure to $\ff_1$. This leads to a sequence $(\ff_i \colon \Gamma \to G)_{i \geq 1}$ such that $\defe(\ff_i) \to 0$ and $r^{a(k-l)} \geq \dist(\ff_i, \ff_{i-1}) \to 0$. The latter condition implies that the sequence $(\ff_i)_{i \geq 1}$ is Cauchy with respect to the uniform norm, so it converges to a homomorphism $\ff$. The inequality $\dist(\f, \ff) \leq r^{a(k-l)}$ holds because it does hold for every $\ff_i$.
\end{proof}

\begin{proposition}
\label{prop:vpfree}

Let $\Gamma$ be virtually $p$-free. Then $\Gamma$ is uniformly $\GGL(\oo)$-stable, with a linear estimate.
\end{proposition}

\begin{proof}
Let $l$ be the supremum from Lemma \ref{lem:vpfree_char}: for every finite quotient $C$ of $\Gamma$, it holds $\nu_p(|C|) \leq l$. Therefore given a map $\f \colon \Gamma \to \GGL_n(\oo)$ with defect $p^{-k} < p^{-2l}$, the previous lemma applies and $\f$ is $p^l \cdot p^{-k}$-close to a homomorphism. We conclude by Lemma \ref{lem:quant}.
\end{proof}


\begin{example}
\label{ex:finst}

All finite groups are $\GGL(\oo)$-stable, with a linear estimate. More precisely, let $\Gamma$ be a finite group, and $\nu_p(\Gamma) = l$. Then for every $\f \colon \Gamma \to \GGL_n(\oo)$ such that $\defe(\f) \leq p^{-2l}$, there exists a homomorphism $\ff \colon \Gamma \to \GGL_n(\oo)$ such that $\dist(\f, \ff) \leq p^l \cdot \defe(\f)$. We will see in Lemma \ref{lem:badestimates} that these estimates are essentially sharp.
\end{example}

\begin{example}
A finitely generated group $\Gamma$ of finite exponent is uniformly $\GGL(\oo)$-stable, with a linear estimate. Indeed, by Zelmanov's solution of the restricted Burnside problem \cite{RBP_odd, RBP_2}, the largest residually finite quotient of $\Gamma$ is finite, and so we conclude by Example \ref{ex:finst} and Theorem \ref{thm:rf}. In particular, free Burnside groups of finite rank are uniformly $\GGL(\oo)$-stable, with a linear estimate.
\end{example}

We saw in Example \ref{ex:npa} that normed $\K$-amenable groups are $\GGL(\oo)$-stable when $\K$ has characteristic $p$. The following example completes the picture:

\begin{example}
\label{ex:npa2}

Let $\Gamma$ be a normed $\K$-amenable group \cite[Definition 1.1]{mio}. Then $\Gamma$ is uniformly $\GGL(\oo)$-stable, with a linear estimate: indeed such groups are characterised as being locally finite and with a bound on the order of their finite $p$-subgroups \cite[Theorem 6.2]{mio}, and so they are virtually $p$-free by Example \ref{ex:locfin:vpfree}.
\end{example}

The stability of finite groups can be pushed further to the following result, which will be used in the next section.

\begin{corollary}
\label{cor:dirsum}

Let $\Gamma$ be a countable direct sum of finite groups. Then $\Gamma$ is pointwise $\GGL(\oo)$-stable.
\end{corollary}

\begin{proof}
Write $\Gamma \coloneqq \bigoplus_{i \geq 1} F_i$, and let $\Gamma_k \coloneqq \bigoplus_{i \leq k} F_i$, which is a finite subgroup of $\Gamma$. Let $r_k \colon \Gamma \to \Gamma_k$ denote the retraction. Let $(\f_n \colon \Gamma \to \GGL_{d_n}(\oo))_{n \geq 1}$ be a pointwise asymptotic homomorphism. Then $(\f_n^k \coloneqq \f_n|_{\Gamma_k} \colon \Gamma_k \to \GGL_{d_n}(\oo))_{n \geq 1}$ is also a pointwise asymptotic homomorphism. Since $\Gamma_k$ is stable (Example \ref{ex:finst}) and finite, there exists a sequence of homomorphism $(\ff_n^k \colon \Gamma_k \to \GGL_{d_n}(\oo))_{n \geq 1}$ which is uniformly asymptotically close to $(\f_n^k)_{n \geq 1}$. We let $n(k)$ be an integer such that $\dist(\f_n^k, \ff_n^k) \leq 1/k$ for all $n \geq n(k)$. By choosing $n(k)$ inductively, we may assume that $n(k+1) > n(k)$. We moreover set $k(n)$ to be the unique integer $k$ such that $n(k) \leq n < n(k+1)$.

Now set $\ff_n \coloneqq \ff^{k(n)}_n \circ r_{k(n)} \colon \Gamma \to \GGL_{d_n}(\oo)$. Clearly $\ff_n$ is a homomorphism. We are left to show that it is pointwise asymptotically close to $\f_n$. Indeed, let $g \in \Gamma$, and let $k \geq 1$ be such that $g \in \Gamma_k$. Then for all $n \geq n(k)$ we have $k(n) \geq k$ and thus:
\[\ff_n(g) = \ff^{k(n)}_n \circ r_{k(n)}(g) = \psi_n^{k(n)}(g).\]
Moreover, since $k(n) \geq k$ and $n \geq n(k(n))$, we have
\[\dist(\f_n(g), \ff_n(g)) = \dist(\f_n^{k(n)}, \ff^{k(n)}_n) \leq 1/k(n).\]
As $n \to \infty$ also $k(n) \to \infty$, and thus $\dist(\f_n(g), \ff_n(g)) \to 0$.
\end{proof}

\subsection{Graphs of groups}

We now use the conjugacy part of Lemma \ref{lem:cohopk} to strengthen the results on stability of graphs of groups from Subsection \ref{ss:gog} from which we borrow the notation. As before, we start with the analogue of Lemma \ref{lem:pifree_gog} and then prove the analogues of Propositions \ref{prop:pifree_gog1} and \ref{prop:pifree_gog2}. Also here, the lemma gives examples of constraint stability \cite{a:const} and stability of an epimorphism \cite{stepi} (see Remark \ref{rem:constraint:epi}).

\begin{lemma}
\label{lem:vpfree_gog}

Let $X$ be a connected graph of groups with vertex groups $\Gamma_v$, edge groups $\Gamma_e$ and edge inclusions $\iota_e^{\pm} \colon \Gamma_e \to \Gamma_{e^{\pm}}$. Let $\Gamma$ be the fundamental group of $X$ with respect to a spanning tree $T$, with the standard presentation $\langle S \mid R \rangle = \langle S_v, t_e \mid R_v, R_e \rangle$.

Let $\hf \colon F_S \to G \in \GGL(\oo)$ be a map with $\defe(\hf) \leq r^{ak} = p^{-k}$ for some $k \geq 1$. Suppose further that for all $v \in V$ the restriction of $\hf$ to $F_{S_v}$ descends to a homomorphism $\f_v \colon \Gamma_v \to G$ such that, for all $m \geq k$, if $e^{\pm} = v$, the image $C$ of $\f_v(\iota_e^{\pm}(\Gamma_e))$ in $G/G_{am}$ satisfies $\nu_p(|C|) \leq l < k/2$. Then there exists a homomorphism $\hff \colon F_S \to G$ such that $\dist(\hf, \hff) \leq r^{a(k-l)} = p^l \cdot p^{-k}$ and $\hff$ descends to a homomorphism of $\Gamma$.
\end{lemma}

\begin{proof}
The proof is essentially the same as that of Lemma \ref{lem:pifree_gog}, using Lemma \ref{lem:cohopk} as we did in the proof of Lemma \ref{lem:vpfree}.

We start by setting $\hf(t_e) = 1$ for all $e \in T$. Next we modify $\hf$ at the vertex groups so that it satisfies the conjugacy relations given by edges in $T$. Using the same induction argument as in the proof of Lemma \ref{lem:pifree_gog}, it suffices to treat the case $(v_0 \xrightarrow{e_i} v_i) \in T$: we need to find $t \in G_{a(k-l)}$ that conjugates the image of $\Gamma_{e_i}$ in $\f(\Gamma_{v_0})$ to that of $\f(\Gamma_{v_i})$. Considering the following lifting problem:
\[\begin{tikzcd}
	&&& {G/G_{2a(k-l)}} \\
	{\Gamma_{e_i}} & {f_k(\Gamma_{e_i})} && {G/G_{ak}} \\
	&&& {G/G_{a(k-l)}}
	\arrow[from=2-4, to=3-4]
	\arrow[from=1-4, to=2-4]
	\arrow[from=2-1, to=2-2]
	\arrow[from=2-2, to=1-4, dashed]
	\arrow[from=2-2, to=2-4]
	\arrow[from=2-2, to=3-4]
\end{tikzcd}\]
and using Lemma \ref{lem:cohopk} to prove that any two lifts of the horizontal arrow are $G_{a(k-l)}$-conjugate, we obtain an element $t \in G_{a(k-l)}$ that conjugates the two images modulo $r^{2a(k-l)}$. Iterating this process yields a sequence that converges to the desired conjugating element.

Finally we modify $\hf$ at the generators of edges not in $T$
so that it satisfies the corresponding conjugacy relations; this is achieved by adapting the proof of Lemma \ref{lem:pifree_gog} in the same way as in the previous step.
\end{proof}

\begin{proposition}
\label{prop:vpfree_gog1}

Let $\G$ be a virtually pro-$\pi$ family and $\Gamma$ the fundamental group of a graph of groups such that all vertex groups $\Gamma_v$ are uniformly $\GGL(\oo)$-stable, and such that $D_{\Gamma_v}^{\GGL(\oo)}(\ee)$ is uniformly bounded away from $0$, for every $\ee > 0$. Suppose that there exists $l \geq 1$ such that for every edge $e$ adjacent to a vertex $v$, the image of $\Gamma_e$ in every finite quotient of $\Gamma_v$ has no subgroup of order $p^l$. Then $\Gamma$ is uniformly $\GGL(\oo)$-stable, and
\[D_{\Gamma}^{\GGL(\oo)}(\ee) \geq p^{-l} \cdot \inf\limits_{v \in V} D_{\Gamma_v}^{\GGL(\oo)}(\ee).\]
\end{proposition}

\begin{remark}
\label{rem:uniformestimate}

The second condition is really asking for the image of $\Gamma_e$ inside $\Gamma_v$ to be ``virtually $p$-free relative to $\Gamma_v$'': that is, there exists a finite-index subgroup $\Gamma_e' \leq \Gamma_e$ such that the image of $\Gamma_e'$ in every finite quotient of $\Gamma_v$ is a $p'$-group (see Lemma \ref{lem:vpfree_char} and its proof). We state the proposition by first fixing $l$ because we need this condition to be uniform on the vertices. If however the graph is finite, then the uniformity hypothesis is automatically satisfied.
\end{remark}

\begin{proof}[Proof of Proposition \ref{prop:vpfree_gog1}]
As in the proof of Proposition \ref{prop:pifree_gog1}, we start with $\hf \colon F_S \to \GGL_n(\oo)$ of small defect, use stability of the vertex groups to modify it at the vertex generators, then apply Lemma \ref{lem:vpfree_gog} to obtain a homomorphism $\hff \colon F_S \to \GGL_n(\oo)$ close to $\hf$ that descends to a homomorphism of $\Gamma$. Note that Lemma \ref{lem:vpfree_gog} can be applied because if a finite group has a subgroup of order divisible by $p^l$, then it has a subgroup of order $p^l$.
\end{proof}


\begin{proposition}
\label{prop:vpfree_gog2}

Let $\Gamma$ be the fundamental group of a graph of groups. Suppose that there exists $l \geq 1$ such that for every vertex $v$ the image of $\Gamma_v$ in every finite quotient of $\Gamma$ has no subgroup of order $p^l$. Then $\Gamma$ is uniformly $\GGL(\oo)$-stable, with a linear estimate.
\end{proposition}

\begin{proof}
As in the proof of Proposition \ref{prop:pifree_gog2}, we start with $\hf \colon F_S \to \GGL_n(\oo)$ of small defect, apply Lemma \ref{lem:vpfree} to modify it at the vertex groups so that it descends to a homomorphism with finite image on each vertex group, and finally apply Lemma \ref{lem:vpfree_gog} to obtain a homomorphism $\ff \colon F_S \to \GGL_n(\oo)$ close to $\hf$ that descends to $\Gamma$.
\end{proof}


\subsection{Corollaries}
\label{ss:char0:cor}

We now apply Propositions \ref{prop:vpfree_gog1} and \ref{prop:vpfree_gog2} to obtain some examples of uniformly $\GGL(\oo)$-stable groups, which strengthen those in Subsections \ref{ss:vpropi:cor} and \ref{ss:GBS}.

\begin{corollary}
\label{cor:vpfree_gog}

The following groups are uniformly $\GGL(\oo)$-stable, with a linear estimate:

\begin{enumerate}
\item Fundamental groups of finite, connected graphs of groups, with virtually $p$-free vertex groups.
\item Fundamental groups of finite, connected graphs of groups, with uniformly $\GGL(\oo)$-stable vertex groups and virtually $p$-free edge groups.
\end{enumerate}
\end{corollary}

\begin{proof}
This is a direct consequence of Propositions \ref{prop:vpfree_gog1} and \ref{prop:vpfree_gog2}. We restrict to finite graphs of groups in order to ensure the existence of an integer $l$ as in the statements (see Remark \ref{rem:uniformestimate}).
\end{proof}

A special case of Item $2.$ is when edge groups are finite. Fundamental groups of finite graphs of groups with finitely generated vertex groups and finite edge groups are precisely the finitely generated groups with more than one end, by Stalling's Theorem \cite{Stallings1, Stallings2}.

\begin{corollary}
\label{cor:vfree}

Finitely generated virtually free groups are uniformly (and thus pointwise) $\GGL(\oo)$-stable, with a linear estimate.
\end{corollary}

\begin{proof}
The uniform stability follows again from Proposition \ref{prop:vpfree_gog1} and \cite{cd1} or \cite{Stallings1, Stallings2}: such groups are fundamental groups of finite connected graphs of groups, with finite vertex groups. The pointwise stability then follows from Theorem \ref{thm:pw_un}.
\end{proof}

As for GBS groups, we have a criterion for stability much simpler than that of Corollary \ref{cor:GBS_p}:

\begin{corollary}
\label{cor:GBS}

Let $\Gamma$ be a GBS group corresponding to the weighted graph $(X, w)$. Suppose that there exist a cycle $C$ in $X$ satisfying $\nu_p(w_-(C)) \neq \nu_p(w_+(C))$. Then $\Gamma$ is uniformly (and thus pointwise) $\GGL(\oo)$-stable, with a linear estimate.
\end{corollary}

Again, pointwise stability follows from uniform stability and Theorem \ref{thm:pw_un}.

\begin{proof}
The proof is similar to Corollary \ref{cor:GBS_p}. Given a vertex $x$ denote by $s_x$ the corresponding generator. We want to show that the graph of groups defining $\Gamma$ satisfies the hypotheses of Proposition \ref{prop:vpfree_gog2}, so we need to show that for every vertex $x$ the image of $s_x$ in every finite quotient of $\Gamma$ has order divisible by at most a uniformly bounded power of $p$. For the sake of brevity, let us say that such a vertex is \emph{virtually $p$-free}. 

\medskip

We start by reducing to a single vertex. Indeed, we claim that if some vertex $x$ is virtually $p$-free, then all vertices are. Since the graph $X$ is finite and connected, it suffices to show that if $x \xrightarrow{e} y$ is an edge in $e$ and $x$ is virtually $p$-free, then so is $y$. Let $(w_-(e), w_+(e)) = (m, n)$ be the weights of $e$. Then $s_x^m$ is conjugate to $s_y^n$, so as in Lemma \ref{lem:GBS_pfree}, if $o_x, o_y$ are the orders of $s_x, s_y$ in a finite quotient $K$ of $\Gamma$, we have $o_x/\mathrm{gcd}(o_x, m) = o_y/\mathrm{gcd}(o_y, n)$. Thus
\[\nu_p(o_y) = \nu_p(o_x) - \nu_p(\mathrm{gcd}(o_x, m)) + \nu_p(\mathrm{gcd}(o_y, n)) \leq \nu_p(o_x) + \nu_p(n)\]
is bounded independently of $K$, and $y$ is virtually $p$-free. 

\medskip

We are left to show that a vertex $x$ lying on a cycle $C$ satisfying $\nu_p(w_-(C)) \neq \nu_p(w_+(C))$ is virtually $p$-free. Let $m \coloneqq w_-(C), n \coloneqq w_+(C)$. By Lemma \ref{lem:GBS_conjP} we know that $s_x^m$ is conjugate to $s_x^n$. Now let $o$ be the order of $s_x$ in a finite quotient of $\Gamma$. As in Lemma \ref{lem:GBS_omn} we have $\mathrm{gcd}(o, m) = \mathrm{gcd}(o, n)$, and so
\[\min \{ \nu_p(o), \nu_p(n) \} = \nu_p(\mathrm{gcd}(o, n)) = \nu_p(\mathrm{gcd}(o, m)) = \min \{\nu_p(o), \nu_p(m)\}.\]
Since $\nu_p(m) \neq \nu_p(n)$, this is only possible if $\nu_p(o) \leq \min\{\nu_p(m), \nu_p(n)\}$, which gives a bound on $\nu_p(o)$ independent of $K$, and concludes the proof.
\end{proof}

We similarly obtain corollaries about the special case of Baumslag--Solitar groups and their largest residually finite quotient:

\begin{corollary}
\label{cor:BS}

Suppose that $\nu_p(m) \neq \nu_p(n)$. Then $\BS(m, n)$ is uniformly (and thus pointwise) $\GGL(\oo)$-stable, with a linear estimate.
\end{corollary}

\begin{corollary}
\label{cor:rfBS}

Suppose moreover that $|m|, |n|$ are distinct from each other and from $1$ and let $d \coloneqq \mathrm{gcd}(m, n)$. Then the group
\[\Gamma = \langle a, b_i : i \in \mathbb{Z} \mid [b_i^d, b_j] = 1, b_i^m = b_{i+1}^n, ab_ia^{-1} = b_{i+1} : i \in \mathbb{Z} \rangle\]
is finitely generated, infinitely presented, and pointwise $\GGL(\oo)$-stable.
\end{corollary}

Note that the special case $d = 1$, which admits the nicer description $\mathbb{Z} \left[ \frac{1}{mn} \right] \rtimes_{\frac{m}{n}} \mathbb{Z}$, is not more general than Corollary \ref{cor:rfBS_p}: indeed, if $\nu_p(m) \neq \nu_p(n)$ and $(m, n) = 1$, then $p$ must divide exactly one of $m, n$.


\subsection{Positive characteristic}
\label{ss:poschar}

Let $\K$ be a non-Archimedean local field \textbf{of characteristic $p > 0$} for the rest of this subsection. The proofs in this section until now all relied on the cohomological Lemma \ref{lem:cohopk}. This cannot have as strong of an analogue in characteristic $p$. For instance if $M$ is an $\Fp$-module, seen as a trivial $\Fp[\mathbb{Z}/p\mathbb{Z}]$-module, then $\HH^n(\mathbb{Z}/p\mathbb{Z}; M) \cong M$ for all $n \geq 1$. Therefore, analogues of the stability results we proved so far need a different approach. 

\medskip

In Sections \ref{s:vpropi} and \ref{s:char0} many results reduced to statements about finite groups. So we restrict our attention to those, and ask:

\begin{question}
\label{q:fin}

Are all finite groups $\GGL(\oo)$-stable?
\end{question}

We already know that finite groups without elements of order $p$ are $\GGL(\oo)$-stable, by Proposition \ref{prop:pifree}. So the next natural question is whether finite $p$-groups are $\GGL(\oo)$-stable. Using a Theorem of Gasch\"utz \cite[Theorem 7.43]{Rotman} and an argument similar to the one of Lemma \ref{lem:vpfree}, one can show that if $\Gamma$ is a finite group such that all $p$-Sylow subgroups are $\GGL(\oo)$-stable with a subquadratic estimate, then $\Gamma$ is stable. The hypothesis of the subquadratic estimate is needed in order to have the kernel $G_k/G_{2k}$ of the corresponding lifting problem be abelian: an example due to Zassenhaus implies that this hypothesis is necessary in Gasch\"utz's Theorem \cite[Postscriptum]{Gaschutz}.

However, stability with a subquadratic estimate for finite $p$-groups is too strong of a requirement to have useful applications. Indeed, the next example shows that the estimate for $\mathbb{Z}/p^k\mathbb{Z}$ is, at best, polynomial of degree $p^k$.

\begin{example}
\label{ex:zp:est}

Let $A \in \MM_n(\oo)$ be such that $\| A \| \leq \ee < 1$. Then $(I + A) \in \GGL_n(\oo)$ by Lemma \ref{lem:GL}. Using that $\K$ has characteristic $p$:
\[\| (I + A)^{p^k} - I \| = \| A^{p^k} \| \leq \| A \|^{p^k} \leq \ee^{p^k}.\]
On the other hand, if we chose $A$ so that $\| A^{p^k} \| = \ee^{p^k}$, then $(I + A)$ is at a distance at least $\ee$ from any matrix $M$ satisfying $M^{p^k} = I$. Indeed, suppose otherwise that $\|(I + A) - M\| < \ee$ and $M^{p^k} = I$. Writing $M = I + B$, we have $\| A - B\| < \ee$ - which implies $\| A \| = \| B \| = \ee$ - and $B^{p^k} = 0$.
Finally, we prove by induction on $i$ that for all $i = 0, \ldots, p^k$ we have $\| A^i B^{p^k-i} \| < \ee^{p^k}$. The base case $i = 0$ is obvious. For $i > 0$:
\begin{align*}
\| A^i B^{p^k-i} \| &= \| A^{i-1} (A-B) B^{p^k-i} + A^{i-1} B^{p^k - i + 1} \| \\
&\leq \max \{ \ee^{p^k-1} \| A - B \|, \| A^{i-1} B^{p^k - (i-1)} \} < \ee^{p^k}.
\end{align*}
When $i = p^k$, this contradicts our choice of $A$ and we conclude.
\end{example}

Notice that for all groups whose stability was proved so far, the estimates were always linear. Thus Example \ref{ex:zp:est} shows that Question \ref{q:fin} is quite different from the stability problems we encountered until now. 

\medskip

Here is a special case that admits a positive answer:

\begin{proposition}
\label{prop:z2z}

Let $\K$ be a non-Archimedean local field of characteristic $2$, with ring of integers $\oo$. Then $\mathbb{Z}/2\mathbb{Z}$ is $\GGL(\oo)$-stable, with a quadratic estimate.
\end{proposition}

\begin{remark}
\label{rem:quadratic:estimate}

This estimate is sharp by Example \ref{ex:zp:est}. It also shows that there exist stable groups whose stability estimate is not linear. It would be interesting to find examples in characteristic $0$.
\end{remark}

Combining Proposition \ref{prop:z2z} with Proposition \ref{prop:pifree_gog1}, we obtain more examples of $\GGL(\oo)$-stable groups that are not covered by the results in Section \ref{s:vpropi}:

\begin{example}
For $\K$ as above, the infinite dihedral group $D_\infty \cong \mathbb{Z}/2\mathbb{Z} * \mathbb{Z}/2\mathbb{Z}$ is $\GGL(\oo)$-stable; more generally, free Coxeter groups are $\GGL(\oo)$-stable. The modular group $\operatorname{PSL}_2(\mathbb{Z}) \cong \mathbb{Z}/2\mathbb{Z} * \mathbb{Z}/3\mathbb{Z}$ is $\GGL(\oo)$-stable. For all of these groups, the estimate is quadratic.
\end{example}

The proof of Proposition \ref{prop:z2z} relies on the following theorem, which classifies similarity classes of involutory matrices over the quotient rings $\oo / \pp^k$:

\begin{theorem}[Brawley--Gamble {\cite[Theorem 3.1]{Braw2}}]
\label{thm:BG}

Let $\mathbf{R}$ be a finite commutative local ring of characteristic a power of $2$ and maximal ideal $m$. Let $A \in \MM_n(\mathbf{R})$ be such that $A^2 = I$. Then there exists $P \in \GGL_n(\mathbf{R}), 0 \leq l \leq n/2$ and $B \in I_{(n-2l)} + \MM_{(n-2l)}(m)$ such that
\[PAP^{-1} = \begin{pmatrix}
0 & I_l & \\
I_l & 0 & \\
& & B
\end{pmatrix}.\]
\end{theorem}

This has the following implication in our setting:

\begin{corollary}
\label{cor:BG}

Let $A \in \GGL_n(\oo)$ be such that $\| A^2 - I \| \leq r^k < 1$. Then there exists $P \in \GGL_n(\oo)$ and $B \in I_{(n-2l)} + \MM_{(n-2l)}(\pp)$ such that
\[PAP^{-1} \equiv \begin{pmatrix}
0 & I_l & \\
I_l & 0 & \\
& & B
\end{pmatrix} \mod \pp^k.\]
\end{corollary}

\begin{proof}
Apply Theorem \ref{thm:BG} to the reduction of $A$ modulo $\pp^k$, which is involutory over $\oo/\pp^k$; then lift the matrices $P$ and $B$ to elements of $\MM_n(\oo)$. Since $k \geq 1$, the lift to $\GGL_n(\oo)$ of a matrix in $\GGL_n(\oo/\pp^k)$ is invertible, by Lemma \ref{lem:GL}.
\end{proof}

We proceed with the proof. This will use repeatedly the identity $(I + M)^2 = I + M^2$, which holds since $\K$ has characteristic $2$.

\begin{proof}[Proof of Proposition \ref{prop:z2z}]
Let $\ee > 0$. We will show that if $A \in \MM_n(\oo)$ is such that $\| A^2 - I \| \leq \ee^2$, then there exists $A' \in \MM_n(\oo)$ such that $(A')^2 = I$ and $\| A - A' \| \leq \ee$. Since $\| \cdot \|$ takes values in $r^\mathbb{Z}$, we may assume that $\ee \eqqcolon r^a$, and since the statement is trivial when $\ee \geq 1$, we may assume $a \geq 1$. The proof is by induction on $n$. For $n = 1$, we compute $\ee^2 \geq |a^2 - 1| = |(a - 1)^2| = |a - 1|^2$, and thus $|a - 1| \leq \ee$. So let $n > 1$ and suppose by induction that the statement holds up to $(n - 1)$ for all $\ee > 0$. 

\medskip

Since $\| A^2 - I \| \leq r^{2a}$, by Corollary \ref{cor:BG} there exist $P \in \GGL_n(\oo), 0 \leq l \leq n/2$ and $B \in I_{(n-2l)} + \MM_{(n-2l)}(\pp)$ such that
\[
PAP^{-1} \equiv \begin{pmatrix}
0 & I_l & 0 \\
I_l & 0 & 0 \\
0 & 0 & B
\end{pmatrix} \mod \pp^{2a}.
\]
Now $\| B^2 - I_{(n-2l)} \| \leq \max \{ r^{2a}, \| P A^2 P^{-1} - I_n \| \} \leq \ee^2$. If $l \geq 1$, it follows by induction that there exists $B' \in \MM_{(n-2l)}(\oo)$ such that $(B')^2 = I_{(n - 2l)}$ and $\| B - B' \| \leq \ee$. Then the desired matrix is:
\[
A' \coloneqq P^{-1} \begin{pmatrix}
0 & I_l & 0 \\
I_l & 0 & 0 \\
0 & 0 & B'
\end{pmatrix} P
\]

Assume finally that $l = 0$, and so $A = I + \uni^k M$ for some $M \in \MM_n(\oo)$, where we take $k \geq 1$ to be maximal. If $r^k \leq \ee$ we may take $A' = I$ and conclude. Otherwise:
\[\ee^2 \geq \| A^2 - I \| = \| I + \uni^{2k} M^2 - I \| = r^{2k} \| M^2 \|;\]
so $\| (I + M)^2 - I \| = \| M^2 \| \leq (r^{-k} \ee)^2 < 1$. The choice of $k$ implies that $(I + M)$ is not congruent to the identity modulo $\pp$, so $(I + M)$ falls in the previous case. Therefore there exists $M' \in \MM_n(\oo)$ such that $(I + M')^2 = I$ and $\| (I + M) - (I + M') \| \leq r^{-k} \ee$. Then $A' \coloneqq I + \uni^k M'$ satisfies $(A')^2 = I$ and
\[\| A - A' \| = \| \uni^k (M - M') \| = r^k \| (I + M) - (I + M') \| \leq r^k \cdot r^{-k} \ee = \ee. \qedhere\]
\end{proof}

The fundamental tool for the proof of Proposition \ref{prop:z2z} was Theorem \ref{thm:BG}, which provides simple representatives for each conjugacy class of representations $\mathbb{Z}/2\mathbb{Z} \to \GGL_n(\oo/\pp^k)$. A similar result for other finite $p$-groups could similarly be used to prove stability. However the group $\mathbb{Z}/2\mathbb{Z}$ is really special from this point of view. Indeed, it is shown in \cite{wild} that for every other finite $p$-group $\Gamma$ and every commutative local ring $\mathbf{R}$ of characteristic a power of $p$, the analogous problem for representations $\Gamma \to \GGL_n(\mathbf{R})$ is \emph{computationally wild}. More precisely, this problem contains the problem of simultaneous similarity classes of pairs of matrices over the residue field $\mathbf{R}/m$, which is a long-standing open problem in its general form (see e.g. \cite{similarity}). Therefore, although the few known special cases could be used to prove further stability results, the general solution to Question \ref{q:fin} needs a different approach.

\pagebreak

\section{Examples of uniform instability}
\label{s:unst}

So far we have only seen positive results for uniform stability: the only negative results were given in Corollary \ref{cor:cex_stab} and concerned pointwise stability. In this section, we construct groups which are not uniformly $\G$-stable, for some of the families $\G$ which were considered so far. This will also prove that some of our previous positive results were sharp. 

\medskip

The following notation will be useful for the rest of this section:

\begin{definition}
\label{def:Hdist}

Let $G$ be a metric group and let $\f \colon \Gamma \to G$ be a map. We define
\[\Hdist \coloneqq \inf\{\dist(\f, \ff) : \ff \colon \Gamma \to G \text{ is a homomorphism} \}.\]
\end{definition}

Therefore, proving that $\Gamma$ is not uniformly $\G$-stable amounts to exhibiting a sequence $(\f_n \colon \Gamma \to G_n \in \G)_{n \geq 1}$ such that $\defe(\f_n) \to 0$ and $\Hdist(\f_n)$ is uniformly bounded away from $0$.

\begin{notation}
For the rest of this section, we fix a non-Archimedean local field $\K$ with ring of integers $\oo$, uniformiser $\uni$, maximal ideal $\pp = \uni \oo$, residue field $\kk$ of characteristic $p$, and value group $r^{\mathbb{Z}} = |\K^\times|$, where $r = |\uni| \in (0, 1)$.
\end{notation}

\subsection{Instability for finite groups}

We start with a simple example to show that, for every prime $p$, there exists a pro-$p$ family $\G$ such that $\mathbb{Z}/p\mathbb{Z}$ is not $\G$-stable. The construction builds on an observation of Kazhdan \cite[Proposition 1]{amenst}.

\begin{example}
\label{ex:2unst}

For each $n \geq 1$, let $d_n$ be the metric on $\mathbb{Z}_p$ defined by the non-Archimedean norm $|\cdot|_p^n$, let $G_n$ denote $\mathbb{Z}_p$ equipped with this metric and $\G$ the corresponding family, which is a pro-$p$ family. We claim that $\mathbb{Z}/p\mathbb{Z}$ is not $\G$-stable. Indeed, let $\f_n \colon \mathbb{Z}/p\mathbb{Z} \to \mathbb{Z}_p$ be the map sending $k \mod p$ to $k$, for $0 \leq k < p$. Since $\mathbb{Z}_p$ is torsion-free, the only homomorphism $\mathbb{Z}/p\mathbb{Z} \to \mathbb{Z}_p$ is the trivial one, which is at a distance $1$ from $\f_n$. On the other hand $\defe(\f_n) = d_n(p, 0) = p^{-n} \to 0$ (we only need to check $d_n(p, 0)$ by Lemma \ref{lem:ultra_lift}, considering the standard presentation $\langle s \mid s^p = 1 \rangle$ of $\mathbb{Z}/p\mathbb{Z}$). Thus $(\f_n \colon \mathbb{Z}/p\mathbb{Z} \to G_n)_{n \geq 1}$ is a uniform asymptotic homomorphism that is not uniformly asymptotically close to a homomorphism.
\end{example}

This justifies why in general arguments we cannot reduce to looking at asymptotic homomorphisms $(\f_n \colon \Gamma \to G_n)_{n \geq 1}$ such that $\f_n(g^{-1}) = \f_n(g)^{-1}$ (see Remark \ref{rem:2unst}). Moreover, it shows that in the results from Section \ref{s:vpropi} the hypotheses of $p$-freeness are necessary. 

\medskip

However the family $\G$ is constructed ad-hoc to make stability fail. In the next sections we will be concerned with examples of groups which are not uniformly $\GGL(\oo)$-stable, providing the counterpart to the positive results from Sections \ref{s:vpropi} and \ref{s:char0}.

\subsection{Countable groups that are not uniformly $\GGL(\oo)$-stable}

In this subsection, we provide an example of a countable $p$-group that is not uniformly $\GGL(\oo)$-stable; in fact, not even uniformly $\{ \GGL_1(\oo) \}$-stable. This will be the starting point to construct finitely generated examples in the next subsection. We will also see that this group is pointwise $\GGL(\oo)$-stable when $\K$ has characteristic $0$, which implies that the hypothesis of finite generation is necessary in Theorem \ref{thm:pw_un}. 

\medskip

The key to the construction is the fact that the estimates for $\mathbb{Z}/p^i \mathbb{Z}$ get worse as $i \to \infty$. In order to prove it, we need the following fact about $p^i$-th roots of unity in $\K$:

\begin{lemma}
\label{lem:roots}

There exists $x \in \pp$ and $\ee = \ee(\K) > 0$ with the following property: for every $i \geq 1$, every finite extension $\K' | \K$, and every $p^i$-th root of unity $\xi \in \K'$, it holds $|(1 + x) - \xi| > \ee.$
\end{lemma}

\begin{remark}
For the proof of Proposition \ref{prop:countable:unst}, only the case $\K' = \K$ is needed. The general case will be used in the proof of Theorem \ref{thm:fg:unst}.
\end{remark}

\begin{proof}[Proof of Lemma \ref{lem:roots}]
Suppose first that $\K$ has characteristic $p$. Then the only $p^i$-th root of unity in any algebraic extension of $\K$ is $1$. Indeed:
\[x^{p^i} = 1 \Rightarrow x^{p^i} - 1 = 0 \Rightarrow (x - 1)^{p^i} = 0 \Rightarrow x = 1.\]
Therefore choosing $0 \neq x \in \pp$ and $0 < \ee < |x|$ works.

Now suppose that $\K$ has characteristic $0$. By Ostrowski's Theorem, it is a finite extension of $\Qp$, so $\K^{alg} = \Qp^{alg}$. Moreover, there exists a continuous homomorphism
$\log \colon (\Qp^{alg})^\times \to \Qp^{alg}$, such that $\log(1 + p\Zp) = p\Zp$ (see e.g. \cite[Propositions 5.4 and 5.5]{ANT}). Since $\log$ is a homomorphism, if $\xi$ is a root of unity in some finite extension of $\K$, then $\log(\xi) = 0$. Since $\log$ is continuous, $\log^{-1}(\{0\})$ is a closed subset of $\Qp^{alg}$. Therefore whenever $x \in \pp$ is such that $\log(1 + x) \neq 0$, if $\ee$ is the distance between $(1 + x)$ and $\log^{-1}(\{ 0 \}) \subset \Qp^{alg}$, then $\ee > 0$ and the desired property holds.
\end{proof}

This allows to produce examples of badly approximated maps $\mathbb{Z} / p^i \mathbb{Z} \to \oo^\times$:

\begin{lemma}
\label{lem:badestimates}

Let
\[\f_i \colon \mathbb{Z}/p^i \mathbb{Z} \to \GGL_1(\mathfrak{o}) : (k \mod p^i) \mapsto (1 + x)^k;\]
where $x \in \pp$. Then $\defe(\f_i) \to 0$.
Moreover, $x$ may be chosen so that $\Hdist(\f_i) > \ee$ for all $i$, where $\ee > 0$ is independent of $i$.
\end{lemma}

\begin{proof}
By Lemma \ref{lem:ultra_lift}, considering the standard presentation $\langle s \mid s^{p^i} = 1 \rangle$ of $\mathbb{Z}/p^i\mathbb{Z}$, to compute $\defe(\f_i)$ it suffices to compute $|(1 + x)^{p^i} - 1|$. If $\K$ has characteristic $p$, then
\[|(1 + x)^{p^i} - 1| = |x^{p^i}| \leq r^{p^i} \to 0.\]
Instead, suppose that $\K$ has characteristic $0$, and let $y$ be any element with $|y| \leq r$. Then
\[|(1 + y)^p - 1| = \left| y^p + \sum\limits_{k = 1}^{p - 1} {p \choose k} y^k \right| \leq \max\{ |y|^p, |py| \} \leq r \cdot |y|,\]
because for all $0 < k < p$, the binomial coefficient ${p \choose k}$ is divisible by $p$. It follows by induction that
\[|(1 + x)^{p^i} - 1| \leq r^i \cdot |x| \to 0.\]
This shows the first statement.

A homomorphism $\ff \colon \mathbb{Z}/p^i\mathbb{Z} \to \GGL_1(\oo) = \oo^\times$ must send $1$ to a $p^i$-th root of unity in $\K$. Therefore the second statement follows from Lemma \ref{lem:roots}.
\end{proof}

We can now use this to construct our first example of a group which is not uniformly $\GGL(\oo)$-stable. In characteristic $0$ it is also pointwise $\GGL(\oo)$-stable, which proves the remaining part of Proposition \ref{intro:prop:sharp}.

\begin{proposition}
\label{prop:countable:unst}

Let $\Gamma \coloneqq \bigoplus_{i \geq 1}  \mathbb{Z}/p^i\mathbb{Z}$. Then $\Gamma$ is not uniformly $\{\GGL_1(\oo)\}$-stable.
If moreover $\K$ has characteristic $0$, then $\Gamma$ is pointwise $\GGL(\oo)$-stable.
\end{proposition}

\begin{proof}
Let $\f_i \colon \mathbb{Z}/p^i\mathbb{Z} \to \GGL_1(\oo)$ be the maps from Lemma \ref{lem:badestimates}. Let $\Phi_i \colon \Gamma \to \GGL_1(\mathfrak{o})$ be the composition of the projection onto $\mathbb{Z}/p^i\mathbb{Z}$ with $\f_i$. Then $\defe(\Phi_i) \leq \defe(\f_i) \to 0$.

Now if $\Psi \colon \Gamma \to \GGL_1(\oo)$ is a homomorphism, we have
\[\dist(\Phi_i, \Psi) \geq \dist(\Phi_i|_{\mathbb{Z}/p^i\mathbb{Z}}, \Psi|_{\mathbb{Z}/p^i\mathbb{Z}}) \geq \Hdist(\f_i) \geq \ee.\]
Therefore $\Hdist(\Phi_i) \geq \ee$ for all $i$, and we conclude that $\Gamma$ is not uniformly $\{ \GGL_1(\oo) \}$-stable.

The pointwise stability in characteristic $0$ follows from Corollary \ref{cor:dirsum}.
\end{proof}

\begin{remark}
\label{rem:unfin}
This result also shows that Proposition \ref{prop:unfin} is sharp, in the sense that an infinitely generated group need not be uniformly stable with respect to finite families.

Moreover, since $\Gamma$ is a $p$-group, it is $q$-free for every $q \neq p$, and so Proposition \ref{prop:pifree} shows that $\Gamma$ is uniformly $\GGL(\oo)$-stable whenever $\K$ has residual characteristic $q$. Thus, there exist groups which are uniformly $\GGL(\oo)$-stable in some, but not all, residual characteristics, answering positively Question 9.4 from the first version of this paper.
\end{remark}

We also record the following stronger version of Lemma \ref{lem:badestimates}, which will be needed in the proof of Theorem \ref{thm:fg:unst} in the next subsection. It states that the lower bound on $\Hdist(\f_i)$ is preserved even for matrix groups of larger degree.

\begin{lemma}
\label{lem:badestimates2}

Let again
\[\f_i \colon \mathbb{Z}/p^i \mathbb{Z} \to \GGL_1(\mathfrak{o}) : (k \mod p^i) \mapsto (1 + x)^k;\]
where $x \in \pp$. Let $n \geq 1$, and define $\diag_n \colon \GGL_1(\oo) \to \GGL_n(\oo) : \lambda \mapsto \lambda I_n$. Then $\defe(\diag_n \circ \f_i) \to 0$.
Moreover, $x$ may be chosen so that $\Hdist(\diag_n \circ \f_i) > \ee$ for all $i$, where $\ee > 0$ is independent of $i$ and $n$.
\end{lemma}

\begin{proof}
The fact that $\defe(\diag_n \circ \f_i) \leq \defe(\f_i) \to 0$ follows immediately from Lemma \ref{lem:badestimates}. Now let $A \coloneqq \diag_n \circ \f_i(1) = (1 + x)I_n$. It remains to show that $x$ can be chosen in such a way that $\| A - \Xi \| \geq \ee$, whenever $\Xi \in \GGL_n(\oo)$ satisfies $\Xi^{p^i} = I_n$, and $\ee > 0$ should be independent of $i$ and $n$. We choose $x$ as in Lemma \ref{lem:roots}, so that there exists $\ee > 0$ such that for every finite extension $\K'|\K$ and every $p^i$-th root of unity $\xi \in \K'$ it holds $|(1 + x) - \xi| > \ee$.

Let $\Xi$ be such a matrix, and let $\K' | \K$ be a splitting field of the characteristic polynomial of $\Xi$, which we equip with the unique norm that extends the one on $\K$. Every eigenvalue $\xi \in \K'$ of $\Xi$ satisfies $\xi^{p^i} = 1$. Moreover, if $v \in (\K')^n$ is an eigenvector of $\xi$ of norm $1$, then
\[\| A - \Xi \| \geq \| Av - \Xi v \| = |(1 + x) - \xi| \cdot \| v \| = |(1 + x) - \xi| > \ee.\]
where in the first inequality we used Lemma \ref{lem:GL}.
\end{proof}

\begin{remark}
Using this lemma, the same proof as Proposition \ref{prop:countable:unst} shows that $\Gamma$ is not uniformly $\{ \GGL_n(\oo) \}$-stable, for any $n \geq 1$.
\end{remark}

\subsection{Finitely generated groups that are not uniformly $\GGL(\oo)$-stable}

In this subsection, we turn the example from Proposition \ref{prop:countable:unst} into a finitely generated one. The idea is to embed the group $\bigoplus_{i \geq 1} \mathbb{Z}/p^i\mathbb{Z}$ into a finitely generated group, in such a way that the asymptotic homomorphism contradicting stability extends to the larger group. The construction takes inspiration from the embedding of countable groups into finitely generated groups by means of wreath products \cite{embedding}, but modified in a way that preserves residual finiteness. Similar constructions have appeared in \cite{embeddingrf0, embeddingrf2, embeddingrf3}, but the idea can be traced back to the work of B. H. Neumann \cite{embeddingrf1} (curiously, B. H. Neumann's examples are pointwise stable in permutations \cite{bhstable} and for the Hilbert--Schmidt norm \cite{diagonalproducts}). 

\medskip

We fix a local field $\K$ of residual characteristic $p$, and keep the notations from the previous subsection. To simplify we will denote $C_k \coloneqq \mathbb{Z}/k\mathbb{Z}$ for the rest of this subsection. We will be working with the groups $C_{p^i} \wr C_{p^i} = (C_{p^i})^{p^i} \rtimes C_{p^i}$, where the group on the right acts on the direct sum by shifting the coordinates. More precisely, if $x = ((x_0, x_1, \ldots, x_{p^i - 2}, x_{p^i - 1}), 0)$ and $z = ((0, \ldots, 0), 1)$, then $z x z^{-1} = ((x_1, x_2, \ldots, x_{p^i - 1}, x_0), 0)$. We denote by $\Delta^1 \colon C_{p^i} \to C_{p^i} \wr C_{p^i}$ the diagonal inclusion, namely $\Delta^1(x) = ((x, \ldots, x), 0)$. 

\medskip

Let $P \coloneqq \prod\limits_{i \geq 1} C_{p^i} \wr C_{p^i}$, and let $\pr_i \colon P \to C_{p^i} \wr C_{p^i}$ denote the natural projection. For each $i$, we define the following element
\[\gamma_i \coloneqq (-(0, 1, 2, \ldots, p^i - 1), 0) \in C_{p^i} \wr C_{p^i}.\]
We identify $\gamma_i$ with an element of $P$ naturally, by setting all other coordinates to be the identity. We also consider the element
\[\zeta_i \coloneqq ((0, \ldots, 0), 1) \in C_{p^i} \wr C_{p^i}\]
and define $\zeta \in P$ by $\pr_i(\zeta) = \zeta_i$ for all $i$.

Define further $\kappa_i \coloneqq [\gamma_i, \zeta_i] \in C_{p^i} \wr C_{p^i}$. Explicitly:
\[\kappa_i = [\gamma_i, \zeta_i] = \gamma_i \cdot \zeta_i \gamma_i^{-1} \zeta_i^{-1} = (-(0, 1, \ldots, p^i - 2, p^i - 1), 0) \cdot ((1, 2, \ldots, p^i - 1, 0), 0) = \Delta^1(1).\]
Again, we identify $\kappa_i$ with an element of $P$, by setting all other coordinates to be the identity. This way $\kappa_i = [\gamma_i, \zeta]$: indeed, this holds at the $i$-th coordinate by definition, and the other coordinates are trivial since they are given by a commutator with the identity. Therefore $\kappa_i \in P$ generates the copy of $C_{p^i}$ given by the diagonal maps $\Delta^1 \colon C_{p^i} \to C_{p^i} \wr C_{p^i}$ at the $i$-th coordinate, and the identity elsewhere. 

\medskip

For the next step, we will need to use sequences of integers with some special properties:

\begin{claim}
\label{claim:riti}

There exists a sequence of integers $r_i \to \infty, i \geq 1$ and $0 < t_i < r_i$ with the following properties.
\begin{enumerate}
\item For each $i$, the integers $0, t_i, -t_i$ are pairwise distinct modulo $r_i$.
\item For each $i \neq j$, the integers $0, t_j, -t_i, (t_j - t_i)$ are pairwise distinct modulo $r_j$.
\end{enumerate}
\end{claim}

\begin{proof}
We choose $r_i \coloneqq 2^{2i}$ and $t_i \coloneqq 2^i - 1$. Then for all $i, j$ it holds:
\[t_j \equiv 3 \mod 4; \quad - t_i \equiv 1 \mod 4; \quad t_j - t_i \equiv 0 \mod 4.\]
Since $4$ divides $r_j$, these elements are also distinct modulo $r_j$, whether $i \neq j$ or not.

It remains to show that $t_j - t_i$ and $0$ are distinct modulo $r_j$, whenever $i \neq j$. This amounts to showing that $2^j$ and $2^i$ are distinct modulo $2^{2j}$. Suppose that this is not the case, so there exists an integer $k$ such that such that $2^i = 2^j + k 2^{2j} = 2^j(1 + k 2^j)$. Dividing both sides by $2^j$ we obtain that $2^{i-j} = 1 + k 2^j$ is an odd integer, which is only possible if $i = j$.
\end{proof}

Fix sequences $r_i, t_i$ as in Claim \ref{claim:riti}. Let $Q \coloneqq \prod\limits_{i \geq 1} (C_{p^i} \wr C_{p^i}) \wr C_{r_i}$, and again denote by $\pr_i \colon Q \to (C_{p^i} \wr C_{p^i}) \wr C_{r_i}$ the natural projection. The inclusion
\[C_{p^i} \wr C_{p^i} \to (C_{p^i} \wr C_{p^i}) \wr C_{r_i} : x \mapsto ((x, \id, \ldots, \id), 0)\]
realises $P$ naturally as a subgroup of $Q$. We have another diagonal inclusion of $C_{p^i}$, as follows:
\[\Delta^2 \colon C_{p^i} \to (C_{p^i} \wr C_{p^i}) \wr C_{r_i} : x \mapsto ((\Delta^1(x), \ldots, \Delta^1(x)), 0).\]

We define:
\[\varrho_i \coloneqq ((\zeta_i, \id, \ldots, \id, \gamma_i, \id, \ldots, \id), 0) \in (C_{p^i} \wr C_{p^i}) \wr C_{r_i} \qquad \text{and} \qquad \varrho \in Q : \pr_i(\varrho) = \varrho_i;\]
where $\zeta_i$ sits at the $0$-th coordinate, and $\gamma_i$ sits at the $t_i$-th coordinate (recall that $0 < t_i < r_i$). Finally, we define
\[\eta_i \coloneqq ((\id, \ldots, \id), 1) \in (C_{p^i} \wr C_{p^i}) \wr C_{r_i} \qquad \text{and} \qquad \eta \in Q : \pr_i(\eta) = \eta_i.\]

\begin{claim}
$[\eta^{t_i} \varrho \eta^{-t_i}, \varrho] = \kappa_i \in Q$ for each $i$.
\end{claim}

\begin{proof}
We verify this by checking each coordinate:
\[\pr_j([\eta^{t_i} \varrho \eta^{-t_i}, \varrho]) = [\eta_j^{t_i} \varrho_j \eta_j^{-t_i}, \varrho_j].\]

Suppose that $i \neq j$. Both $\eta_j^{t_i} \varrho_j \eta_j^{-t_i}$ and $\varrho_j$ are elements of $(C_{p^j} \wr C_{p^j})^{r_j} \leq (C_{p^j} \wr C_{p^j}) \wr C_{r_j}$. The former has non-identity entries only at its $-t_i$-th and $(t_j - t_i)$-th coordinates; while the latter has non-identity entries only at its $0$-th and $t_j$-th coordinates, modulo $r_j$. By Claim \ref{claim:riti}, at each coordinate at least one of the two elements is the identity, and thus the commutator is trivial.

Suppose that $i = j$. Again, both $\eta_i^{t_i} \varrho_i \eta_i^{-t_i}$ and $\varrho_i$ are elements of $(C_{p^i} \wr C_{p^i})^{r_i} \leq (C_{p^i} \wr C_{p^i}) \wr C_{r_i}$. The former has $\gamma_i$ as its $0$-th coordinate, $\zeta_i$ as its $-t_i$-th coordinate, and identities elsewhere; while the latter has $\zeta_i$ as its $0$-th coordinate, $\gamma_i$ as its $t_i$-th coordinate, and identities elsewhere, modulo $r_i$. By Claim \ref{claim:riti}, the $0$-th coordinate is $[\gamma_i, \zeta_i] = \kappa_i$; while at all other coordinates at least one of the elements is the identity, and so the commutator is trivial.
\end{proof}

Moreover, we have
\[\prod\limits_{j = 0}^{r_i - 1} \eta_i^j \kappa_i \eta_i^{-j} = ((\Delta^1(1), \id, \ldots, \id), 0) \cdots ((\id, \ldots, \id, \Delta^1(1)), 0) = \Delta^2(1) \in (C_{p^i} \wr C_{p^i}) \wr C_{r_i}.\]
Call $\delta_i$ this element, and identify it with an element of $Q$, as usual, so that $\delta_i = \prod\limits_{j = 0}^{r_i - 1} \eta^j \kappa_i \eta^{-j}$. We have thus proved the following:

\begin{proposition}
\label{prop:def:unst}

There exists a sequence of integers $r_i \to \infty$ and a finitely generated group $\Gamma = \langle \varrho, \eta \rangle \leq \prod\limits_{i \geq 1} (C_{p^i} \wr C_{p^i}) \wr C_{r_i}$ which for each $i$ contains the element $\delta_i$ defined by
\[\pr_j(\delta_i) = 
\begin{cases}
\id \in (C_{p^j} \wr C_{p^j}) \wr C_{r_j} & \text{ if } i \neq j, \\
\Delta^2(1) \in (C_{p^i} \wr C_{p^i}) \wr C_{r_i} & \text{ if } i = j;
\end{cases}\]
where
\[\Delta^2 \colon C_{p^i} \to (C_{p^i} \wr C_{p^i}) \wr C_{r_i} : x \mapsto ((\Delta^1(x), \ldots, \Delta^1(x)), 0).\]
\end{proposition}

We can now prove that the group $\Gamma$ is the desired counterexample to uniform stability.

\begin{theorem}
\label{thm:fg:unst}

Let $\Gamma$ be as above, and let $\K$ be a local field of residual characteristic $p$. Then $\Gamma$ is not uniformly $\GGL(\oo)$-stable.
\end{theorem}

\begin{proof}
Recall from Lemma \ref{lem:badestimates} that there exists a sequence $\f_i \colon C_{p^i} \to \oo^\times$ such that $\defe(\f_i) \to 0$ but $\Hdist(\f_i) \geq \ee$, for a fixed $\ee > 0$. Moreover, by Lemma \ref{lem:badestimates2}, the same estimates hold for $\diag_n \circ \f_i$, where $\diag_n \colon \GGL_1(\oo) \to \GGL_n(\oo) : \lambda \mapsto \lambda I_n$. Crucially, $\ee$ is independent of $i$ and $n$.

We extend $\f_i$ to a map $\f^1_i \colon C_{p^i} \wr C_{p^i} \to \GGL_{p^i}(\oo)$ by sending $(x_0, \ldots, x_{p^i - 1}) \in (C_{p^i})^{p^i}$ to the diagonal matrix with entries $\f_i(x_0), \ldots, \f_i(x_{p^i - 1})$; and the generator of the acting group to the permutation matrix corresponding to the cycle $(1 2 \cdots p^i)$. We take the following presentation of $C_{p^i} \wr C_{p^i}$:
\[\langle a, b \mid a^{p^i} = 1, b^{p^i} = 1, [b^j a b^{-j}, a] = 1 : 0 < j < p^i \rangle.\]
The map $\f_i^1$ satisfies all relations except the first one, which is controlled by $\f_i$, and therefore $\defe(\f^1_i) \leq \defe(\f_i) \to 0$. Note that
\[\f_i^1(\kappa_i) = \f_i^1 \circ \Delta^1(1) = \diag_{p^i} \circ \f_i(1).\]
We extend $\f_i^1$ further to a map
\[\f_i^2 \colon (C_{p^i} \wr C_{p^i}) \wr C_{r_i} \to \GGL_{p^i \cdot r_i}(\oo)\]
by sending $(x_0, \ldots, x_{r_i - 1}) \in (C_{p^i} \wr C_{p^i})^{r_i}$ to the block-diagonal matrix whose blocks are the square $p^i$ matrices $\f_i^1(x_0), \ldots, \f_i^1(x_{r_i - 1})$; and the generator of $C_{r_i}$ to the permutation matrix corresponding to $(1 2 \cdots p^{p^i \cdot r_i})^{p^i}$, which permutes cyclically the $r_i$ blocks. Then, by the same argument as above, $\defe(\f^2_i) \leq \defe(\f^1_i) \to 0$, and
\[\f_i^2(\delta_i) = \f_i^2 \circ \Delta^2(1) = \diag_{p^i \cdot r_i} \circ \f_i(1).\]

We define $\Phi_i \colon \Gamma \to \GGL_{p^i \cdot r_i}(\oo)$ as the composition of the projection onto $(C_{p^i} \wr C_{p^i}) \wr C_{r_i}$ with the map $\f_i^2$ above. Then $\defe(\Phi_i) \leq \defe(\f_i^2) \to 0$. Now if $\Psi \colon \Gamma \to \GGL_{p^i \cdot r_i}(\oo)$ is a homomorphism, restricting $\Psi$ to $\langle \delta_i \rangle \leq \Gamma$ provides a homomorphism of $C_{p^i}$ which must be $\ee$-far from $\f_i^2 \circ \Delta^2 = \diag_{p^i \cdot r_i} \circ \f_i \colon C_{p^i} \to \GGL_{p^i}(\oo)$ by Lemma \ref{lem:badestimates2}. In particular $\Hdist(\Phi_i) > \ee$, and we conclude that $\Gamma$ is not uniformly $\GGL(\oo)$-stable.
\end{proof}

\begin{remark}
The existence of a finitely generated group which is not uniformly stable answers positively Question 9.1 from the first version of this paper.

By Theorem \ref{thm:pw_un}, since $\Gamma$ is finitely generated, it is not pointwise $\GGL(\oo)$-stable either. Moreover, being a subgroup of a direct product of finite groups, $\Gamma$ is residually finite. Thus, this is our first example of a finitely generated residually finite group that is not pointwise $\GGL(\oo)$-stable, answering positively Question 9.2 from the first version of this paper.
\end{remark}

\pagebreak

\section{Stability via (bounded) cohomology}
\label{s:BC}

The goal of this section is to give criteria for $\GGL(\oo)$-stability of finitely presented groups in terms of (bounded) cohomology. The approach follows that in \cite{GLT}, and the reason we restrict to finitely presented groups is the same: the ultraproduct techniques work best for pointwise stability, but the quantitative approach needs a quantity that controls all local defects, and this is only possible for finitely presented groups. The ultrametric nature of our setting will produce \emph{bounded} cocycles, and so the bounded cohomological approach is the more natural one to take in this case. Over non-Archimedean fields, cohomology vanishing is, a priori, stronger than bounded cohomology vanishing (Proposition \ref{prop:comp}); this will imply the cohomological analogue of the statements.

The boundedness of the cocycles is a consequence of the fact that pointwise asymptotic homomorphisms are asymptotically close to uniform ones for finitely presented groups, by Item $2.$ of Proposition \ref{prop:ultra_asy}. So the reader may suspect that a bounded cohomological criterion for uniform stability should also hold in the Archimedean setting. However, in that case the situation is more delicate and requires the introduction of a different cohomology theory, called \emph{asymptotic cohomology} \cite{Bharat}. The advantage of our setting is that, while cocycles are bounded thanks to the uniform nature of the problem, we can still use the same ultraproduct techniques that apply to the pointwise setting of \cite{GLT}.

\medskip

We state the correspondence between asymptotic homomorphisms and bounded cohomology classes (Proposition \ref{prop:defdim:BC}) by means of the recently introduced \emph{property of defect diminishing} \cite{defdim}. This leads to a criterion for stability in terms of (bounded) cohomology vanishing (Corollary \ref{cor:BC:abs}), analogous to the one of \cite{GLT}. This criterion does not seem as widely applicable as in the Archimedean case, and we are only able to apply it to virtually free groups, whose stability was already shown in Corollaries \ref{cor:vfree_p} and \ref{cor:vfree}. At the end of this section, we discuss why this is the case, and propose a stronger bounded cohomological criterion that could be more suitable for applications (Question \ref{q:dual}).

However, phrasing the correspondence precisely allows to also deduce a relative version of the criterion (Corollary \ref{cor:BC:rel}). We use this to show Theorem \ref{intro:thm:fi} from the introduction: for finitely presented groups, uniform stability with a linear estimate passes to finite-index supergroups and quotients by certain normal subgroups. In particular, we obtain many results from Section \ref{s:char0} as corollaries of the corresponding ones of Section \ref{s:vpropi}, and produce new examples.

\begin{notation}
For the rest of this section, we fix a non-Archimedean local field $\K$ with ring of integers $\oo$, uniformiser $\uni$, maximal ideal $\pp = \uni \oo$, residue field $\kk$ of characteristic $p$, and value group $r^{\mathbb{Z}} = |\K^\times|$, where $r = |\uni| \in (0, 1)$.
\end{notation}

\subsection{Bounded cohomology}

We review here the basics of bounded cohomology of discrete groups that are needed for the rest of the section. For more information, see \cite{monod, Frig} for bounded cohomology over the reals, and \cite{mio} for bounded cohomology over non-Archimedean fields. We refer the reader back to Subsection \ref{ss:nona} for the basic notions from non-Archimedean functional analysis.

\medskip

We will work with the bar resolution throughout, since it is the easiest one to treat lifting problems with. Let $\Gamma$ be a group, $E$ a $\K[\Gamma]$-module, without a specified norm. Let
\[\CC^n(\Gamma; E) \coloneqq \{ f \colon \Gamma^n \to E \},\]
which is a $\K$-vector space with pointwise addition and scalar multiplication. Define the coboundary map $\delta^n \colon \CC^n(\Gamma; E) \to \CC^{n+1}(\Gamma; E)$ by the formula:
\begin{align*}
\delta^n(f)(g_1, \ldots, g_{n+1}) &\coloneqq g_1 \cdot f(g_2, \ldots, g_{n+1}) \\
& + \sum\limits_{i = 1}^n (-1)^i f(g_1, \ldots, g_i g_{i+1}, \ldots, g_{n+1}) \\
& + (-1)^{n+1} f(g_1, \ldots, g_n).
\end{align*}
This defines a cochain complex of $\K$-vector spaces $(\CC^\bullet(\Gamma; E), \delta^\bullet)$: we denote the cocycles by $\ZZ^\bullet(\Gamma; E)$, the coboundaries by $\BB^\bullet(\Gamma; E)$, and the \emph{cohomology} by $\HH^\bullet(\Gamma; E) \coloneqq \ZZ^\bullet(\Gamma; E) / \BB^\bullet(\Gamma; E)$. 

\medskip

Now suppose that $E$ is a normed $\K[\Gamma]$-module: recall that this means that $E$ is a normed $\K$-vector space equipped with a linear isometric action of $\Gamma$. Let
\[\CC^n_b(\Gamma; E) \coloneqq \{ f \colon \Gamma^n \to E : \| f \|_\infty < \infty \} \subseteq \CC^n(\Gamma; E),\]
which is a normed $\K$-vector space with the $\ell^\infty$-norm $\| \cdot \|_\infty$, and even a Banach space if $E$ is also Banach. With the same coboundary map, we obtain the cochain complex $(\CC^\bullet_b(\Gamma; E), \delta^\bullet)$, the bounded cocycles $\ZZ^\bullet_b(\Gamma; E)$, the bounded coboundaries $\BB^\bullet_b(\Gamma; E)$, and the \emph{bounded cohomology} $\HH^\bullet_b(\Gamma; E) \coloneqq \ZZ^\bullet_b(\Gamma; E) / \BB^\bullet_b(\Gamma; E)$. 

\medskip

Both bounded and ordinary cohomology are functorial: a homomorphism $f \colon \Lambda \to \Gamma$ induces \emph{pullbacks} in bounded and ordinary cohomology. These are the maps $f^*_{(b)} \colon \HH^n_{(b)}(\Gamma; E) \to \HH^n_{(b)}(\Lambda; E)$ obtained by precomposing cocycles on $\Gamma$ with $f$, where a $\K[\Gamma]$-module is turned into a $\K[\Lambda]$-module by pulling back the action via $f$. When $f$ is the inclusion of a subgroup, we call this the \emph{restriction} and denote it by $\res^\Gamma_\Lambda$.

The inclusion $\CC^n_b(\Gamma; E) \hookrightarrow \CC^n(\Gamma; E)$ is a chain map, that induces the \emph{comparison map}
\[c^n \colon \HH^n_b(\Gamma; E) \to \HH^n(\Gamma; E).\]
The kernel of this map, called \emph{exact bounded cohomology}, is denoted by $\EH^n_b(\Gamma, E)$. Over the reals it is very rich and interesting, even in the case where $n = 2$ and $E$ is the trivial normed $\mathbb{R}[\Gamma]$-module $\mathbb{R}$, leading to the theory of \emph{quasimorphisms} \cite[Chapter 2]{Frig}. However, over non-Archimedean fields, the exact bounded cohomology in degree $2$ is trivial for finitely generated groups:

\begin{proposition}[{\cite[Corollary 8.7]{mio}}]
\label{prop:comp}

Let $\Gamma$ be a finitely generated group, $E$ a normed $\K[\Gamma]$-module. Then the comparison map $c^2 \colon \HH^2_b(\Gamma; E) \to \HH^2(\Gamma; E)$ is injective.
\end{proposition}

This implies that for such groups, the injectivity of a pullback in cohomology implies the injectivity of the corresponding pullback in bounded cohomology:

\begin{corollary}
\label{cor:pullback}

Let $\Gamma$ be a finitely generated group, $f \colon \Lambda \to \Gamma$ a homomorphism, and $E$ a normed $\K[\Gamma]$-module. Suppose that $f^* \colon \HH^2(\Gamma; E) \to \HH^2(\Lambda; E)$ is injective. Then $f^*_b \colon \HH^2_b(\Gamma; E) \to \HH^2_b(\Lambda; E)$ is also injective.
\end{corollary}

\begin{proof}
We have the following commutative diagram:
\[\begin{tikzcd}
	{\HH^2_b(\Gamma;E)} && {\HH^2_b(\Lambda;E)} \\
	\\
	{\HH^2(\Gamma;E)} && {\HH^2(\Lambda;E)}
	\arrow["{f^*_b}", from=1-1, to=1-3]
	\arrow["{f^*}"', hook, from=3-1, to=3-3]
	\arrow["{c^2}", from=1-3, to=3-3]
	\arrow["{c^2}"', hook, from=1-1, to=3-1]
\end{tikzcd}\]
where the left vertical arrow is injective by Proposition \ref{prop:comp}, and the bottom horizontal arrow is injective by hypothesis. Thus the top horizontal arrow must also be injective.
\end{proof}

\subsection{Diminishing the defect}

Lemma \ref{lem:stab_up} rephrases stability in terms of a lifting property. However the kernel in this lifting problem is not tractable with cohomology: it is not even abelian. The goal of this subsection is to show that a weaker quantitative statement (intuitively: every asymptotic homomorphism is close to one with smaller defect) can be related to a simpler lifting problem, where the kernel is more approachable and will be analysed in the next subsection. 

\medskip

Let us fix some terminology and notation concerning ultrafilters. Fix a free ultrafilter $\omega$. We say that an event $E_n$ holds for \emph{most} $n$ if $\{ n : E_n \text{ holds}\} \in \omega$. Accordingly we denote $\en \neq_\omega 0, \dd_n \leq_\omega \en$, and so on. Given sequences $\en >_\omega 0$ and $\dd_n \geq_\omega 0$, we write $\dd_n = O_\omega(\en)$ if there exists $C \geq 0$ such that $\dd_n \leq_\omega C \en$. The infimal such $C$ can be characterised as $\lim\limits_{n \to \omega} \dd_n/\en$: this limit makes sense since $\en \neq_\omega 0$, and it is finite since $\dd_n/\en \in [0, C]$ for most $n$. If this limit is $0$, we write $\dd_n = o_\omega(\en)$. 

\medskip

Fix a finite presentation $\Gamma = \langle S \mid R \rangle$. We will show stability in terms of asymptotic homomorphisms at the level of the free group $F_S$: see Proposition \ref{prop:stab_equiv}. For our purposes, it will be more convenient to replace usual convergence by convergence along $\omega$. In this setting, the following notion, introduced in \cite{GLT} and formalised in \cite{defdim}, allows to show stability with a linear estimate:

\begin{definition}
\label{def:defdim}

Let $\Gamma = \langle S \mid R \rangle$ be a finitely presented group, let $\G$ be a family of groups equipped with bi-invariant metrics, and let $(\hf_n \colon F_S \to G_n \in \G)_{n \geq 1}$ be a sequence with $\defe(\hf_n) \xrightarrow{n \to \omega} 0$. We say that a sequence $(\hff_n \colon F_S \to G_n)_{n \geq 1}$ \emph{diminishes the defect} of $(\hf_n)_{n \geq 1}$ if $\defe(\hff_n) = o_\omega(\defe(\hf_n))$ and $\dist(\hf_n, \hff_n) = O_\omega(\defe(\hf_n))$.

If such a $(\hff_n)_{n \geq 1}$ exists, we say that $(\hf_n)_{n \geq 1}$ has the \emph{property of defect diminishing}. If this is true for every $(\hf_n)_{n \geq 1}$, then $\Gamma$ is said to have the \emph{property of defect diminishing with respect to $\G$}.
\end{definition}

\begin{theorem}[{\cite[Theorem 3.5]{defdim}}]
Let $\Gamma$ be a finitely presented group, $\G$ a family of groups equipped with bi-invariant \emph{complete} metrics. Then $\Gamma$ has the property of defect diminishing with respect to $\G$ if and only if it is pointwise $\G$-stable with a linear estimate.
\end{theorem}

Here the linear estimate is to be intended for pointwise $\G$-stability in terms of presentations \cite[Subsection 2.1]{defdim}. However, we have seen that when $\G$ is an ultrametric family and $\Gamma$ is a finitely presented group, pointwise and uniform $\G$-stability of $\Gamma$ are equivalent (Theorem \ref{thm:pw_un}), and moreover the two quantitative notions also coincide (Remark \ref{rem:quant:fp}). Thus:

\begin{corollary}
\label{cor:defdim:ultra}

Let $\Gamma$ be a finitely presented group, $\G$ an ultrametric family whose metric groups are complete. Then $\Gamma$ has the property of defect diminishing with respect to $\G$ if and only if it is uniformly $\G$-stable with a linear estimate.
\end{corollary}

The family $\GGL(\oo)$ satisfies the hypotheses. Our goal is to show that the property of defect diminishing can be established via the vanishing of certain bounded cohomology classes (Proposition \ref{prop:defdim:BC}). To this end, we start by relating the property of defect diminishing to a lifting problem. 

\medskip 

From now on, let $G_n \coloneqq \GGL_{k_n}(\oo)$, and fix a sequence $(\hf_n \colon F_S \to G_n)_{n \geq 1}$ with $\en \coloneqq \defe(\hf_n) \xrightarrow{n \to \omega} 0$. Since the metric on $G_n$ takes values in $r^{\mathbb{Z}}$, the same holds for the sequence $\en$. Moreover we may assume that $\en >_\omega 0$, since otherwise $(\hf_n)_{n \geq 1}$ is already asymptotically close to a homomorphism that descends to $\Gamma$. The asymptotic homomorphism $(\hf_n)_{n \geq 1}$ induces a homomorphism to the metric ultraproduct of $(G_n)_{n \geq 1}$ (see Lemma \ref{lem:stab_up}). Here we will use a modified ultraproduct that takes into account the sequence $(\en)_{n \geq 1}$ as well. In analogy with notation which will shortly be introduced, we denote $\pi G \coloneqq \prod\limits_{n \geq 1} G_n$.

For a sequence $\dd_n \xrightarrow{n \to \omega} 0$, denote
\[N(\dd_n) \coloneqq \{ (A_n)_{n \geq 1} \in \pi G : \| A_n - I_{k_n} \| \leq_\omega \dd_n \},\]
and similarly $N(O_\omega(\dd_n))$ and $N(o_\omega(\dd_n))$.

\begin{lemma}
\label{lem:Ndn:normalsubgroup}
$N(\dd_n), N(O_\omega(\dd_n))$ and $N(o_\omega(\dd_n))$ are normal subgroups of $\pi G$.
\end{lemma}

\begin{remark}
The fact that $N(\dd_n)$ is a subgroup relies strongly on the ultrametric inequality, and will allow to streamline a few arguments, and to make them more quantitatively precise. This leads to a phrasing of the problems in terms of bounded cohomology instead of cohomology. By contrast, in the Archimedean case \cite{GLT} one needs to work with $N(O_\omega(\dd_n))$.
\end{remark}

\begin{proof}[Proof of Lemma \ref{lem:Ndn:normalsubgroup}]
We prove the statement for $N(\dd_n)$ (which is the only one specific to the ultrametric case), the rest is similar. Suppose that $(A_n)_{n \geq 1}, (B_n)_{n \geq 1} \in N(\dd_n)$. Then
\[\| A_n B_n - I_{k_n} \| \leq \max \{ \| A_n B_n - B_n \|, \| B_n - I_{k_n} \| \} \leq_\omega \dd_n,\]
where we have used that $\| \cdot \|$ is right-invariant. If now $(C_n)_{n \geq 1} \in \pi G$, then
\[\| C_n A_n C_n^{-1} - I_{k_n} \| = \| A_n - I_{k_n} \| \leq_\omega \dd_n. \qedhere\]
\end{proof}

We denote $\pi_{\dd_n}(G) \coloneqq \pi G /N(\dd_n)$, and similarly $\pi_{O_{\omega} (\dd_n)} G$ and $\pi_{o_{\omega} (\dd_n)} G$. Given a constant $C \geq 1$, we have $\en \leq_\omega C \en$, so $N(o_\omega(\en)) \leq N(\en) \leq N(C \en)$. Therefore there are quotient maps $\pi_{o_\omega(\en)} G \to \pi_{\en} G \to \pi_{C \en} G$. The sequence $(\hf_n)_{n \geq 1}$ induces a homomorphism $\f(\en) \colon \Gamma \to \pi_{\en} G$, as well as a homomorphism $\f(C \en) \colon \Gamma \to \pi_{C \en} G$. This yields the following lifting problem similar to the one from Lemma \ref{lem:cohopk}:
\[\begin{tikzcd}
	& {\pi_{o_\omega(\en)}G} \\
	\Gamma & {\pi_{\en}G} \\
	& {\pi_{C\en}G}
	\arrow["\ff", dashed, from=2-1, to=1-2]
	\arrow["{\f(C \en)}"', from=2-1, to=3-2]
	\arrow["{\f(\en)}", from=2-1, to=2-2]
	\arrow[from=2-2, to=3-2]
	\arrow[from=1-2, to=3-2, bend left = 50]
\end{tikzcd}\]

The solution to the above lifting problem implies (a more precise version of) the property of defect diminishing $(\hf_n)_{n \geq 1}$:

\begin{lemma}
\label{lem:red_lift}

The existence of a solution $\ff$ to the above lifting problem is equivalent to the existence of a sequence $(\hff_n \colon F_S \to G_n)_{n \geq 1}$ such that $\defe(\hff_n) = o_\omega(\en)$ and $\dist(\hf_n, \hff_n) \leq_\omega C \en$.
\end{lemma}

\begin{proof}
The proof is the same as that of \cite[Theorem 4.2]{a:comm} (see Lemma \ref{lem:stab_up}).
\end{proof}

We denote by $E_C \coloneqq N(C \en) / N(o_\omega(\en)) = \ker \{ \pi_{o_\omega(\en)}G \to \pi_{C \en} G \}$. As in Subsection \ref{ss:preli:split}, the above lifting problem reduces to a splitting problem
\[1 \to E_C \to \left( \pi_{o_\omega(\en)}G \times_{\f(C \en)} \Gamma \right) \to \Gamma \to 1\]
and in turn, if $E_C$ is abelian, to a cohomology vanishing problem with coefficients in $E_C$. Moreover, as in Lemma \ref{lem:cohopk}, the particular form of this lifting problem implies that the relevant cocycles take values in $N_1 = N(\en) / N(o_\omega(\en)) = \ker \{ \pi_{o_\omega(\en)}G \to \pi_{\en}G \}$.

\subsection{Additional structures on the kernel}
\label{ss:kerbanach}

In this subsection we show that not only is $E_C$ abelian, which allows to apply cohomology to the previous splitting problem, but moreover that it is the closed $C$-ball of a Banach $\K[\Gamma]$-module. This Banach $\K[\Gamma]$-module will be $E \coloneqq N(O_\omega(\en)) / N(o_\omega(\en))$. We can characterise $N(O_\omega(\en))$ as 
\[N(O_\omega(\en)) \coloneqq \bigcup\limits_{C \geq 1} N(C \en) = \left\{ (A_n)_{n \geq 1} \in \pi G : \lim\limits_{n \to \omega} \frac{\| A_n - I_{k_n} \|}{\en} < \infty \right\}.\]
Here we are using that $\en >_\omega 0$, and that a bounded sequence always admits an $\omega$-limit. We will denote by $A = (A_n)_{n \geq 1}$ elements of $\pi G$, and if $A \in N(O_\omega(\en))$ we will denote by $[A]$ its image in $E$. We also denote by $I \coloneqq (I_{k_n})_{n \geq 1}$ the identity of $\pi G$. 

\medskip

We start by showing that $E$ is a $\K$-vector space. It already has a well-defined product structure, but to underline that it is abelian we will denote it by $[A] + [B] \coloneqq [AB]$. The scalar multiplication will be given by a stretch fixing the identity, namely $\lambda [A] \coloneqq [\lambda A + (1 - \lambda) I] = [I + \lambda(A - I)]$.

\begin{lemma}
With these operations, $E$ is a $\K$-vector space.
\end{lemma}

\begin{proof}
We already know that $(E, +)$ is a group. It is moreover abelian: for this we need to show that given $A, B \in N(O_\omega(\en))$, it holds $ABA^{-1}B^{-1} \in N(o_\omega(\en))$. This follows from the submultiplicativity of the norm $\| \cdot \|$ (see Lemma \ref{lem:GL}):
\begin{align*}
\|A_n B_n A_n^{-1} B_n^{-1} - I_{k_n}\| &= \|A_n B_n - B_n A_n\| \\
&= \|(A_n - I_{k_n})(B_n - I_{k_n}) - (B_n - I_{k_n})(A_n - I_{k_n})\| \\
&\leq \|A_n - I_{k_n}\| \cdot \|B_n - I_{k_n}\| = O_\omega(\ee_n^2) = o_\omega(\en).
\end{align*}

The scalar multiplication is well-defined: indeed
\[\lambda A_n + (1 - \lambda) I_{k_n} = I_{k_n} + \lambda(A_n - I_{k_n}),\]
is close to the identity, thus invertible, when $\|A_n - I_{k_n}\|$ is small enough. It is also easy to see that it is still in $N(O_\omega(\en))$, and that $(\lambda \mu) [A] = \lambda (\mu [A])$. Finally we prove bilinearity of the scalar multiplication. First, given $\lambda, \mu \in \K, [A] \in N(O_\omega(\en))$, we have:
\begin{align*}
\lambda \cdot [A] + \mu \cdot [A] &= [I + \lambda(A - I)] + [I + \mu(A - I)] \\
&= [I + (\lambda + \mu)(A - I) + \lambda \mu (A - I)^2] = (\lambda + \mu)[A],
\end{align*}
where we used that $\| (A_n - I_n)^2 \| = o_\omega(\en)$. Similarly, given $\lambda \in \K, [A], [B] \in N(O_\omega(\en))$, we have:
\begin{align*}
\lambda \cdot [A] + \lambda \cdot [B] &= [\lambda(A - I) + I] + [\lambda(B - I) + I] \\
&= [\lambda^2 (A - I)(B - I) + \lambda (A - I) + \lambda(B - I) + I] \\
&= [\lambda(A + B - 2I) + I] = [-\lambda(A - I)(B - I) + \lambda(AB - I) + I] \\
&= [\lambda(AB - I) + I] = \lambda \cdot [AB]. \qedhere
\end{align*}
\end{proof}

Now since $E$ is abelian, the splitting problem after Lemma \ref{lem:red_lift} gives an action of $\Gamma$ on $E$ by conjugacy. Concretely, this action is defined on the free group by $w \cdot [A] = [(\hf_n(w) A_n \hf_n(w)^{-1})_{n \geq 1}]$, and it descends to $\Gamma$ since $\defe(\hf_n) \leq_\omega \en$. Next, we introduce a norm on $E$, namely
\[\| [A] \|_E \coloneqq \lim\limits_{n \to \omega} \frac{\| A_n - I_{k_n} \|}{\en}.\]

\begin{lemma}
$\| \cdot \|_E$ is a solid Banach norm on $E$, and the action of $\Gamma$ is isometric.
\end{lemma}

\begin{proof}
$\| \cdot \|_E$ is well-defined on $N(O_\omega(\en))$ and it is zero precisely on $N(o_\omega(\en))$, by very definition of these two spaces. Now
\[\frac{\| A_n B_n - I_{k_n} \|}{\en} \leq \frac{\max \{ \|A_n B_n - B_n\|, \|B_n - I_{k_n}\| \} }{\en} = \max \left\{ \frac{\|A_n - I_{k_n}\|}{\en}, \frac{\|B_n - I_{k_n}\|}{\en} \right\},\]
and taking the $\omega$-limit shows that $\| AB \|_E \leq \max\{ \| A \|_E, \| B \|_E \}$. This implies at once that $\| \cdot \|_E$ is well-defined on $E$, and that it satisfies the ultrametric inequality. That it is $\K$-multiplicative is clear. It takes values in $r^\mathbb{Z}$ because the $\ell^\infty$-norm does, so $\| \cdot \|_E$ is a solid norm on $E$. Since $\Gamma$ acts by conjugation by elements of $\GGL_{k_n}(\oo)$, and the $\ell^\infty$-norm is bi-invariant, the action is isometric. 

\medskip

We are left to show that $E$ is Banach. So let $\{ [A^k] = [(A^k_n)_{n \geq 1}] \}_{k \geq 1}$ be a Cauchy sequence. Explicitly, this means that
\[\| [A^k] - [A^l] \|_E = \lim\limits_{n \to \omega} \frac{\|A^k_n (A^l_n)^{-1} - I_{k_n} \|}{\en} = \lim\limits_{n \to \omega} \frac{\|A^k_n - A^l_n\|}{\en} \xrightarrow{k, l \to \infty} 0.\]
The product space $\pi G$ is compact, so up to subsequence we may assume that a sequence of representatives $A^k$ converges in this topology to some $A \in \pi G$. This means pointwise convergence, that is: $\|A^k_n - A_n\| \xrightarrow{k \to \infty} 0$ for all $n$, although the convergence is not necessarily uniform in $n$. We need to show that $A \in N(O_\omega(\en))$ and that $\|[A^k] - [A]\|_E \xrightarrow{k \to \infty} 0$.

Let $\delta > 0$ be fixed. We will show that there exists some $m = m(\delta) \geq 1$ such that for all $k \geq m$ and for most $n$ we have: $\frac{\|A^k_n - A_n\|}{\en} < \delta$. Then, an application of the triangle inequality shows that $A \in N(O_\omega(\en))$, and $\omega$-convergence follows by letting $\delta \to 0$. We choose $m$ to be such that $\|[A^k] - [A^l]\|_E < \delta$ for all $k, l \geq m$. By definition of the $\omega$-limit it follows that
\[X \coloneqq \left\{ n \geq 1 : \en \neq 0, \, \frac{\|A^k_n - A^l_n\|}{\en} < \delta \, \, \, \forall \, k, l \geq m \right\} \in \omega.\]
For all $n \in X$, let $m_n$ be such that $\|A^l_n - A_n\| < \delta \en$ for all $l \geq m_n$ (we can do this because $n \in X$ and so $\en \neq 0$). Then for all $k \geq m$, given $n \in X$ and choosing $l \geq \max \{ m_n, m \}$ we have:
\[\frac{\|A^k_n - A_n \|}{\en} \leq \max \left\{ \frac{\|A^k_n - A^l_n \|}{\en} , \frac{\|A^l_n - A_n \|}{\en} \right\} < \max \left\{ \delta, \frac{\delta \en}{\en} \right\} = \delta.\]
This concludes the proof.
\end{proof}

By very definition:

\begin{lemma}
\label{lem:ker_ball}

$E_C$ is the closed $C$-ball in $E$ with respect to the norm $\| \cdot \|_E$.
\end{lemma}

\begin{remark}
\label{rem:matrix:ultraproduct}

The Banach $\K[\Gamma]$-module $E$ is isometrically $\Gamma$-isomorphic to another Banach $\K[\Gamma]$-module which admits a nicer description, namely the matrix ultraproduct $\prod\limits_{n \to \omega} \MM_{k_n}(\K)$ with the $\ell^\infty$-norm. This is the quotient of the subspace of bounded sequences in the direct product by the subspace of sequences $(M_n)$ such that $\| M_n \| \xrightarrow{n \to \omega} 0$, and can be endowed with a natural norm $\| [M] \| = \lim\limits_{n \to \omega} \| M_n \|$ and a $\Gamma$-action by conjugacy via $(\hf_n)_{n \geq 1}$. Then the isometric $\Gamma$-isomorphism is given by
\[\prod\limits_{n \to \omega} M_{k_n}(\K) \to E : [M = (M_n)_{n \geq 1}] \mapsto [(I_n + \uni^{-k(n)} M_n)],\]
where $k(n)$ is such that $\ee_n = r^{k(n)}$.
We will not be making use of this alternative description.
\end{remark}

\subsection{Stability via vanishing}

We can finally prove the correspondence between asymptotic homomorphisms and bounded cohomology classes.

\begin{proposition}
\label{prop:defdim:BC}

Let $\Gamma = \langle S \mid R \rangle$ be a finitely presented group. Then for every sequence $(\hf_n \colon F_S \to G_n)_{n \geq 1}$ with $\defe(\hf_n) \xrightarrow{n \to \omega} 0$, there exists a Banach $\K[\Gamma]$-module with a solid norm $E$, and a class $\alpha \in \HH^2_b(\Gamma; E)$, which vanishes if and only if $(\hf_n)_{n \geq 1}$ has the defect diminishing property.
\end{proposition}

\begin{proof}
Let $\ee_n \coloneqq \defe(\hf_n)$. Recall that $(\hf_n)_{n \geq 1}$ is said to have the defect diminishing property if there exists $(\hff_n)_{n \geq 1}$ such that $\defe(\hff_n) = o_\omega(\ee_n)$ and $\dist(\hf_n, \hff_n) = O_\omega(\ee_n)$, or equivalently $\dist(\hf_n, \hff_n) \leq_\omega C \ee_n$ for some $C \geq 1$. By Lemma \ref{lem:red_lift}, this is equivalent to the existence of a solution to the following lifting problem:
\[\begin{tikzcd}
	& {\pi_{o_\omega(\en)}G} \\
	\Gamma & {\pi_{\en}G} \\
	& {\pi_{C\en}G}
	\arrow["\ff", dashed, from=2-1, to=1-2]
	\arrow["{\f(C \en)}"', from=2-1, to=3-2]
	\arrow["{\f(\en)}", from=2-1, to=2-2]
	\arrow[from=2-2, to=3-2]
	\arrow[from=1-2, to=3-2, bend left = 50]
\end{tikzcd}\]

We let $E, E_C$ be as in the previous subsection, and recall that $E$ is a Banach $\K[\Gamma]$-module with a solid norm, and that $E_C$ is the closed $C$-ball in $E$. From the discussion following Lemma \ref{lem:red_lift}, we can construct a cocycle $z \colon \Gamma^2 \to E_1$ whose class in $\HH^2_b(\Gamma; E_C)$ vanishes if and only if there exists a solution to the above lifting problem. Now by very definition, this occurs if and only if there exists $b \colon \Gamma \to E$ such that $b$ takes values in $E_C$ and $\delta^1(b) = z$. Therefore this holds for \emph{some} $C$, if and only if the class $\alpha \in \HH^2_b(\Gamma; E)$ represented by $z \colon \Gamma^2 \to E$ vanishes.
\end{proof}

We immediately obtain:

\begin{corollary}
\label{cor:BC:abs}

Let $\Gamma$ be a finitely presented group, and suppose that $\HH^2_b(\Gamma; E) = 0$ for every Banach $\K[\Gamma]$-module $E$ with a solid norm. Then $\Gamma$ is uniformly $\GGL(\oo)$-stable, with a linear estimate.
In particular, this holds if $\HH^2(\Gamma; E) = 0$ for every such $E$.
\end{corollary}

\begin{proof}
The bounded-cohomological statement follows from Proposition \ref{prop:defdim:BC} and Corollary \ref{cor:defdim:ultra}. The cohomological one then follows from Proposition \ref{prop:comp}.
\end{proof}

\begin{example}
The vanishing of $H^2$ in Corollary \ref{cor:BC:abs} holds for finitely generated virtually free groups, when $\K$ has characteristic $0$, and for finitely generated virtually free groups without elements of order $p$, when $\K$ has characteristic $p$. Indeed, these are precisely the finitely generated groups of $\K$-cohomological dimension at most $1$ (and they are automatically finitely presented) \cite{cd1}. Thus we recover Corollaries \ref{cor:vfree_p} and \ref{cor:vfree}.
\end{example}

With a little more work, we also obtain a relative version of the above:

\begin{corollary}
\label{cor:BC:rel}

Let $\Lambda, \Gamma$ be finitely presented groups, where $\Lambda$ is uniformly $\GGL(\oo)$-stable with a linear estimate, and $f \colon \Lambda \to \Gamma$ a homomorphism. Suppose that $f^*_b \colon \HH^2_b(\Gamma; E) \to \HH^2_b(\Lambda; E)$ is injective for every Banach $\K[\Gamma]$-module $E$ with a solid norm. Then $\Gamma$ is also uniformly $\GGL(\oo)$-stable with a linear estimate.
In particular, this holds if $f^* \colon \HH^2(\Gamma; E) \to \HH^2(\Lambda; E)$ is injective for every such $E$. When $f \colon \Lambda \to \Gamma = \Lambda/N$ is a quotient, this is implied by the vanishing of $\HH^1(N; \K)$.
\end{corollary}

\begin{proof}
Let $\langle S' \mid R' \rangle$ be a finite presentation of $\Lambda$, and extend it to a finite presentation $\langle S \mid R \rangle$ of $\Gamma$, with $S' \subset S$ and $R' \subset R$, so that the following diagram commutes:
\[\begin{tikzcd}
	{F_{S'}} & {F_S} \\
	\Lambda & \Gamma
	\arrow[hook, from=1-1, to=1-2]
	\arrow[two heads, from=1-1, to=2-1]
	\arrow[two heads, from=1-2, to=2-2]
	\arrow["f"', from=2-1, to=2-2]
\end{tikzcd}\]
This way Proposition \ref{prop:stab_equiv} establishes a correspondence between uniform asymptotic homomorphisms of $\Gamma$ and sequences $(\hf_n \colon F_S \to G_n)_{n \geq 1}$, and uniform asymptotic homomorphisms of $\Lambda$ and sequences $(\hf_n' \colon F_{S'} \to G_n)_{n \geq 1}$. Moreover, by the choice of the presentations, precomposing a uniform asymptotic homomorphism of $\Gamma$ by $f$ to obtain one of $\Lambda$ corresponds to restricting a sequence $(\hf_n)_{n \geq 1}$ from $F_S$ to $F_{S'}$.

Lemma \ref{lem:red_lift} translates the property of defect diminishing to a lifting problem, which in turns reduces to a splitting problem: see Subsection \ref{ss:preli:split}. Restricting the sequence $(\hf_n \colon F_S \to G_n)_{n \geq 1}$ to $(\hf_n' \colon F_{S'} \to G_n)_{n \geq 1}$ corresponds to precomposing the lifting problem by $f$, which in turn reduces to pulling back the splitting problem to $\Lambda$. The two splitting problems now fit into a commutative diagram:
\[\begin{tikzcd}
	1 & {E_C} & {\left( \pi_{o_\omega(\en)}G \times_{\f(C \en) \circ f} \Lambda \right)} & \Lambda & 1 \\
	1 & {E_C} & {\left( \pi_{o_\omega(\en)}G \times_{\f(C \en)} \Gamma \right)} & \Gamma & 1
	\arrow[from=1-1, to=1-2]
	\arrow[from=2-1, to=2-2]
	\arrow[from=1-2, to=1-3]
	\arrow[from=2-2, to=2-3]
	\arrow[from=1-3, to=1-4]
	\arrow[from=1-4, to=1-5]
	\arrow[from=2-4, to=2-5]
	\arrow[from=2-3, to=2-4]
	\arrow["f"', from=1-4, to=2-4]
	\arrow[from=1-2, to=2-2]
	\arrow["{(\id, f)}"', from=1-3, to=2-3]
\end{tikzcd}\]
A section $\Gamma \to \left( \pi_{o_\omega(\en)}G \times_{\f(C \en)} \Gamma \right)$ is of the form $g \mapsto (\sigma(g), g)$. This produces a section
\[\Lambda \to \left( \pi_{o_\omega(\en)}G \times_{\f(C \en) \circ f} \Lambda \right) : g \mapsto (\sigma(f(g)), g).\]

Therefore, a cocycle representing the top-row extension can be chosen to be the pullback to $\Lambda$ of a cocycle representing the bottom-row extension. This shows that for a given sequence $(\hf_n \colon F_S \to G_n)_{n \geq 1}$ such that $\defe(\hf_n) \xrightarrow{n \to \omega} 0$, and the class $\alpha \in \HH^2_b(\Gamma; E)$ given by Proposition \ref{prop:defdim:BC}; the class corresponding to the restriction $(\hf_n' \colon F_{S'} \to G_n)_{n \geq 1}$ is simply $f^*_b(\alpha) \in \HH^2_b(\Lambda; E)$. By assumption $\Lambda$ is uniformly $\GGL(\oo)$-stable with a linear estimate, and so by Corollary \ref{cor:defdim:ultra} it has the property of defect diminishing. In particular $f^*_b(\alpha)$ must vanish, by Proposition \ref{prop:defdim:BC}. But $f^*_b$ is assumed to be injective, so $\alpha$ must vanish as well, showing that $(\hf_n)_{n \geq 1}$ has the property of defect diminishing, again by Proposition \ref{prop:defdim:BC}.

Thus, $\Gamma$ has the property of defect diminishing, and so it is uniformly $\GGL(\oo)$-stable with a linear estimate, again by Corollary \ref{cor:defdim:ultra}. This concludes the proof of the bounded-cohomological statement. The cohomological one then follows from Corollary \ref{cor:pullback}.

Finally, suppose that $f \colon \Lambda \to \Gamma = \Lambda/N$ is a quotient. The Lyndon--Hochschild--Serre Spectral Sequence \cite[VII.6]{Brown} yields a $5$-term exact sequence
\[0 \to \HH^1(\Gamma; E) \to \HH^1(\Lambda; E) \to \HH^1(N; E) \to \HH^2(\Gamma; E) \xrightarrow{f^*} \HH^2(\Lambda; E),\]
which shows that if $\HH^1(N; E) = 0$, then $f^*$ is injective. Since $E$ is a \emph{trivial} $\K[N]$-module, this is implied by $\HH^1(N; \K) = 0$.
\end{proof}

We apply Corollary \ref{cor:BC:rel} to obtain Theorem \ref{intro:thm:fi} from the introduction:

\begin{theorem}
\label{thm:fi}

Let $\Gamma$ be a finitely presented group, let $\Lambda \leq \Gamma$ be a finite-index subgroup, and let $N \leq \Gamma$ be a normal subgroup that is finitely normally generated.
\begin{enumerate}
\item Suppose that $[\Gamma : \Lambda]$ is not divisible by the characteristic of $\K$, and that $\Lambda$ is uniformly $\GGL(\oo)$-stable with a linear estimate. Then the same is true of $\Gamma$.
\item Suppose that $\HH^1(N; \K) = 0$, and that $\Gamma$ is uniformly $\GGL(\oo)$-stable with a linear estimate. Then the same is true of $\Gamma/N$.
\end{enumerate}
\end{theorem}

We recall that, in contrast with Ulam stability \cite[Lemma 2.2]{BOT}, uniform $\GGL(\oo)$-stability is not preserved under taking general quotients (Remark \ref{rem:quotients}).

\begin{proof}
$1.$ Since $\Lambda$ has finite index in $\Gamma$, it is also finitely presented. The hypothesis implies that $\res^\Gamma_\Lambda \colon \HH^n(\Gamma; E) \to \HH^n(\Lambda; E)$ is injective for every $n \geq 1$ and every $\K[\Gamma]$-module $E$ \cite[Proposition III.10.4]{Brown}. We conclude by Corollary \ref{cor:BC:rel}. 

\medskip

$2.$ Since $N$ is normally finitely generated, $\Gamma/N$ is also finitely presented. We conclude again by Corollary \ref{cor:BC:rel}.
\end{proof}

\begin{example}
\label{ex:fi:application}

This theorem allows to recover several results of the previous sections. In particular, some of the examples from Theorem \ref{intro:thm:vpfree} can be obtained by combining Theorem \ref{thm:fi} with the corresponding examples in Theorem \ref{intro:thm:pifree}. Of course this criterion is more flexible, and so applies further: as a very basic example, $\BS(2, 3) \times \mathbb{Z} / 2 \mathbb{Z}$ is stable in residual characteristic $3$.
\end{example}

\subsection{Applying the criteria}
\label{ss:conj}

The cohomological version of Corollary \ref{cor:BC:abs} is an analogue of the criterion in \cite{GLT}: a finitely presented group is $(\U(n), \| \cdot \|_{Frob})$-stable if $\HH^2(\Gamma; E) = 0$ for every unitary representation $E$. The so-called \emph{Garland Method}, initially introduced in \cite{Garland} and since then vastly generalised to encompass general Banach coefficients \cite{Garland2}, allows to give many examples of non-Archimedean lattices satisfying this condition (more recently, Archimedean lattices have been tackled by an alternative approach \cite{BS, BLSW}). Our criterion asks for vanishing over Banach spaces with a solid norm, which at a first glance may seem too restrictive compared to the Archimedean setting. However, on the one hand there is no clear analogue of Hilbert spaces over non-Archimedean fields \cite[2.4]{NFA}, and on the other hand the hypothesis of the norm being solid has strong implications: such spaces are isometrically classified \cite[Theorem 2.5.4]{NFA}. But even then it seems hard to prove a cohomology vanishing criterion by adapting Garland's Method, since there the distinction between positive and negative real eigenvalues plays a fundamental role. This is why we are unable to verify the cohomology vanishing condition in Corollary \ref{cor:BC:abs} for groups that are not virtually free. 

\medskip

The bounded cohomological statements are a priori stronger than the cohomological statements. But the hypothesis of $\Gamma$ being finitely presented may play an important role. Indeed, the comparison map $c^2 \colon \HH^2_b(\Gamma, \K) \to \HH^2(\Gamma, \K)$ is an isomorphism when $\Gamma$ is finitely presented, and $\K$ is seen as the trivial normed $\K[\Gamma]$-module \cite[Corollary 8.13]{mio}. Although we do not know whether surjectivity also holds with non-trivial coefficients, this seems to suggest that the bounded cohomological criterion of Corollary \ref{cor:BC:abs} also does not apply further than virtually free groups. 

\medskip

Looking at the real setting, one may hope that Corollary \ref{cor:BC:abs} could be strengthened by only asking for vanishing with \emph{dual} normed $\K[\Gamma]$-modules. For instance, in the real setting all amenable groups have vanishing bounded cohomology with dual coefficients \cite[Chapter 3]{Frig}, and in degree $2$ this even applies to high-rank lattices \cite{burmon, rigidity, monodshalom}. In our setting, there is a significant obstacle, namely that no infinite-dimensional $\K$-Banach space is reflexive \cite[Corollary 7.4.20]{NFA}, so proving that a Banach $\K[\Gamma]$-module is dual requires an explicit construction of a pre-dual. This seems hard considering that the spaces appearing in our setting are quite complicated (Remark \ref{rem:matrix:ultraproduct}). Using the classification of Banach spaces with a solid norm, and assuming that the degree $k_n \to \infty$ (else Proposition \ref{prop:unfin} applies), we are able to show that all spaces appearing in Proposition \ref{prop:defdim:BC} are isometrically isomorphic to $\ell^\infty(\mathbb{N})$, which is the dual of $c_0(\mathbb{N})$. So there is a chance that these spaces are dual $\K[\Gamma]$-modules, but to show that the action itself is dual, one would have to construct an explicit pre-dual. We formulate this as an open question:

\begin{question}
\label{q:dual}

Does Proposition \ref{prop:defdim:BC} hold when the statement is restricted to dual $\K[\Gamma]$-modules? Do Corollaries \ref{cor:BC:abs} and \ref{cor:BC:rel}?
\end{question}

Such a strengthening would then motivate the study of more general vanishing results in degree $2$ than those that are known so far \cite[Theorem 7.4, Corollary 7.13]{mio}, possibly by adapting the work of Burger, Monod and Shalom \cite{burmon, rigidity, monodshalom} to lattices in $\K$-analytic groups. 

\medskip

The relative statement (Corollary \ref{cor:BC:rel}) is more widely applicable as we have seen in Theorem \ref{thm:fi}. However it is again unclear whether further applications than those are possible. In the real case, the injectivity of the restriction $\res^\Gamma_\Lambda \colon \HH^2_b(\Gamma; E) \to \HH^2_b(\Lambda; E)$ holds more generally than just for finite-index subgroups, namely it holds for \emph{coamenable} subgroups, again assuming that $E$ is dual \cite[8.6]{monod} (see also \cite{coamenable}). However the analogous notion over non-Archimedean fields is likely to reduce to finite-index, at least for finitely generated groups, since this is the case for (normed) amenability \cite{Schik, mio}.

\pagebreak

\section{Further remarks and open questions}
\label{s:q}

In this section we survey some open questions on ultrametric stability, give a few partial answers, and propose directions for further research. The first two subsections contain open questions about $\GGL(\oo)$-stability of certain groups, with special attention to $\mathbb{Z}^2$. Subsection \ref{ss:q:fam} proposes other ultrametric families whose study may be of interest. We refer the reader to Subsections \ref{ss:poschar} and \ref{ss:conj} for further open questions, about stability of finite groups in positive characteristic and bounded cohomology, respectively.

\begin{notation}
For the rest of this section, we fix a non-Archimedean local field $\K$ with ring of integers $\oo$, uniformiser $\uni$, maximal ideal $\pp = \uni \oo$, residue field $\kk$ of characteristic $p$, and value group $r^{\mathbb{Z}} = |\K^\times|$, where $r = |\uni| \in (0, 1)$.
\end{notation}

\subsection{Finding more non-examples}

Most of this paper has been concerned with giving positive results on stability, with negative results in Subsection \ref{ss:unst} and Section \ref{s:unst}. These solve many of the questions on instability from the first version of this paper. However, a few questions remain open, the main one being the following.

\begin{question}
\label{q:fp:st}

Does there exist a finitely presented group that is not $\GGL(\oo)$-stable?
\end{question}

It would be very surprising if this question had a negative answer. Recall that in Section \ref{s:unst}, we started by constructing a countable unstable group, and then used a modified version of the embedding of countable groups inside finitely generated groups to promote it to a finitely generated example. A key feature of our construction is that it preserved residual finiteness. It is possible that a finitely presented example may be obtained from the group in Theorem \ref{thm:fg:unst} by means of a modified version of the Higman Embedding Theorem \cite{HET}. However such embedding construction do not combine well with residual finiteness \cite{HET:rf}.

More specific candidates seem to be free abelian groups, surface groups, free nilpotent groups and free solvable groups. The case of $\mathbb{Z}^2$ is discussed in detail in Subsection \ref{ss:q:z2}, the other ones are briefly mentioned after Example \ref{ex:z2:est}.

Other potential non-examples are given by graphs of groups as in Subsection \ref{ss:gog}, where the coprimality conditions on finite quotients are not satisfied. For instance, an interesting family of examples to look at is that of \emph{GBS-tree products}, that is, GBS groups whose underlying graph is a tree, to which the criteria from Corollaries \ref{cor:GBS_p} and \ref{cor:GBS} cannot apply. In particular:

\begin{question}
Which torus knot groups $K_{m, n}$ are $\GGL(\oo)$-stable? Does some condition on $\nu_p(m)$ and $\nu_p(n)$ imply stability, or instability?
\end{question}

All our examples of finitely generated groups that are uniformly but not pointwise stable use Corollary \ref{cor:cex_stab}, so they cannot be residually finite. Therefore we ask:

\begin{question}
\label{q:fgrf:unnonpw}

Does there exist a finitely generated residually finite group that is uniformly but not pointwise $\GGL(\oo)$-stable?
\end{question}

Recall from Corollary \ref{cor:fgrf} that if $\Gamma$ is a finitely generated residually finite group that can be expressed as the largest residually finite quotient of a finitely presented group, then the uniform and pointwise $\GGL(\oo)$-stability of $\Gamma$ are equivalent. We saw in the discussion after Corollary \ref{cor:fgrf} that not all finitely generated groups satisfy this. The examples given there, all coming from \cite{covers}, could provide a positive answer to Question \ref{q:fgrf:unnonpw}. 

\medskip

The proofs of Section \ref{s:approx} seem to suggest that by working on the space of marked groups one could adapt the results from Sections \ref{s:vpropi} and \ref{s:char0} to the pointwise setting.

\begin{question}
Do the results from Sections \ref{s:vpropi} and \ref{s:char0} have a pointwise counterpart?
\end{question}

However, this presents more technical subtleties, since one would have to prove lifting results for \emph{local} homomorphisms to the metric quotients: we are not aware of such results, or of a connection to cohomology analogous to the classical one (see Subsection \ref{ss:preli:split}). 

\medskip

Before moving on to more specific examples in the next subsections, let us comment on why a well-known method for producing non-examples of stability does not work for $\GGL(\oo)$. For instance when $\G$ is a family of unitary groups, one starts with $(n+1)$-dimensional irreducible unitary representations of the finitely generated group $\Gamma$, and restricts to the top $(n \times n)$ corner to obtain an asymptotic homomorphism. Assuming this is close to a homomorphism, one arrives at a contradiction with the irreducibility of the initial representation. The same idea works for permutations, by starting with a transitive action of $\Gamma$ on $\{ 1, \ldots, n\}$. To our knowledge this method was first used in \cite{T} to prove that if $\Gamma$ is infinite, sofic and has property (T), then it is not pointwise stable in permutations; this was later generalised \cite{us:local}. It also appears in \cite{BChap} and \cite{uHS}, where it is shown that uniform stability, in permutations and with respect to unitary groups with the Hilbert--Schmidt norm respectively, is a very restrictive property.

It is key in the arguments that by looking at $n$ out of $(n+1)$ entries (or $n^2$ out of $(n+1)^2$) one does not lose much in terms of normalised metrics. This cannot be the case for the $\ell^\infty$-norm on $\GGL(\oo)$, where a big difference in a single entry is detected as a big difference overall. 

\medskip

This projection trick is precisely the motivation behind introducing notions of \emph{flexible stability} \cite{T}, that have proved fruitful in some contexts \cite{surf}. The discussion above shows that it seems hard to define an analogous notion of flexible $\GGL(\oo)$-stability. Given the rigidity of this context, it is even possible that na\"ive definitions of flexible $\GGL(\oo)$-stability are equivalent to ordinary $\GGL(\oo)$-stability. 

\begin{question}
Is there a meaningful notion of flexible $\GGL(\oo)$-stability? Is it different from ordinary $\GGL(\oo)$-stability?
\end{question}

Note that the operator norm on $\U(n)$ also has the feature that modifying a small proportion of entries can lead to a big difference in norm. Nevertheless, shortly before publication of this paper, Lubotzky and Salomon showed that flexible stability is a meaningful notion in the operator norm: indeed $\ZZ^2$ is flexibly stable but not stable \cite{LS}.

\subsection{(In)stability of $\mathbb{Z}^2$}
\label{ss:q:z2}

The following is the main open question we would like to draw attention to:

\begin{question}
Is $\mathbb{Z}^2$ $\GGL(\oo)$-stable?
\end{question}

An answer for free abelian groups of arbitrary finite rank would be ideal, but we stick to $\mathbb{Z}^2$ for this discussion. An algebraic-geometric approach seems to be the most appropriate: indeed in \cite[Theorem D, Example 10.1]{Zordan} the author proves a result -- including an explicit example -- that suggests that $\mathbb{Z}^2$ is not \emph{constraint $\GGL(\oo)$-stable} with respect to a direct factor (see \cite{a:const} for the definition of constraint stability). The notions of stability and constraint stability can be quite distinct, for instance $\mathbb{Z}^2$ is stable in permutations \cite{a:comm} but not constraint stable with respect to a direct factor \cite{a:const}. Still, Zordan's result is the closest instance to a result on $\GGL(\oo)$-instability of $\mathbb{Z}^2$ that we were able to find in the literature. 

\medskip

Zordan's result is in terms of Lie algebras, not of Lie groups. This is equivalent to our setting: more precisely, instead of working with $\GGL_n(\oo)$, one could work with $\MM_n(\oo)$ and prove stability of the Lie bracket.

\begin{lemma}
\label{lem:z2:lin}

The following are equivalent:
\begin{enumerate}
\item $\mathbb{Z}^2$ is $\GGL(\oo)$-stable.
\item For all $\ee > 0$ there exists $\dd > 0$ such that if $A, B \in \MM_n(\oo)$ satisfy $\| AB - BA \| < \dd$, then there exist $A', B' \in \MM_n(\oo)$ such that $A'B' = B'A'$ and $\| A - A' \|, \|B - B'\| < \ee$.
\end{enumerate}
\end{lemma}

\begin{proof}
We use the characterisation from Corollary \ref{cor:quant}. Since $\| ABA^{-1}B^{-1} - I \| = \| AB - BA \|$ by $\GGL(\oo)$-invariance of $\| \cdot \|$ (Lemma \ref{lem:GL}), instability of $\mathbb{Z}^2$ implies that $2.$ does not hold. Conversely, given matrices $A, B$ contradicting $2.$ for a given $\ee > 0$, the matrices $(I + \uni A), (I + \uni B)$ contradict the characterisation of stability from Corollary \ref{cor:quant}, for a rescaled $\ee$. Indeed, they are invertible by Lemma \ref{lem:GL}, and
\[\| (I + \uni A)(I + \uni B) - (I + \uni B)(I + \uni A) \| = r^2 \| AB - BA \|. \qedhere\]
\end{proof}

Since the norm $\| \cdot \|$ on $\MM_n(\K)$ coincides with the operator norm (Lemma \ref{lem:GL}), it is tempting to try and adapt Voiculescu's counterexample \cite{Voie} to prove instability of $\mathbb{Z}^2$. Voiculescu's matrices proving that $\mathbb{Z}^2$ is not pointwise stable with respect to $\{ (\U(n), \| \cdot \|_{op}) : n \geq 1 \}$ are the \emph{Pauli shift and clock matrices}: the permutation matrix $P$ corresponding to the cycle $(1 \cdots n)$ and the diagonal matrix $D \coloneqq \diag( 1, \omega_n, \omega_n^2, \ldots, \omega_n^{n-1})$, where $\omega_n = e^{\frac{2 \pi i}{n}}$. It is easy to adapt this example to produce asymptotic homomorphisms from $\mathbb{Z}^2$ to $\GGL(\oo)$, but they are all close to homomorphisms. This is true even for an arbitrary matrix $D$, as long as $P$ is monomial (that is, it has a unique non-zero entry in each row and column):

\begin{lemma}
Let $P \in \GGL_n(\oo)$ be a monomial matrix, and let $D \in \MM_n(\K)$. Then there exists $D' \in \MM_n(\K)$ such that $PD' = D'P$ and $\| D - D' \| \leq \| PD - DP \|$.
\end{lemma}

\begin{proof}
Let $\ee \coloneqq \| PD - DP \| = \|PDP^{-1} - D \|$. Let $\sigma \in S_n$ be the permutation such that $P_{ij} \neq 0$ precisely when $j = \sigma(i)$, and set $\lambda_i \coloneqq P_{i \sigma(i)}$, which is in $\oo^\times$ since $P \in \GGL_n(\oo)$. Then $(PDP^{-1})_{ij} = \lambda_i \lambda_j^{-1} D_{\sigma(i)\sigma(j)}$, so $|D_{\sigma(i)\sigma(j)} - \lambda_i^{-1} \lambda_j D_{ij}| \leq \ee$ for each $(i, j) \in \{ 1, \ldots, n\}^2$. By induction it follows that
\[\left|D_{\sigma^k(i)\sigma^k(j)} - \left( \prod\limits_{l = 0}^{k-1} \lambda_{\sigma^l(i)}^{-1} \lambda_{\sigma^l(j)} \right) \cdot D_{ij}\right| \leq \ee\]
for every $k \geq 0$.

Now choose representatives for each orbit of the diagonal action of $\sigma$ on $\{1, \ldots, n\}^2$, and for each representative $(i, j)$ and each $k \geq 0$ smaller than the size of the corresponding orbit, set
\[D'_{\sigma^k(i) \sigma^k(j)} \coloneqq \left( \prod\limits_{l = 0}^{k-1} \lambda_{\sigma^l(i)}^{-1} \lambda_{\sigma^l(j)} \right) \cdot D_{ij}.\]
Then $\| D - D' \| \leq \ee$ and $PD' = D'P$.
\end{proof}

In particular if $D \in \GGL_n(\oo)$ and $\| PD - DP \| < 1$, then $\| D - D' \| < 1$ and so $D' \in \GGL_n(\oo)$ as well, by Lemma \ref{lem:GL}. This shows that such matrices produce asymptotic homomorphisms that are close to homomorphisms, even with an optimal estimate. Moreover there is no need to modify $P$, so this also does not even work as a counterexample to constraint $\GGL(\oo)$-stability of $\mathbb{Z}^2$. 

\medskip

Let us end by noticing that the stability estimate of $\mathbb{Z}^2$ is, at best, quadratic. So this example really is different from the ones treated in this paper, where the stability estimates were always at worst linear (with the exception of Subsection \ref{ss:poschar}).

\begin{example}
\label{ex:z2:est}

Let $A, B \in \MM_n(\oo)$ be such that $\| A \|, \| B \| \leq \ee < 1$, and consider the map $\f \colon F_2 \to \GGL_n(\oo)$ sending the generators to $(I + A)$ and $(I + B)$, which are invertible by Lemma \ref{lem:GL}. This homomorphism almost descends to $\mathbb{Z}^2$ with a defect of $\ee^2$:
\[\| (I + A)(I + B) - (I + B)(I + A) \| = \| AB - BA \| \leq \| A \| \cdot \| B \| \leq \ee^2.\]
On the other hand, if $A$ and $B$ are chosen so that $\| AB - BA \| = \ee^2$, then this homomorphism is at a distance at least $\ee$ from any homomorphism that descends to $\mathbb{Z}^2$. Indeed, suppose otherwise that there exist commuting matrices $M, N \in \GGL_n(\oo)$ such that $\| (I + A) - M \| < \ee$ and $\|(I+B) - N \| < \ee$. Writing $M = (I+A')$ and $N = (I+B')$, we have $\|A-A'\| < \ee, \|B - B'\| < \ee$ and $A'B' = B'A'$. Now
\[AB = (A - A')B + A'(B - B') + A' B'\]
and so $\| AB - A'B' \| < \ee^2$. Similarly $\| BA - B'A' \| < \ee^2$ and thus $\| AB - BA \| < \ee^2$, contradicting the choice of $A$ and $B$.
\end{example}

With the same idea one can show that free abelian groups and surface groups have at best quadratic estimates. For free nilpotent groups, applying the above argument inductively on the length of the lower central series, one can show that the estimate is at best polynomial, with the degree increasing together with the length. Similarly for free solvable groups with the length of the derived series. 

\medskip

Recall that all of the stability results proved in this paper were with a polynomial estimate. In fact, even with a linear estimate, with the exception Proposition \ref{prop:z2z}. For the groups mentioned above, we gave lower bounds on the stability rate, which is still polynomial. Therefore the following question is natural.

\begin{question}
Let $\G$ be an ultrametric family, and let $\Gamma$ be a group which is uniformly $\G$-stable. Is $\Gamma$ uniformly $\G$-stable with a polynomial estimate? What if $\G = \GGL(\oo)$, where $\oo$ is the ring of integers of a non-Archimedean local field?
\end{question}

In the case of characteristic $0$, one could even ask the question with ``linear'' instead of ``polynomial'' (Remark \ref{rem:quadratic:estimate}).

In the Archimedean case, all of the uniform stability results that we are aware of give a linear estimate \cite{BOT, BChap, uHS, Bharat, Bharatandi}. In the pointwise setting, finitely generated abelian groups are stable with a polynomial estimate, which is superlinear in the case of higher rank \cite{quant}. We are not aware of examples of finitely presented pointwise stable groups with a superpolynomial estimate. There is a framework for quantitative pointwise stability of finitely generated groups, and for that there exist examples of groups with arbitrarily bad stability estimates \cite{diagonalproducts, badstability}.

\subsection{Other ultrametric families}
\label{ss:q:fam}

Most of this paper was concerned with $\GGL(\oo)$-stability, where $\oo$ is the ring of integers of a non-Archimedean local field. The groups $\GGL_n(\oo)$ are compact because $\oo$ is, and compactness played an important role in our arguments, especially to have finiteness of the metric quotients. The general picture could be more complicated:

\begin{problem}
Study $\GGL(\oo)$-stability, where $\oo$ is the ring of integers of a (not necessarily local) non-Archimedean field with residual characteristic $p > 0$. How does it compare to the case of local fields? Does completeness play a role? Does spherical completeness?
\end{problem}

It is likely that some results from Section \ref{s:vpropi} and \ref{s:char0} carry over, assuming at least completeness. In the general case the residue field $\kk$ is not finite, but at least it has characteristic $p$, which makes it possible to recover some arguments. In case $\K$ has residual characteristic $0$, the analogy with local fields breaks down, so this is likely to need a separate study:

\begin{problem}
Study $\GGL(\oo)$-stability, where $\oo$ is the ring of integers of a non-Archimedean field with residual characteristic $0$. Does completeness play a role? Does spherical completeness?
\end{problem}

Another direction in which to generalise $\GGL(\oo)$ while retaining compactness is to look at other compact $\K$-analytic groups equipped with suitable bi-invariant ultrametrics. For instance, using a result of Segal \cite{Segal}, some of the stability results on graphs of groups from Section \ref{s:char0} could be generalised to every family of compact $p$-adic analytic groups equipped with a suitable metric.

\begin{problem}
Study $\G$-stability, for other families $\G$ of compact $\K$-analytic groups equipped with suitable bi-invariant ultrametrics.
\end{problem}

In the introduction we mentioned that the $\ell^\infty$-norm on $\MM_n(\K)$ has the special feature of being at once an ultrametric analogue of the operator norm, of the Frobenius norm, and of the Hilbert--Schmidt norm on $\U(n)$. There is a fourth norm on matrix groups that one could consider, namely the normalised rank, leading to the \emph{rank metric}: the corresponding approximable groups are called \emph{linear sofic} and are studied in \cite{a:lin}. The rank metric can also be defined on non-Archimedean fields, however it is not an ultrametric, so it does not fall in the framework of this paper. Therefore we ask:

\begin{question}
Does the rank metric admit an ultrametric analogue?
\end{question}

In another direction, the problem of stability with respect to the family $\Gal(K)$ from Example \ref{ex:gal:prof} remains an interesting open problem (see Remark \ref{rem:galois:problem}), even for specific fields $K$.

\begin{problem}
Study stability with respect to the family $\Gal(K)$.
\end{problem}

Another interesting family of Galois groups of a different flavour is $\G \coloneqq \{ (\Gal(\mathbb{Q}_p^{alg}/\mathbb{Q}_p), d_p) : p \text{ prime} \}$, where $d_p$ is a bi-invariant ultrametric obtained as in Example \ref{ex:prof}, with respect to a fixed sequence $\eee$ and some suitable choice of open finite-index normal subgroups of $\Gal(\mathbb{Q}_p^{alg}/\mathbb{Q}_p)$.

\begin{problem}
Study stability with respect to the family $\G$ above.
\end{problem}

Our first trivial example of ultrametric family was a family of discrete groups $\G$ equipped with discrete metrics (Example \ref{ex:discr}). We saw that stability is less interesting in this setting (Example \ref{ex:discr:stab}) but as we mentioned in the discussion after Example \ref{ex:discr}, probabilistic versions of stability with respect to such families appear often in property testing \cite{PT}. Therefore it would be interesting to develop a general framework of probabilistic ultrametric stability, analogously to what is done for the family $\{ (S_n, d_H) : n \geq 1 \}$ in \cite{BChap}, which could produce new results in property testing.

\begin{problem}
Study probabilistic analogues of ultrametric stability.
\end{problem}

Finally, it would be interesting to produce and study more examples of ultrametric families of finite groups. Beyond discrete families, the only case we treated is the family $T(\mathbf{R})$, where $\mathbf{R}$ is a finite commutative unital ring: we proved in Corollary \ref{cor:approxTR} a result concerning approximation with respect to such families, and stability was treated in Proposition \ref{prop:aut_stab} without the finiteness hypothesis. While we studied in detail an ultrametric analogue of $\U(n)$, it would be interesting to find an ultrametric analogue of $(S_n, d_H)$ and compare the corresponding stability problems.

\begin{problem}
Produce new examples of ultrametric families of finite groups, and study the corresponding stability problems.
\end{problem}

\pagebreak

\footnotesize

\bibliographystyle{amsalpha}
\bibliography{References}

\normalsize

\vspace{0.5cm}

\noindent{\textsc{Department of Pure Mathematics and Mathematical Statistics, University of Cambridge, UK}}

\noindent{\emph{E-mail address:} \texttt{ff373@cam.ac.uk}}

\end{document}